\documentclass{amsart}

\usepackage[margin=1.35in]{geometry}
\usepackage{amsmath,amssymb}
\usepackage[colorlinks=true,linkcolor=blue,citecolor=blue,urlcolor=blue]{hyperref}
\usepackage{graphicx}
\usepackage{subcaption}
\usepackage{float}
\usepackage[english]{babel}

\newtheorem{theorem}{Theorem}[section]
\newtheorem{prop}[theorem]{Proposition}
\newtheorem{lem}[theorem]{Lemma}
\newtheorem{cor}[theorem]{Corollary}

\theoremstyle{definition}
\newtheorem{defn}[theorem]{Definition}
\newtheorem{ass}[theorem]{Assumption}
\newtheorem{statement}{Statement}
\newtheorem{algorithm}{Algorithm}

\theoremstyle{remark}

\newtheorem{remark}[theorem]{Remark}

\numberwithin{equation}{section}

\newcommand{\be}{\begin{equation}}
\newcommand{\ee}{\end{equation}}

\newcommand{\lb}{\left(}
\newcommand{\rb}{\right)}
\newcommand{\lsb}{\left[}
\newcommand{\rsb}{\right]}
\newcommand{\lcb}{\left\{}
\newcommand{\rcb}{\right\}}

\newcommand{\ip}[1]{\langle#1\rangle}
\newcommand{\norm}[1]{\left\lVert#1\right\rVert}

\newcommand{\wh}{\widehat}

\newcommand{\ve}{\varepsilon}

\renewcommand{\P}{\mathbb{P}}
\newcommand{\E}{\mathbb{E}}

\newcommand{\X}{X}
\newcommand{\Y}{Y}
\newcommand{\Z}{Z}
\renewcommand{\L}{L}
\newcommand{\bm}{{W}}

\newcommand{\x}{f}
\newcommand{\y}{g}
\newcommand{\z}{h}
\newcommand{\push}{\ell}

\newcommand{\phib}{\mathcal{J}}
\newcommand{\etab}{\mathcal{K}}
\newcommand{\psib}{\mathcal{H}}

\newcommand{\F}{\mathcal{F}}

\newcommand{\N}{\mathbb{N}}
\newcommand{\R}{\mathbb{R}}
\newcommand{\U}{\mathcal{N}}

\newcommand{\ellindex}{l}

\newcommand{\Id}{I}

\newcommand{\allN}{\mathcal{I}}
\newcommand{\lip}{\kappa}

\newcommand{\spaan}{\text{span}}
\newcommand{\conv}{{\text{cone}}}
\newcommand{\osc}{\text{Osc}}

\newcommand{\cts}{\mathbb{C}}

\newcommand{\dr}{\mathbb{D}}
\newcommand{\simple}{\mathbb{D}_\text{s}}

\newcommand{\sphere}{\mathbb{S}^{J-1}}

\newcommand{\sm}{\Gamma}
\newcommand{\esm}{\bar{\Gamma}}
\newcommand{\dm}{\Lambda}

\newcommand{\proj}{\mathcal{L}}
\newcommand{\hyper}{\mathbb{H}}

\begin{document}

\title[Sensitivity analysis of reflected diffusions]{A Monte Carlo method for estimating sensitivities of reflected diffusions in convex polyhedral domains}

\date{\today}

\author[David Lipshutz]{David Lipshutz*}
\address{Department of Electrical Engineering \\ 
                Technion --- Israel Institute of Technology \\ 
                Haifa, Israel}
\email{lipshutz@ee.technion.ac.il}

\author[Kavita Ramanan]{Kavita Ramanan\textsuperscript{$\dagger$}}
\address{Division of Applied Mathematics \\ 
	Brown University \\ 
	Providence, USA}
\email{kavita{\_}ramanan@brown.edu}



\thanks{*The first author was supported in part by NSF grants DMS-1148284 and CMMI-1407504.}
\thanks{\textsuperscript{$\dagger$}The second author was supported in part by NSF grant CMMI-1407504.}

\dedicatory{Technion --- Israel Institute of Technology and Brown University}

\begin{abstract}
In this work we develop an effective Monte Carlo method for estimating sensitivities, or gradients of expectations of sufficiently smooth functionals, of a reflected diffusion in a convex polyhedral domain with respect to its defining  parameters --- namely, its initial condition, drift and diffusion coefficients, and directions of reflection. Our method, which falls into the class of infinitesimal perturbation analysis (IPA) methods, uses a probabilistic representation for such sensitivities as the expectation of a functional of the reflected diffusion and its associated derivative process. The latter process is the unique solution to a constrained linear stochastic differential equation with jumps whose coefficients, domain and directions of reflection are modulated by the reflected diffusion. We propose an asymptotically unbiased estimator for such sensitivities using an Euler approximation of the reflected diffusion and its associated derivative process. Proving that the Euler approximation converges is challenging because the derivative process jumps whenever the reflected diffusion hits the boundary (of the domain). A key step in the proof is establishing a continuity property of the related derivative map, which is of independent interest. We compare the performance of our IPA estimator to a standard likelihood ratio estimator (whenever the latter is applicable), and provide numerical evidence that the variance of the former is substantially smaller than that of the latter. We illustrate our method with an example of a rank-based interacting diffusion model of equity markets. Interestingly, we show that estimating certain sensitivities of the rank-based interacting diffusion model using our method for a reflected Brownian motion description of the model outperforms a finite difference method for a stochastic differential equation description of the model.
\end{abstract}

\maketitle

\tableofcontents

\section{Introduction}\label{sec:intro}

\subsection{Overview}

Reflected diffusions in convex polyhedral domains arise in numerous applications. For instance, they arise in the study of rank-based interacting diffusion models in mathematical finance \cite{Banner2005,Ichiba2011} and as diffusion approximations in queueing theory \cite{Chen1991,Mandelbaum1998,Peterson1991,Ramanan2003,Ramanan2008a,Reiman1984}. For use in uncertainty qualification, stochastic optimization and other areas (see \cite[Chapter VII]{Asmussen2007} for a list of the numerous applications), it is of interest to estimate the gradient of the expectation of a functional of a reflected diffusion with respect to its defining parameters --- namely, its initial condition, drift and diffusion coefficients, and directions of reflection along the boundary of its domain. (Henceforth, we simply write ``sensitivities'' to mean ``gradients of expectations of functionals''.) The main contribution of this work is to develop a broadly applicable asymptotically unbiased estimator of a large class of sensitivities of reflected diffusions, which can be used to approximate sensitivities of reflected diffusions via a Monte Carlo method.

Our estimator is based on a probabilistic representation for sensitivities of reflected diffusions that was obtained in \cite[Corollary 3.15]{Lipshutz2017}. This representation expresses the sensitivity as the expectation of a functional of a reflected diffusion and its associated derivative process. The derivative process satisfies a constrained linear stochastic differential equations with jumps whose coefficients, domain and directions of reflection are modulated by the reflected diffusion, and it was shown in \cite[Theorem 3.13]{Lipshutz2017} that the pathwise derivatives of a reflected diffusion can be described via the derivative process. While this representation provides an unbiased estimator for sensitivities of a reflected diffusion, computation of sensitivities via this representation would entail simulation of the reflected diffusion and its associated derivative process, which typically involves discrete-time approximations. There is a large literature devoted to approximating reflected diffusions in convex polyhedral domains (see, e.g., \cite{Blanchet2014,Gobet2004,Gobet2000,Liu1993,Pettersson1995,Pettersson1997,Slominski1994,Slominski2001}); however, there is no method for approximating the derivative process. In this work we propose an Euler scheme for the approximation of a reflected diffusion and its derivative process, and prove that the associated estimators of sensitivities of the reflected diffusion are asymptotically unbiased, as the discretization parameter goes to zero. The proof that the Euler scheme for the reflected diffusion is asymptotically unbiased follows an argument analogous to the one used in the proof of \cite[Theorem 3.2]{Slominski2001}, which established the result in the case of normal reflection. The proof that the Euler scheme for the derivative process is asymptotically unbiased is much more challenging because the derivative process jumps whenever the reflected diffusion hits the boundary of the polyhedral domain. Establishing convergence of the approximation is quite subtle and relies on a continuity property of a certain map, called the derivative map. This continuity property is of independent interest; for example, it is used in \cite{Lipshutz2017b} to prove that a reflected Brownian motion (RBM) in a convex polyhedral cone and its derivatives process are jointly Feller continuous.

Our method, which relies on pathwise derivatives of a reflected diffusion, falls into the class of infinitesimal perturbation analysis (IPA) methods used in sensitivity analysis (see, e.g., \cite[Chapter VII.2]{Asmussen2007}). We compare the performance of our IPA method to a standard likelihood ratio (LR) method and provide numerical evidence that the variance of the former is substantially smaller than that of the latter. It is also worth mentioning here that the LR method only applies to perturbations of the drift and in many applications it is of interest to study perturbations with respect to all of the parameters (e.g., when estimating sensitivities of diffusion approximations of queueing networks, which we plan to investigate in future work). We also apply our method to study a particular rank-based interacting diffusion model called the {\em Atlas model}, originally proposed by Fernholz \cite[Example 5.3.3]{Fernholz2002} to model equity markets, and subsequently generalized by Banner, Fernholz and Karatzas \cite{Banner2005} and Ichiba et.\ al.\ \cite{Ichiba2011}. This model is described by a stochastic differential equation (SDE) with discontinuous drift coefficients, and its sensitivities can be estimated using a standard finite difference (FD) method (although this method remains biased as the time-discretization vanishes). On the other hand, the dynamics of this model can also be expressed in terms of an RBM in a convex polyhedral domain (see, e.g., \cite{Fernholz2002} and Section \ref{subs-Atlas} below). We estimate sensitivities by applying our method to this {RBM} representation, and provide numerical evidence to show that it performs much better than the FD method applied to the SDE description of the model. 

In summary, the main contributions of this work are as follows:
\begin{itemize}
	\item An IPA method for estimating sensitivities of a reflected diffusion (Algorithm \ref{alg:ipa}).
	\item Proof of convergence of Euler schemes for a reflected diffusion (Theorem \ref{thm:EulerRD}) and its associated derivative process (Theorem \ref{thm:approx}).
	\item Comparison of our method with the LR method, with numerical evidence that our method performs better, especially {over long time horizons} (Section \ref{sec:comparison}).
	\item Application of our method to estimate certain sensitivities of the Atlas model, and numerical evidence that our method applied to the RBM representation of the model performs better than the FD method applied to an SDE description of the model (Section \ref{subs-Atlas}).
	\item A continuity property of the derivative map (Theorem \ref{thm:dmcontinuous}).
\end{itemize}

\subsection{Outline}

The remainder of this paper is organized as follows. In Section \ref{sec:setup} we give precise definitions of a reflected diffusion and its associated derivative process, and we recall the probabilistic representation of sensitivities of reflected diffusions from \cite{Lipshutz2017}.  In Section \ref{sec:main} we define an Euler approximation for a reflected diffusion and its derivative process, state our main convergence result and describe a numerical algorithm for estimating sensitivities.  In Section \ref{sec:comparison} we compare our algorithm with an LR algorithm for gradient estimation in cases when the latter applies. In Section \ref{sec:examples} we present numerical results from applying our method to a one-dimensional RBM and to the Atlas model. In Section \ref{sec:spdp}, we define and state properties of the Skorokhod problem (SP) and the derivative problem, which are used to characterize a reflected diffusion and its derivative process, respectively. These are used in the proofs of our main results, which are presented in Sections \ref{sec:EulerRDproof}--\ref{sec:dmcontinuous}.

\subsection{Notation}

Let $\N\doteq\{1,2,\dots\}$ denote the set of positive integers, $\N_0\doteq\N\cup\{0\}$ and $\N_\infty\doteq\N\cup\{\infty\}$. Given $J\in\N$, we use $\R_+^J$ to denote the closed nonnegative orthant in $J$-dimensional Euclidean space $\R^J$. When $J=1$, we suppress $J$ and write $\R$ for $(-\infty,\infty)$ and $\R_+$ for $[0,\infty)$. {For $r,s\in\R$ let $\lfloor r\rfloor=\max\{k\in\N_0:k\leq r\}$, $r\wedge s=\min(r,s)$ and $r\vee s=\max(r,s)$.} For a subset $A\subset\R$, let $\inf A$ and $\sup A$ denote the infimum and supremum, respectively, of $A$, with the convention that the infimum and supremum of the empty set are respectively defined to be $\infty$ and $-\infty$. For a column vector $x\in\R^J$, let $x^j$ denote the $j$th component of $x$, for $j=1,\dots,J$, and let $|x|$ denote the usual Euclidean norm of $x$. We let $\{e_1,\dots,e_J\}$ denote the standard normal basis in $\R^J$, where $e_j$ is the column vector in $\R^J$ whose $j$th component is one and whose other components are zero, for $j=1,\dots,J$. {We let $\sphere=\{x\in\R^J:|x|=1\}$ denote the unit sphere in $\R^J$}. For $J,K\in\N$, let $\R^{J\times K}$ denote the set of real-valued matrices with $J$ rows and $K$ columns. We write $M^T\in\R^{K\times J}$ for the transpose of a matrix $M\in\R^{J\times K}$. Let $\Id_J$ denote the $J\times J$ identity matrix. Let $\norm{\cdot}$ denote the operator norm on $\R^{J\times K}$.

Given $E\subseteq\R^J$ we let $\conv(E)$ denote the convex cone generated by $E$; that is,
	$$\conv(E)\doteq\lcb\sum_{k=1}^Kr_kx_k,K\in\N,x_k\in E,r_k\in\R_+\rcb,$$
and let $\spaan(E)$ denote the set of all possible finite linear combinations of vectors in $E$, with the convention that $\conv(\emptyset)$ and $\spaan(\emptyset)$ are equal to $\{0\}$. Given $E\subseteq\R^J$ or $E\subseteq\R^{J\times K}$ let $\dr(E)$ denote the set of functions mapping $[0,\infty)$ to $E$ that are right-continuous with finite left limits (RCLL). Let $\cts(E)$ denote the subset of continuous functions in $\dr(E)$. For a subset $A\subseteq E$, let $\dr_A(E)\doteq\{f\in\dr(E):f(0)\in A\}$ and $\cts_A(E)\doteq\{f\in\cts(E):f(0)\in A\}$. We endow $\dr(E)$ and its subsets with the $J_1$-topology and recall that the $J_1$-topology relativized to $\cts(E)$ coincides with the topology of uniform convergence on compact subsets of $[0,\infty)$. Given $f\in\dr(E)$ and $t>0$, we let $f(t-)$ denote the left limit of $f(\cdot)$ at $t$. {Given a function $f:(0,\infty)\mapsto\R$, we say that $f(\ve)=o(\ve)$ as $\ve\downarrow0$ if $|f(\ve)|/\ve\to0$ as $\ve\downarrow0$, and we say that $f(\ve)=O(\ve)$ as $\ve\downarrow0$ if $|f(\ve)|/\ve$ is bounded as $\ve\downarrow0$.}

Throughout this paper we fix a filtered probability space $(\Omega,\F,\{\F_t\},\P)$ satisfying the usual conditions; that is, $(\Omega,\F,\P)$ is a complete probability space, $\F_0$ contains all $\P$-null sets in $\F$ and the filtration $\{\F_t\}$ is right-continuous. We write $\E$ to denote expectation under $\P$. We abbreviate ``almost surely'' as ``a.s.'' By a $K$-dimensional $\{\F_t\}$-Brownian motion $\bm$ on $(\Omega,\F,\P)$, we mean that $(\bm^1,\dots,\bm^K)$ are independent and  for each $k=1,\dots,K$, $\{\bm^k(t),\F_t, t \geq 0\}$ is a continuous martingale with $\bm^k(0) = 0$ and quadratic variation $[\bm^k]_t=t$ for $t\geq0$.  We let $C_p<\infty$, for $p\geq 2$, denote the universal constants in the Burkholder-Davis-Gundy (BDG) inequalities (see, e.g., \cite[Chapter IV, Theorem 42.1]{Rogers2000a}). 

\section{Background on reflected diffusions and their sensitivities}\label{sec:setup}

\subsection{A parameterized family of reflected diffusions}\label{sec:reflecteddiffusions}

Let $G$ be a closed polyhedron in $\R^J$ equal to the intersection of finitely many closed half spaces in $\R^J$; that is,
\be\label{eq:G}G\doteq\bigcap_{i=1,\dots,N}\lcb x\in\R^J:\ip{x,n_i}\geq c_i\rcb,\ee
for a positive integer $N\in\N$, unit vectors $n_i\in \sphere$ and constants $c_i\in\R$, for $i=1,\dots,N$. Let $\allN=\{1,\dots,N\}$. For  $i\in\allN$, we let $F_i\doteq\{x\in\partial G:\ip{x,n_i}=c_i\}$ denote the $i$th face of $G$. For $x\in G$, we write $\label{eq:allNx}\allN(x)\doteq\{i\in\allN:x\in F_i\}$ to denote the (possibly empty) set of indices associated with the faces that intersect at $x$. 

Let $M\in\N$ and $U$ be an open \emph{parameter set} in $\R^M$. For each $i\in\allN$, fix a continuously differentiable function $d_i:U\mapsto\R^J$ that satisfies $\ip{d_i(\alpha),n_i}=1$ for all $\alpha\in U$. For $\alpha\in U$, $d_i(\alpha)$ will denote the (constant) direction of reflection along the face $F_i$ associated with the parameter $\alpha$. Since the direction of reflection $d_i(\alpha)$ can be renormalized, our assumption $\ip{d_i(\alpha),n_i}=1$ for all $\alpha\in U$ is without loss of generality. We let 
	$$R:U\mapsto\R^{J\times N}$$
denote the continuous differentiable function defined by $R(\alpha)\doteq\begin{pmatrix}d_1(\alpha)&\cdots&d_N(\alpha)\end{pmatrix}$ for $\alpha\in U$, and let $R'(\alpha)$ denote the Jacobian of $R(\cdot)$ at $\alpha\in U$. We refer to $R(\alpha)$ as the reflection matrix associated with $\alpha$. Fix continuously differentiable functions 
\begin{align*}
	x_0:U\mapsto G,\qquad b:U\times G\mapsto\R^J,\qquad \sigma:U\times G\mapsto\R^{J\times K}.
\end{align*}
For $\alpha\in U$, $x_0(\alpha)$, $b(\alpha,\cdot)$ and $\sigma(\alpha,\cdot)$ will respectively be the initial condition, drift and dispersion coefficients for the reflected diffusion associated with $\alpha$. We refer to $a(\alpha,\cdot)\doteq\sigma(\alpha,\cdot)\sigma^T(\alpha,\cdot)$ as the diffusion coefficient associated with $\alpha\in U$. For $\alpha\in U$ and $x\in G$, we let $x_0'(\alpha)$ denote the Jacobian of $x_0(\cdot)$ at $\alpha\in U$, $b_\alpha(\alpha,x)$ denote the Jacobian of $b(\cdot,x)$ at $\alpha\in U$, $b_x(\alpha,x)$ denote the Jacobian of $b(\alpha,\cdot)$ at $x\in G$, and similarly define $\sigma_\alpha(\alpha,x)$ and $\sigma_x(\alpha,x)$.

\begin{defn}\label{def:rd}
Given $\alpha\in U$ and a $K$-dimensional $\{\F_t\}$-Brownian motion $\bm$ on $(\Omega,\F,\P)$, a reflected diffusion associated with $\alpha$ and driving Brownian motion $\bm$ is a $J$-dimensional continuous $\{\F_t\}$-adapted process $\Z^\alpha=\{\Z^\alpha(t),t\geq0\}$ such that a.s.\ for all $t\geq0$, $\Z^\alpha(t)\in G$ and $\Z^\alpha$ satisfies
	\be\label{eq:Z}\Z^\alpha(t)=x_0(\alpha)+\int_0^tb(\alpha,\Z^\alpha(s))ds+\int_0^t\sigma(\alpha,\Z^\alpha(s))d\bm(s)+R(\alpha)\L^\alpha(t),\qquad t\geq0,\ee
where $\L^\alpha=\{\L^\alpha(t),t\geq0\}$ is an $N$-dimensional continuous $\{\F_t\}$-adapted process such that a.s.\ $\L^\alpha(0)=0$ and for every $i\in\allN$, the $i$th component of $\L^\alpha$, denoted $\L^{\alpha,i}$, is nondecreasing and can only increase when $\Z^\alpha$ lies in face $F_i$; that is,
	$$\int_0^\infty1_{\{\Z^\alpha(s)\not\in F_i\}}d\L^{\alpha,i}(s)=0.$$  
We say that \emph{pathwise uniqueness} holds if given $\alpha\in U$ and a $K$-dimensional $\{\F_t\}$-Brownian motion $\bm$ on $(\Omega,\F,\P)$, any two reflected diffusions associated with $\alpha$ and driving Brownian motion $\bm$ are indistinguishable.
\end{defn}

\begin{remark}\label{rmk:initial}
In \cite[Definition 2.1]{Lipshutz2017} the authors define a family of reflected diffusions in which the drift and dispersion coefficients and directions of reflection are parameterized by $\alpha\in U$, but  the initial condition is parameterized by $x\in G$. In \cite{Lipshutz2017}, this allowed for a characterization of pathwise derivatives of flows of reflected diffusions and was a convenient representation in the proofs. In contrast, here we will find it more convenient to assume that the initial condition is a continuously differentiable function $x_0(\cdot)$ on $U$ taking values in $G$.
\end{remark}

\begin{remark}\label{rmk:Y}
Given $\alpha\in U$ and a reflected diffusion $\Z^\alpha$ with $\L^\alpha$ satisfying the conditions in Definition \ref{def:rd}, define the $J$-dimensional $\{\F_t\}$-adapted constraining process $\Y^\alpha=\{\Y^\alpha(t),t\geq0\}$ by $\Y^\alpha\doteq R(\alpha)\L^\alpha$. It follows from the definition of $R(\alpha)$ and the conditions on $\L^\alpha$ in Definition \ref{def:rd} that $\Y^\alpha$ satisfies, for all $0\leq s<t<\infty$,
	\be\label{eq:Y}\Y^\alpha(t)-\Y^\alpha(s)\in\conv\lsb\cup_{u\in(s,t]}d(\alpha,\Z^\alpha(u))\rsb,\ee
where
	\be\label{eq:dalphax}d(\alpha,x)\doteq\conv(\{d_i(\alpha),i\in\allN(x)\}),\qquad \alpha\in U,\; x\in G.\ee 
In particular, we see that $\Z^\alpha$ satisfies \cite[Definition 2.1]{Lipshutz2017} for a reflected diffusion (with initial condition $x=x_0(\alpha)$ there). The more general condition \eqref{eq:Y} allows for reflected diffusions that are not semimartingles. In this work we impose a mild linear independence condition on the directions of reflection (see Assumption \ref{ass:independent} below) under which $\Z^\alpha$ satisfies Definition \ref{def:rd} if and only if $\Z^\alpha$ satisfies \cite[Definition 2.1]{Lipshutz2017}.
\end{remark}

Let $\zeta_1,\zeta_2$ lie in $C_b^1(G)$, the space of real-valued functions on $G$ that are continuously differentiable with bounded first partial derivatives. Suppose that for each $\alpha\in U$ there exists a unique reflected diffusion $\Z^\alpha$ associated with $\alpha$. Then for $t\geq0$, define $\Theta_t(\cdot)$ to be the mapping from $U$ to $\R$ defined by
	\be\label{eq:Falpha}\Theta_t(\alpha)\doteq\E\lsb\int_0^t\zeta_1(\Z^\alpha(s))ds+\zeta_2(\Z^\alpha(t))\rsb,\qquad\alpha\in U.\ee
In the next two sections we state our main assumptions, introduce the derivative process along $\Z^\alpha$ and characterize the Jacobian of $\Theta_t(\cdot)$ in terms of $\Z^\alpha$ and the derivative process along $\Z^\alpha$.

\subsection{Main assumptions}\label{sec:assumptions}

\emph{The assumptions stated in this section are assumed to hold, without restatement, throughout this work.}

The first three assumptions on the data $\{(d_i(\cdot),n_i,c_i),i\in\allN\}$ ensure the associated SP is well defined and the associated Skorokhod map (SM) is Lipschitz continuous (see Proposition \ref{prop:sp} below), which is useful for proving strong existence of reflected diffusions and establishing pathwise uniqueness.

\begin{ass}\label{ass:independent}
	For each $\alpha\in U$ and $x\in\partial G$, $\{d_i(\alpha),i\in\allN(x)\}$ is a set of linearly independent vectors.
\end{ass}

Given a convex set $B$, let $\nu_B(z)\doteq\{\nu\in\sphere:\ip{\nu,y-z}\geq0\;\forall\;y\in B\}$ denote the set of inward normal vectors to the set $B$ at $z\in\partial B$.

\begin{ass}\label{ass:setB}
For each $\alpha\in U$ there exists $\delta(\alpha)>0$ and a compact, convex, symmetric set $B^\alpha$ in $\R^J$ with $0\in(B^\alpha)^\circ$ such that for $i\in\allN$,
	\be\label{eq:setB}\lcb\begin{array}{l}z\in\partial B^\alpha\\|\ip{z,n_i}|<\delta(\alpha)\end{array}\rcb\qquad\Rightarrow\qquad\ip{\nu,d_i(\alpha)}=0\qquad\text{for all }\;\nu\in\nu_{B^\alpha}(z).\ee
\end{ass}

Recall the definition of $d(\cdot,\cdot)$ given in \eqref{eq:dalphax}.

\begin{ass}\label{ass:projection}
For each $\alpha\in U$ there is a map $\pi^\alpha:\R^J\mapsto G$ satisfying $\pi^\alpha(x)=x$ for all $x\in G$ and $\pi^\alpha(x)-x\in d(\alpha,\pi^\alpha(x))$ for all $x\not\in G$.
\end{ass}

\begin{remark}\label{rmk:pix}
Under Assumption \ref{ass:setB}, there can be at most one map $\pi^\alpha:\R^J\mapsto G$ that satisfies the conditions stated in Assumption \ref{ass:projection} (see, e.g., the paragraph before \cite[Lemma 4.3]{Dupuis1999a} on page 184).
\end{remark}

\begin{remark}
See \cite[Example 2.14]{Lipshutz2016} for a general set of algebraic conditions on $\{(d_i(\cdot),n_i,c_i),i\in\allN\}$ that imply Assumptions \ref{ass:independent}, \ref{ass:setB} and \ref{ass:projection} hold.
\end{remark}

Along with the last three assumptions, the final two assumptions on the coefficients $b(\cdot,\cdot)$, $\sigma(\cdot,\cdot)$ and $R(\cdot)$ ensure existence and pathwise uniqueness of the reflected diffusion and the derivative process, as well as the characterization of sensitivities of reflected diffusions in terms of the derivative process (see Theorem \ref{thm:DF} below).

\begin{ass}\label{ass:coefficients}
There exists $\lip_1<\infty$ such that for all $\alpha\in U$ and $x\in G$,
	$$\max\lcb\norm{b_\alpha(\alpha,x)},\norm{b_x(\alpha,x)},\norm{\sigma_\alpha(\alpha,x)},\norm{\sigma_x(\alpha,x)},\norm{R'(\alpha)}\rcb\leq\lip_1.$$
In addition, there exist $\lip_2<\infty$ and $\gamma\in(0,1]$ such that for all $\alpha,\beta\in U$ and $x,y\in G$,
\begin{align*}
	\max\lcb
	\begin{array}{l l}
	\norm{b_\alpha(\alpha,x)-b_\alpha(\beta,y)},&\norm{b_x(\alpha,x)-b_x(\beta,y)},\\
	\norm{\sigma_\alpha(\alpha,x)-\sigma_\alpha(\beta,y)},&\norm{\sigma_x(\alpha,x)-\sigma_x(\beta,y)},\\
	\norm{R'(\alpha)-R'(\beta)}
	\end{array}
	\rcb
	\leq\lip_2|(\alpha,x)-(\beta,y)|^\gamma.
\end{align*}
\end{ass}

\begin{ass}\label{ass:elliptic}
For each $\alpha\in U$ there exists $c(\alpha)>0$ such that for all $x\in G$,
	\be y^Ta(\alpha,x)y\geq c(\alpha)|y|^2,\qquad y\in\R^J.\ee
\end{ass}

\subsection{Sensitivities of reflected diffusions in terms of the derivative process}\label{sec:derivativeprocess}

We first define the derivative process, which was introduced in \cite[Definition 3.5]{Lipshutz2017} to characterize pathwise derivatives of a reflected diffusion. Given $x\in G$, define
	\be\label{eq:Hx}\hyper_x\doteq\bigcap_{i\in\allN(x)}\lcb y\in\R^J:\ip{y,n_i}=0\rcb.\ee
Recall the definition of $d(\cdot,\cdot)$ in \eqref{eq:dalphax}.

\begin{defn}\label{def:de}
Let $\alpha\in U$, $\bm$ be a $K$-dimensional $\{\F_t\}$-Brownian motion on $(\Omega,\F,\P)$ and $\Z^\alpha$ be a reflected diffusion associated with $\alpha$ and driving Brownian motion $\bm$. A derivative process along $\Z^\alpha$ is an RCLL $\{\F_t\}$-adapted process $\phib^\alpha=\{\phib^\alpha(t),t\geq0\}$ taking values in $\R^{J\times M}$ such that a.s.\ for all $t\geq0$, $\phib_m^\alpha(t)\in \hyper_{\Z^\alpha(t)}$ for $m=1,\dots,M$ and $\phib^\alpha$ satisfies
\begin{align}\label{eq:phib}
	\phib^\alpha(t)&=x_0'(\alpha)+\int_0^tb_\alpha(\alpha,\Z^\alpha(s))ds+\int_0^tb_x(\alpha,\Z^\alpha(s))\phib^\alpha(s) ds\\ \notag
	&\qquad+\int_0^t\sigma_\alpha(\alpha,\Z^\alpha(s))d\bm(s)+\int_0^t\sigma_x(\alpha,\Z^\alpha(s))\phib^\alpha(s) d\bm(s)\\ \notag
	&\qquad+R'(\alpha)\L^\alpha(t)+\etab^\alpha(t),
\end{align}
where $\etab^\alpha=\{\etab^\alpha(t),t\geq0\}$ is an RCLL $\{\F_t\}$-adapted process taking values in $\R^{J\times M}$ such that a.s.\ $\etab^\alpha(0)=0$ and for each $m=1,\dots,M$ and all $0\leq s<t<\infty$,
\be\label{eq:etab}\etab_m^\alpha(t)-\etab_m^\alpha(s)\in\spaan\lsb\cup_{u\in(s,t]}d(\alpha,\Z^\alpha(u))\rsb.\ee
We say that \emph{pathwise uniqueness} holds if for each $\alpha\in U$, $K$-dimensional $\{\F_t\}$-Brownian motion $\bm$ and reflected diffusion $\Z^\alpha$ associated with $\alpha\in U$ with driving Brownian motion $\bm$, any two derivative processes along $\Z^\alpha$ are indistinguishable.
\end{defn}

The equation \ref{eq:phib} for the derivative process can be viewed as a  linearized version of the 
equation \eqref{eq:Z} for the reflected diffusion $Z^\alpha$, with the key feature, established in \cite{Lipshutz2017}, that $R'(\alpha)\L^\alpha + \etab^\alpha$ serves as the appropriate linearization of the constraining process $R(\alpha) L^\alpha$.

\begin{remark}\label{rmk:de}
Definition \ref{def:de} for the derivative process is slightly different from the definition given in \cite[Definition 3.5]{Lipshutz2017} due to the fact that the initial condition here is a function of $\alpha\in U$. To clarify the relation, suppose $\phib^\alpha$ satisfies Definition \ref{def:de} and $\phib^{\alpha,x}$ satisfies \cite[Definition 3.5]{Lipshutz2017}. Then, for $m=1,\dots,M$, the $m$th column vector of $\phib^\alpha$ satisfies $\phib_m^\alpha(t)=\phib_t^{\alpha,x_0(\alpha)}[e_m,x_0'(\alpha)e_m]$ for all $t\geq0$ (where the right-hand side of the equality is written in the notation of \cite{Lipshutz2017}).
\end{remark}

In the following theorem we state the probabilistic representation of sensitivities of reflected diffusions that was obtained in \cite{Lipshutz2017} and serves as the starting point for the method for sensitivity estimation that we introduce in this work. Recall that the assumptions stated in Section \ref{sec:assumptions} hold.

\begin{theorem}
\label{thm:DF}
For each $\alpha\in U$ and $K$-dimensional $\{\F_t\}$-Brownian motion $W$ on $(\Omega,\F,\P)$, the following hold:
\begin{itemize}
	\item[(i)] There exists a pathwise unique reflected diffusion $\Z^\alpha$ associated with $\alpha$ and driving Brownian motion $W$, and it is a strong Markov process.
	\item[(ii)] There exists a pathwise unique derivative process $\phib^\alpha$ along $\Z^\alpha$.
	\item[(iii)] Given $t\geq0$, a.s.\ $\phib^\alpha(\cdot)$ is continuous at $t$, and a.s.\ $\phib^\alpha(\cdot)$ is continuous at almost every $s\in[0,t)$.
	\item[(iv)] Given $t\geq0$, the function $\Theta_t(\cdot)$ defined in \eqref{eq:Falpha} is differentiable at $\alpha$ and its Jacobian satisfies
	\begin{align}\label{eq:F'}
	\Theta_t'(\alpha)=\E\lsb\int_0^t\zeta_1'(\Z^\alpha(s))\phib^\alpha(s) ds+\zeta_2'(\Z^\alpha(t))\phib^\alpha(t)\rsb.
	\end{align}
\end{itemize}
\end{theorem}

\begin{proof}[Proof of Theorem \ref{thm:DF}]
Part (i) follows from \cite[Proposition 2.16]{Lipshutz2017} (see also, \cite[Theorem 4.3]{Ramanan2006}) and the equivalence between solutions of Definition \ref{def:rd} and \cite[Definition 2.1]{Lipshutz2017} stated in Remark \ref{rmk:Y}. Parts (ii) and (iii) follow from \cite[Corollary 3.15]{Lipshutz2017} and Remark \ref{rmk:de}. Part (iv) follows from \cite[Corollary 3.16]{Lipshutz2017} and the chain rule.
\end{proof}

While \eqref{eq:F'} provides an unbiased estimator for $\Theta'(\alpha)$, exact sampling of functionals of $(\Z^\alpha,\phib^\alpha)$ is complicated by the discontinuous dynamics when $\Z^\alpha$ reaches the boundary $\partial G$. As is often the case when simulating diffusion processes, we sample from a discrete-time Euler approximation of $(\Z^\alpha,\phib^\alpha)$. Our main result (see Corollary \ref{cor:approx} below) states that the Euler approximation can be used to construct an asymptotically unbiased estimator for $\Theta_t'(\alpha)$.

\section{Main results}\label{sec:main}

Recall that the assumptions stated in Section \ref{sec:assumptions} hold. Fix $\alpha\in U$ and a $K$-dimensional $\{\F_t\}$-Brownian motion $\bm$ on $(\Omega,\F,\P)$. Let $\Z^\alpha$ denote the pathwise unique reflected diffusion associated with $\alpha$ and driving Brownian motion $\bm$, and let $\phib^\alpha$ denote the pathwise unique derivative process along $\Z^\alpha$. Throughout this section, given a time step $\Delta>0$, define the sequence $\{t_n^\Delta\}_{n\in\N_0}$ by 
\begin{equation}
\label{eq:tnDelta}
t_n^\Delta\doteq n\Delta,\qquad n\in\N_0, 
\end{equation}
and let $\{\delta_n^\Delta\bm\}_{n\in\N}$ be the sequence of i.i.d.\ $K$-dimensional Gaussian random variables with mean zero and diagonal covariance matrix $\Delta I_K$ given by
	\be\label{eq:deltanDelta}\delta_n^\Delta\bm\doteq\bm(t_n^\Delta)-\bm(t_{n-1}^\Delta),\qquad n\in\N.\ee
In addition, let $\{\F_t^\Delta\}$ denote the discrete filtration defined by 
	\be\label{eq:FtDelta}\F_t^\Delta\doteq\sigma\lb\lcb\delta_k^\Delta\bm,k=1,\dots,n-1\rcb\rb,\qquad t\in[t_{n-1}^\Delta,t_n^\Delta),\qquad n\in\N.\ee

\subsection{Euler scheme for the reflected diffusion}\label{sec:EulerRD}

In this section we present an Euler scheme for approximating just the reflected diffusion $\Z^\alpha$. This is an extension of a result obtained in \cite{Slominski2001} to allow for reflected diffusions with oblique reflection. We also prove a convergence result for an Euler approximation of the process $\L^\alpha$ that will be needed in the next section. Let $\pi^\alpha$ denote the unique mapping satisfying the conditions in Assumption \ref{ass:projection}. In order to define the Euler scheme, we need the following lemma.

\begin{lem}\label{lem:projdecomp}
There exists a unique map $\xi^\alpha:\R^J\mapsto\R_+^N$ such that for each $x\in\R^J$,  $\pi^\alpha(x)-x=R(\alpha)\xi^\alpha(x)$ and the $i$th component of $\xi^\alpha(x)$ satisfies $\xi^{\alpha,i}(x)>0$ only if $\pi^\alpha(x)\in F_i$, for $i\in\allN$.
\end{lem}

\begin{remark}
In the case of a simple polyhedral cone with vertex at the origin (i.e., $N=J$ and $c_i=0$ for $i\in\allN$), Assumption \ref{ass:coefficients} implies $R(\alpha)$ is invertible for each $\alpha\in U$ and so $\xi^\alpha(x)=(R(\alpha))^{-1}(\pi^\alpha(x)-x)$ for each $\alpha\in U$ and $x\in\R^J$.
\end{remark}

\begin{proof}
According to Assumption \ref{ass:projection} and Remark \ref{rmk:pix}, $\pi^\alpha:\R^J\mapsto G$ is the unique mapping that satisfies $\pi^\alpha(x)-x\in d(\alpha,\pi^\alpha(x))$. By the definition of $d(\cdot,\cdot)$ in \eqref{eq:dalphax} and the linear independence of the vectors $\{d_i(\alpha),i\in\allN(\pi^\alpha(x))\}$ guaranteed by Assumption \ref{ass:independent}, it follows that for each $x\in\R^J$ there is a unique vector $\xi^\alpha(x)\in\R_+^N$ satisfying the conditions of the lemma.
\end{proof}

For $\Delta>0$ let $(\Z^{\alpha,\Delta},\L^{\alpha,\Delta})$ denote the pair of piecewise constant RCLL $\{\F_t^\Delta\}$-adapted processes taking values in $G\times\R_+^N$ defined as follows: Set 
	\be\label{eq:ZLDelta0}(\Z^{\alpha,\Delta}(0),\L^{\alpha,\Delta}(0))\doteq (x_0(\alpha),0)\ee
and, for $n\in\N$, set 	
	\be\label{eq:ZLDeltat}(\Z^{\alpha,\Delta}(t),\L^{\alpha,\Delta}(t))\doteq(\Z^{\alpha,\Delta}(t_{n-1}^\Delta),\L^{\alpha,\Delta}(t_{n-1}^\Delta)),\qquad t\in[t_{n-1}^\Delta,t_n^\Delta),\ee
and recursively define
\begin{align}\label{eq:Znalphah}
	(\Z^{\alpha,\Delta}(t_n^\Delta),\L^{\alpha,\Delta}(t_n^\Delta)) \doteq (\pi^\alpha(\Xi_n^{\alpha, \Delta}),\L^{\alpha,\Delta}(t_{n-1}^\Delta)+\xi^\alpha(\Xi_n^{\alpha,\Delta})),
\end{align}
where $\Xi_n^{\alpha,\Delta}$ is the random $J$-dimensional vector defined by
	\be\label{eq:XinalphaDelta}\Xi_n^{\alpha, \Delta} \doteq \Z^{\alpha,\Delta}(t_{n-1}^\Delta)+b(\alpha,\Z^{\alpha,\Delta}(t_{n-1}^\Delta))\Delta
+\sigma(\alpha,\Z^{\alpha,\Delta}(t_{n-1}^\Delta))\delta_n^\Delta\bm.\ee
We have the following result on the convergence of the Euler approximations.

\begin{theorem}\label{thm:EulerRD}
For each $p\geq2$, as $\Delta\downarrow0$,
\begin{align}\label{eq:Zalphaorder}\E\lsb\sup_{0\leq s\leq t}|\Z^{\alpha,\Delta}(s)-\Z^\alpha(s)|^p\rsb&=O\lb\lb\Delta\log\frac{1}{\Delta}\rb^\frac{p}{2}\rb,\\ \label{eq:Lalphaorder}
	\E\lsb\sup_{0\leq s\leq t}|\L^{\alpha,\Delta}(s)-\L^\alpha(s)|^p\rsb&=O\lb\lb\Delta\log\frac{1}{\Delta}\rb^\frac{p}{2}\rb.
\end{align}
\end{theorem}

The proof of Theorem \ref{thm:EulerRD} is given in Section \ref{sec:EulerRDproof}. The proof is a straightforward adaptation of the proof of \cite[Theorem 3.2(i)]{Slominski2001}, which proves \eqref{eq:Zalphaorder} in the case of normal reflection along the boundary. The main difference between the two results is that we allow for oblique reflection along the boundary and also prove convergence of the constraining process in \eqref{eq:Lalphaorder}.

\subsection{Euler scheme for the derivative process}\label{sec:EulerDP}

We now construct an Euler scheme for the derivative process.  We start with 
a lemma that  introduces a linear projection map that can be interpreted as a linearization of $\pi^\alpha$. Recall the definition of the linear subspace $\hyper_x$, for $x\in G$, given in \eqref{eq:Hx}.

\begin{lem}
\label{lem:projx}
For each $x\in G$, there is a unique map $\proj_x^\alpha:\R^J\mapsto \hyper_x$ that satisfies $\proj_x^\alpha(y)-y\in\ spaan[d(\alpha,x)]$ for all $y\in\R^J$. Furthermore, $\proj_x^\alpha$ is linear.
\end{lem}

\begin{proof}
By Assumption \ref{ass:independent} and \cite[Lemma 8.2]{Lipshutz2016}, the set $\mathcal{W}^\alpha\doteq\{x\in G:\spaan(\hyper_x\cup d(\alpha,x))\neq\R^J\}$ is empty. The lemma then follows from \cite[Lemma 8.3]{Lipshutz2016}.
\end{proof}

\begin{remark}\label{rmk:projx}
For each $x\in G$, since $\proj_x^\alpha$ is a linear map from $\R^J$ to $\hyper_x\subseteq\R^J$, we can write $\proj_x^\alpha$ as a $J\times J$ matrix whose column vectors lie in $\hyper_x$. Throughout this work we view $\proj_x^\alpha$ as this $J\times J$ matrix and refer to $\proj_x^\alpha$ as the derivative projection matrix.
\end{remark}

For $\Delta>0$ recall the sequence $\{t_n^\Delta\}_{n\in\N_0}$, the random vectors $\delta_n^\Delta W$, $n\in\N$, the discrete filtration $\{\F_t^\Delta\}$ and the pair of processes $(\Z^{\alpha,\Delta},\L^{\alpha,\Delta})$ defined in the last section. Let $\phib^{\alpha,\Delta}=\{\phib^{\alpha,\Delta}(t),t\geq0\}$ denote the piecewise constant RCLL $\{\F_t^\Delta\}$-adapted process taking values in $\R^{J\times M}$ defined as follows: Set 
	\be\label{eq:phibDelta0}\phib^{\alpha,\Delta}(0)\doteq x_0'(\alpha)\ee
and, for $n\in\N$, set
	\be\label{eq:phibDeltat}\phib^{\alpha,\Delta}(t)\doteq\phib^{\alpha,\Delta}(t_{n-1}^\Delta),\qquad t\in[t_{n-1}^\Delta,t_n^\Delta),\ee
and recursively define
\begin{align}\label{eq:psibalphah}
	\phib^{\alpha,\Delta}(t_n^\Delta)&\doteq\proj_{\Z^{\alpha,\Delta}(t_n^\Delta)}^\alpha\mathcal{X}_n^{\alpha,\Delta},
\end{align}
where $\mathcal{X}_n^{\alpha,\Delta}$ is the random element taking values in $\R^{J\times M}$ defined by
\begin{align}\label{eq:mathcalXnDeltaalpha}
	\mathcal{X}_n^{\alpha,\Delta}&\doteq\phib^{\alpha,\Delta}(t_{n-1}^\Delta)+b_\alpha(\alpha,\Z^{\alpha,\Delta}(t_{n-1}^\Delta))\Delta+b_x(\alpha,\Z^{\alpha,\Delta}(t_{n-1}^\Delta))\phib^{\alpha,\Delta}(t_{n-1}^\Delta)\Delta\\ \notag
	&\qquad+\sigma_\alpha(\alpha,\Z^{\alpha,\Delta}(t_{n-1}^\Delta))\delta_n^\Delta\bm+\sigma_x(\alpha,\Z^{\alpha,\Delta}(t_{n-1}^\Delta))\phib^{\alpha,\Delta}(t_{n-1}^\Delta)\delta_n^\Delta\bm\\ \notag
	&\qquad+R'(\alpha)(\L^{\alpha,\Delta}(t_n^\Delta)-\L^{\alpha,\Delta}(t_{n-1}^\Delta)).
\end{align}

The following establishes convergence of the Euler scheme for the derivative process. 

\begin{theorem}\label{thm:approx}
A.s.\ $(\Z^{\alpha,\Delta},\phib_1^{\alpha,\Delta},\dots,\phib_m^{\alpha,\Delta})$ converges to $(\Z^\alpha,\phib_1^\alpha,\dots,\phib_m^\alpha)$ in $\cts(G)\times\dr(\R^J)\times\cdots\times\dr(\R^J)$ as $\Delta\downarrow0$. In addition, given any $t<\infty$, the family $\{\phib^{\alpha,\Delta}(s)\}_{s\in[0,t],\Delta>0}$ is a uniformly integrable family of random variables.
\end{theorem}

The proof of Theorem \ref{thm:approx}, which  is given in Section \ref{sec:approx}, relies on continuity properties of the related derivative map, which are established in Theorem \ref{thm:dmcontinuous}.

\begin{cor}\label{cor:approx}
For all $t\geq0$,
	\be\label{eq:F'alphaapprox} \Theta_t'(\alpha)=\lim_{\Delta\downarrow0}\E\lsb\int_0^t\zeta_1'(\Z^{\alpha,\Delta}(s))\phib^{\alpha,\Delta}(s)ds+\zeta_2'(\Z^{\alpha,\Delta}(t))\phib^{\alpha,\Delta}(t)\rsb.\ee
\end{cor}

\begin{proof}
The corollary follows from Theorem \ref{thm:DF}(iv) and Theorem \ref{thm:approx}.
\end{proof}

Corollary \ref{cor:approx} suggests an asymptotically unbiased estimator for $\Theta_t'(\alpha)$. In the next section we describe an algorithm for estimating $\Theta_t'(\alpha)$ based on the Euler discretization. It is also of interest to investigate whether there is an exact sampling algorithm for the joint process $(\Z^\alpha,\phib^\alpha)$, which would avoid the bias introduced by the Euler discretization. Even in the setting of an RBM, where an exact sampling method has been proposed in \cite{Blanchet2014} for RBMs in the nonnegative orthant with reflection matrices that are so-called $\mathcal{M}$-matrices, it is a challenging problem to exactly sample the associated derivative process due to the fact that the derivative process jumps whenever the RBM reaches the boundary of the domain. We leave this as an interesting open problem for future work.

\subsection{Numerical algorithm}\label{sec:algorithm}

Fix $0<\Delta<t<\infty$. Since $\Delta$ is typically much smaller than $t$, for simplicity we can assume $N\doteq t/\Delta$ is a positive integer. Since $\Z^{\alpha,\Delta}$ and $\phib^{\alpha,\Delta}$ are constant on intervals of the form $[t_{n-1}^\Delta,t_n^\Delta)$, $n=1,\dots,N$, we have
\begin{align}\label{eq:whThetaIPA}
	&\int_0^t\zeta_1'(\Z^{\alpha,\Delta}(s))\phib^{\alpha,\Delta}(s)ds+\zeta_2'(\Z^{\alpha,\Delta}(t))\phib^{\alpha,\Delta}(t)\\ \notag
	&\qquad=\sum_{n=0}^{N-1}\zeta_1'(\Z^{\alpha,\Delta}(t_n^\Delta))\phib^{\alpha,\Delta}(t_n^\Delta)\Delta+\zeta_2'(\Z^{\alpha,\Delta}(t_N^\Delta))\phib^{\alpha,\Delta}(t_N^\Delta).
\end{align}

\begin{algorithm}
\label{alg:ipa}
Set $\Z^{\alpha,\Delta}(t_0^\Delta)\doteq x_0(\alpha)$ and $\phib^{\alpha,\Delta}(t_0^\Delta)\doteq x_0'(\alpha)$. For $n=1,\dots,N$, recursively complete the following steps:
\begin{itemize}
	\item[1.] Generate a $K$-dimensional Gaussian random variable $\delta_n^\Delta\bm$ with mean zero and diagonal covariance matrix $\Delta\Id_K$.
	\item[2.] Define $(\Z^{\alpha,\Delta}(t_n^\Delta),\L^{\alpha,\Delta}(t_n^\Delta))$ as in \eqref{eq:Znalphah}--\eqref{eq:XinalphaDelta}.
	\item[3.] Define $\phib^{\alpha,\Delta}(t_n^\Delta)$ as in \eqref{eq:psibalphah}--\eqref{eq:mathcalXnDeltaalpha}.
\end{itemize}
Set the (IPA) estimator $\wh{\Theta_t'(\alpha)}_{\text{IPA}}$ of $\Theta_t'(\alpha)$ equal to the right hand side of \eqref{eq:whThetaIPA}.
\end{algorithm}

In the next section we compare our algorithm with other methods for sensitivity analysis. In Section \ref{sec:examples} below, we illustrate our algorithm with an example of a one-dimensional RBM and an example of an RBM that arises in the study of a rank-based interacting diffusion model of equity markets.

\section{Comparison with existing methods}\label{sec:comparison}

The main existing (asymptotically unbiased) estimator for estimating sensitivities of reflected diffusions is the LR method. The LR method is applicable only when the law of the perturbed process is absolutely continuous with the law of the original process. In the context of multidimensional (reflected) diffusions, this only holds for perturbations to the drift. In this case, the LR method uses a change of measure argument to recast expectations of functionals of the perturbed process as the expectation, under the law of the original process, of the functional multiplied by the likelihood ratio or the Radon-Nikodym derivative. The LR method for sensitivities of diffusions with respect to drift was introduced in \cite{Yang1991} and the authors describe an extension of their method to reflected diffusions (see \cite[Section 9]{Yang1991}). Here we provide a brief summary of the method in the context of reflected diffusions in convex polyhedral domains and refer the reader to \cite{Yang1991} for the details.

Since we can only consider perturbations to the drift in this section, in addition to the assumptions stated in Section \ref{sec:assumptions}, we also assume that the initial condition, dispersion coefficient and directions of reflection are constant in $\alpha\in U$; that is, $x_0(\cdot)\equiv x_0$, $\sigma(\cdot,\cdot)\equiv \sigma(\cdot)$ and $R(\cdot)\equiv R$. Fix $t<\infty$ and $\alpha\in U$. For $m=1,\dots,M$, by a standard argument using the Lipschitz continuity of $b(\cdot,\cdot)$ and Girsanov's transformation (see, e.g., \cite[Chapter IV.38]{Rogers2000a}), there is a family of probability measures $\{\P_m^{\alpha,\ve}\}_{\ve>0}$ on the measurable space $(\Omega,\F_t)$ that are mutually absolutely continuous with respect to the underlying measure $\P$ with Radon-Nikodym derivative	
	\be\frac{d\P_m^{\alpha,\ve}}{d\P}\doteq\exp\lcb\int_0^t\ip{\mu_m^{\alpha,\ve}(\Z^\alpha(s)),d\bm(s)}-\frac{1}{2}\int_0^t|\mu_m^{\alpha,\ve}(\Z^\alpha(s))|^2ds\rcb,\qquad\ve>0,\ee
where $\mu_m^{\alpha,\ve}:G\mapsto\R^K$ is given by
	$$\mu_m^{\alpha,\ve}(x)\doteq\sigma^T(x)a^{-1}(x)(b(\alpha+\ve e_m,x)-b(\alpha,x)),\qquad x\in G,$$ 
such that $\{\Z^\alpha(s),0\leq s\leq t\}$ under $\P_m^{\alpha,\ve}$ is equal in distribution to $\{\Z^{\alpha+\ve e_m}(s),0\leq s\leq t\}$ under $\P$. Let $D^\alpha$ denote the $M$-dimensional random variable whose $m$th component, denoted $D^{\alpha,m}$ for $m=1,\dots,M$, is defined by
\begin{align*}
	D^{\alpha,m}&\doteq\lim_{\ve\to0}\frac{(d\P_m^{\alpha,\ve}/d\P)-1}{\ve}=\int_0^t\ip{\sigma^T(\Z^\alpha(s))a^{-1}(\Z^\alpha(s))b_{\alpha^m}(\alpha,\Z^\alpha(s)),d\bm(s)},
\end{align*}
where $b_{\alpha^m}(\alpha,x)$ denotes the partial derivative of $b(\cdot,x)$ at $\alpha$ with respect to the $m$th component of $\alpha$. According to \cite[Theorem 3.1]{Yang1991},
\begin{align*}
	\Theta_t'(\alpha)&=\lim_{\ve\to0}\E\lsb\lb\int_0^t\zeta_1(\Z^\alpha(s))ds+\zeta_2(\Z^\alpha(t))\rb\frac{(d\P_m^{\alpha,\ve}/d\P) -1}{\ve}\rsb\\
	&=\E\lsb\lb\int_0^t\zeta_1(\Z^\alpha(s))ds+\zeta_2(\Z^\alpha(t))\rb D^\alpha\rsb.
\end{align*}
As described in Section \ref{sec:EulerRD}, we have an Euler approximation $\Z^{\alpha,\Delta}$ for $\Z^\alpha$, which yields the approximation $D^{\alpha,\Delta,m}$ for $D^{\alpha,m}$, for $m=1,\dots,M$, given by
	\be\label{eq:DalphaDelta}D^{\alpha,\Delta,m}=\sum_{n=0}^{N-1}\ip{\sigma^T(\Z^{\alpha,\Delta}(t_n^\Delta))a^{-1}(\Z^{\alpha,\Delta}(t_n^\Delta))b_{\alpha^m}(\alpha,\Z^{\alpha,\Delta}(t_n^\Delta)),\delta_n^\Delta\bm}.\ee
According to \cite[Theorem 4.1]{Yang1991},
	\be\label{eq:ThetatlimLR}\Theta_t'(\alpha)=\lim_{\Delta\to0}\E\lsb\lb\int_0^t\zeta_1(\Z^{\alpha,\Delta}(s))ds+\zeta_2(\Z^{\alpha,\Delta}(t))\rb D^{\alpha,\Delta}\rsb.\ee
We now propose a numerical algorithm for estimating $\Theta_t'(\alpha)$ based on \eqref{eq:DalphaDelta} and \eqref{eq:ThetatlimLR}. Fix $\Delta\in(0,t)$ so that $N\doteq t/\Delta$ is a positive integer. Then we have 
\begin{align}\label{eq:whThetaLR}
	&\lb\int_0^t\zeta_1(\Z^{\alpha,\Delta}(s))ds+\zeta_2(\Z^{\alpha,\Delta}(t))\rb D^{\alpha,\Delta}=\lb\sum_{n=0}^{N-1}\zeta_1(\Z^{\alpha,\Delta}(t_n^\Delta))\Delta+\zeta_2(\Z^{\alpha,\Delta}(t_N^\Delta))\rb D^{\alpha,\Delta}.
\end{align}

\begin{algorithm}\label{alg:lr}
Set $\Z^{\alpha,\Delta}(t_0^\Delta)=x_0(\alpha)$. For $n=1,\dots,N$, recursively complete the following two steps:
\begin{itemize}
	\item[1.] Generate a $K$-dimensional Gaussian random variable $\delta_n^\Delta\bm$ with mean zero and diagonal covariance matrix $\Delta\Id_K$.
	\item[2.] Define $\Z^{\alpha,\Delta}(t_n^\Delta)$ as in \eqref{eq:Znalphah}.
\end{itemize}
Define $D^{\alpha,\Delta}$ as in \eqref{eq:DalphaDelta} and set the (LR) estimator $\wh{\Theta_t'(\alpha)}_{\text{LR}}$ of $\Theta_t'(\alpha)$ equal to the right-hand side of \eqref{eq:whThetaLR}.
\end{algorithm}

We briefly mention some of the advantages and drawbacks of the LR method.  When applicable, the main advantage of the LR method is that it applies to a broad class of functionals of reflected diffusions, requiring only that $\zeta_1$ and $\zeta_2$ be measurable (and integrable) without imposing additional  smoothness conditions that are  required for our IPA method (Algorithm \ref{alg:ipa}). In addition, since there are exact sampling methods for reflected diffusions (see \cite{Blanchet2014}), they can potentially be adopted to obtain an unbiased estimator of $\Theta_t'(\alpha)$.  On the other hand, as mentioned above, the main drawback of the LR method is that it is only applicable to perturbations of the drift. This is quite restrictive as it is often of interest to compute sensitivities with respect to parameters other than the drift {(e.g., in diffusion approximations of queueing networks, which we plan to study in future work)}. In \cite[Section 8]{Yang1991} the authors show that in the one-dimensional setting, the approach can be adapted to allow for perturbations of the diffusion coefficient, but those methods do not extend to higher dimensions. Furthermore, even when considering perturbations to the drift, numerical evidence suggests that the variance of an LR estimator performs worse than the our method, especially over large time intervals (see Figure \ref{fig:varIPALR} below). {This supports the observation made in \cite[Chapter VII.5]{Asmussen2007} that IPA estimators usually outperform LR estimators when both methods are applicable}.

\begin{remark}
\label{rem-fd}
It is worthwhile to compare our IPA algorithm with the class of FD methods, which can be used to estimate sensitivities.  These are based on approximating the derivative by a finite difference and include forward, backward and central difference approximations  (see, e.g., \cite[Chapter VII.1]{Asmussen2007}). For example, the forward difference estimator associated with fixed $\ve, \Delta > 0$ (and assuming once again that $N\doteq t/\Delta$ is a positive integer) for $\partial_m\Theta_t(\alpha)$, for $m=1,\dots,M$, is given by 
\begin{align}\label{eq:whpartialmTheta}
	&\wh{\partial_m\Theta_t(\alpha)}_{\text{FD}}\\ \notag
	&\qquad=\sum_{n=0}^{N-1}\frac{\zeta_1(\Z^{\alpha+\ve e_m,\Delta}(t_n^\Delta))-\zeta_1(\Z^{\alpha,\Delta}(t_n^\Delta))}{\ve}\Delta+\frac{\zeta_2(\Z^{\alpha+\ve e_m,\Delta}(t_N^\Delta))-\zeta_2(\Z^{\alpha,\Delta}(t_N^\Delta))}{\ve}, 
\end{align}
with $t_n^\Delta$ defined as in \eqref{eq:tnDelta}. The appeal of an FD method is that it is simple to implement and can be used to approximate sensitivities with respect to all of the parameters. The main drawback is that it introduces an extra source of bias from the derivative approximation. Our IPA method can be viewed as a limiting version of the FD method that eliminates the bias from the derivative approximation, without any increase in the computational complexity of implementation. Since there is no computational advantage to using the FD method and in general there do not appear to be any other advantages to using an FD method over an IPA method when both are applicable (see, e.g., the first paragraph at the the top of \cite[page 219]{Asmussen2007}), we do not provide a numerical comparison of our method and an FD method.
\end{remark}

\section{Examples}\label{sec:examples}

We illustrate { the value of} our method (Algorithm \ref{alg:ipa}) with two examples: a one-dimensional RBM and a three-dimensional RBM that arises in the study of rank-based interacting diffusion models of equity markets. We used the program R for all computations, which were performed on an Apple machine with a 2.26 GHz Intel Core 2 Duo processor. 

\subsection{One-dimensional RBM}\label{sec:1dRBM}

Let $J=1$, $U\doteq(0,\infty)\times\R\times(0,\infty)$ and $\bm$ be a one-dimensional $\{\F_t\}$-Brownian motion on $(\Omega,\F,\P)$. For each $\alpha=(x,b,\sigma)\in U$, let $\Z^{(x,b,\sigma)}$ be a one-dimensional RBM in $\R_+$ with initial condition $x$, constant drift $b$ and variance $\sigma^2$, and driving Brownian motion $\bm$. It is well known (see, e.g., \cite[Chapter X.8]{Asmussen2007}) that $\Z^{(x,b,\sigma)}$ satisfies
	$$\Z^{(x,b,\sigma)}(t)=x+bt+\sigma\bm(t)+\Y^{(x,b,\sigma)}(t),\qquad t\geq0,$$
where $\Y^{(x,b,\sigma)}$ is given by
	$$\Y^{(x,b,\sigma)}(t)=\sup_{0\leq s\leq t}\lb-x-bs-\sigma\bm(s)\rb\vee0,\qquad t\geq0.$$
It is readily verified that the assumptions in Section \ref{sec:assumptions} hold. For $t\geq0$, define $\Theta_t:U\mapsto\R_+$ by 
	\be\label{eq:Ftxbsigma}\Theta_t(x,b,\sigma)\doteq\E\lsb\Z^{(x,b,\sigma)}(t)\rsb,\qquad(x,b,\sigma)\in U.\ee
We use our algorithm (Algorithm \ref{alg:ipa}) to estimate the first partial derivatives $\partial_x \Theta_1(1,-1,1)$, $\partial_b \Theta_1(1,-1,1)$ and $\partial_\sigma \Theta_1(1,-1,1)$. In the case of $\partial_b \Theta_1(1,-1,1)$, we compare our method IPA with the LR method. In Table \ref{tab:RBM} we compare the estimates with the analytic values obtained using the formula for the cumulative distribution function of a one-dimensional RBM given on \cite[page 49]{Harrison1985}, and, in the case of perturbations to the drift, with the estimates obtained using the LR method (Algorithm \ref{alg:lr}). In Figure \ref{fig:bias} we plot confidence intervals for the magnitude of the relative biases of the LR estimator and our IPA estimator for different time steps $\Delta$, which, as expected, decrease with $\Delta$. (The ``relative bias'' of an estimator is equal to the difference between the mean of the estimator and the analytic value divided by the magnitude of the analytic value.) In Figure \ref{fig:varIPALR} we compare the empirical variance of our estimator of $\partial_b\Theta_t(1,-1,1)$ with the empirical variance of the LR estimator of $\partial_b\Theta_t(1,-1,1)$ over the time interval $0\leq t\leq 10$. Our estimator performs substantially better than the LR estimator, especially as $t$ increases.

\begin{center}
\begin{table}
\caption{Analytic values (approximated up to four decimal places) of the first partial derivatives of $\Theta_1(\cdot,\cdot,\cdot)$ defined as in \eqref{eq:Ftxbsigma} and evaluated at $(1,-1,1)$, as well as 95\% confidence intervals for the LR estimate (when applicable) and the IPA estimate of the first partial derivatives of $\Theta_1(\cdot,\cdot,\cdot)$. Here $\Delta$ denotes the time step. Each estimate was obtained using $10^5$ trials and took approximately $10\Delta^{-1}$ CPU seconds to compute.}
\label{tab:RBM}
\begin{tabular}{| p{1.95cm} | p{.65cm} | p{2.1cm} | p{2.1cm} | p{2.1cm} |}
	\hline
	Method & $\Delta$ & $\partial_x\Theta_1(1,-1,1)$ & $\partial_b\Theta_1(1,-1,1)$ & $\partial_\sigma\Theta_1(1,-1,1)$ \\ \hline
\end{tabular}
\begin{tabular}{| p{1.95cm} | p{.65cm} | p{2.1cm} | p{2.1cm} | p{2.1cm} |}
	\hline
	Analytic & N/A & $.3319$ & $.4351$ & $.6681$ \\ \hline
\end{tabular}
\begin{tabular}{| p{1.95cm} | p{.65cm} | p{2.1cm} | p{2.1cm} | p{2.1cm} |}
	\hline
	LR & $10^{-2}$ & N/A & $.4504\pm.0071$ & N/A \\ \hline
	IPA & $10^{-2}$ & $.3634\pm.0029$ & $.4527\pm.0027$ & $.6209\pm.0035$ \\ \hline
\end{tabular}
\begin{tabular}{| p{1.95cm} | p{.65cm} | p{2.1cm} | p{2.1cm} | p{2.1cm} |}
	\hline
	LR & $10^{-3}$ & N/A & $.4415\pm.0073$ & N/A \\ \hline

	IPA & $10^{-3}$ & $.3414\pm.0029$ & $.4411\pm.0027$ & $.6527\pm.0035$ \\ \hline
\end{tabular}
\begin{tabular}{| p{1.95cm} | p{.65cm} | p{2.1cm} | p{2.1cm} | p{2.1cm} |}
	\hline
	LR & $10^{-4}$ & N/A & $.4329\pm.0072$ & N/A \\ \hline
	IPA & $10^{-4}$ & $.3359\pm.0029$ & $.4393\pm.0025$ & .$6611\pm.0035$ \\ \hline
\end{tabular}
\end{table}
\end{center}

\begin{figure}
        \begin{subfigure}{0.4\textwidth}
                \centering
                \includegraphics[width=.95\linewidth]{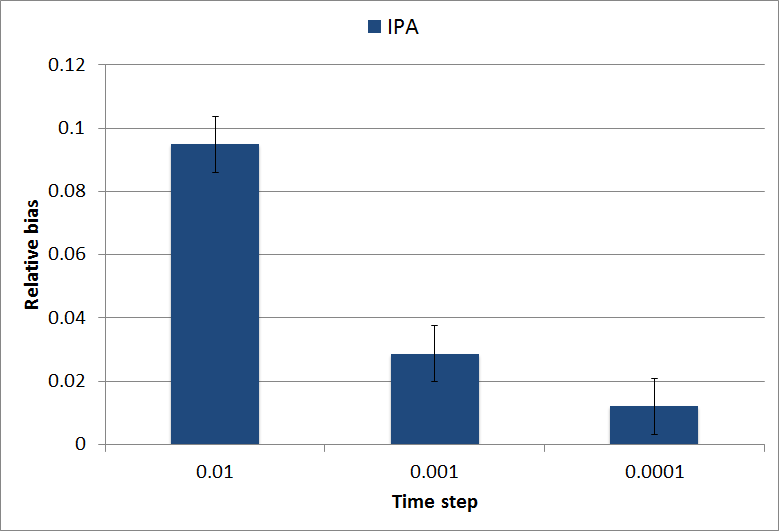}
                \caption{$\partial_x\Theta_1(1,-1,1)$}
                \label{fig:Dx}
        \end{subfigure}
        \begin{subfigure}{0.4\textwidth}
                \centering
                \includegraphics[width=.95\linewidth]{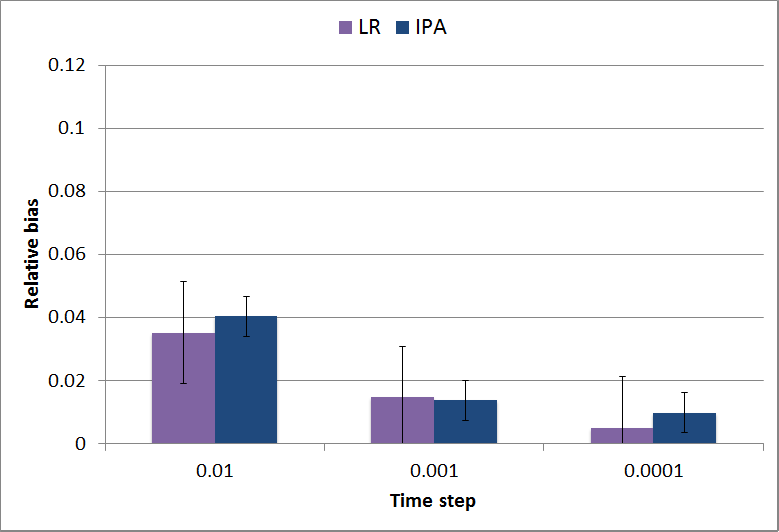}
                \caption{$\partial_b\Theta_1(1,-1,1)$}
                \label{fig:b'}
        \end{subfigure}
        \begin{subfigure}{0.4\textwidth}
                \centering
                \includegraphics[width=.95\linewidth]{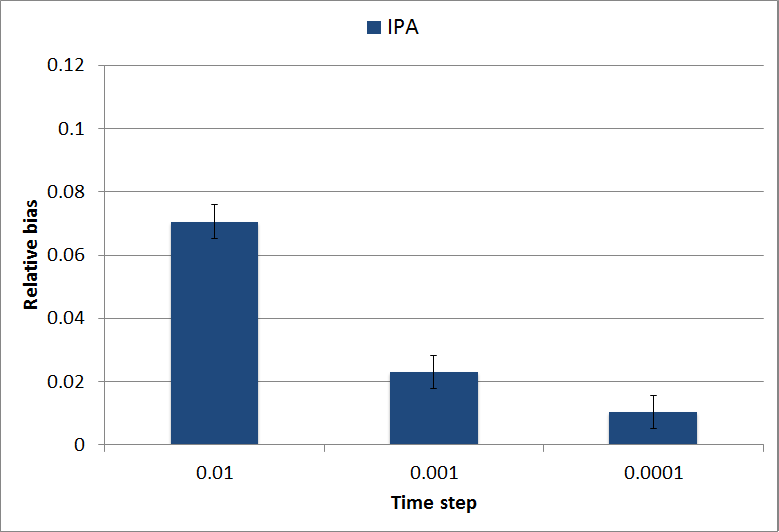}
                \caption{$\partial_\sigma\Theta_1(1,-1,1)$}
                \label{fig:Dsigma}
        \end{subfigure}%
        \caption{95\% confidence intervals for the magnitude of the relative biases of the LR estimators (when applicable) and IPA estimators for the first partial derivatives of $\Theta_1(\cdot,\cdot,\cdot)$ at $(1,-1,1)$ using different time steps. Each estimate was obtained using $10^5$ trials.}\label{fig:bias}
\end{figure}

\begin{figure}
	\includegraphics[width=.8\textwidth]{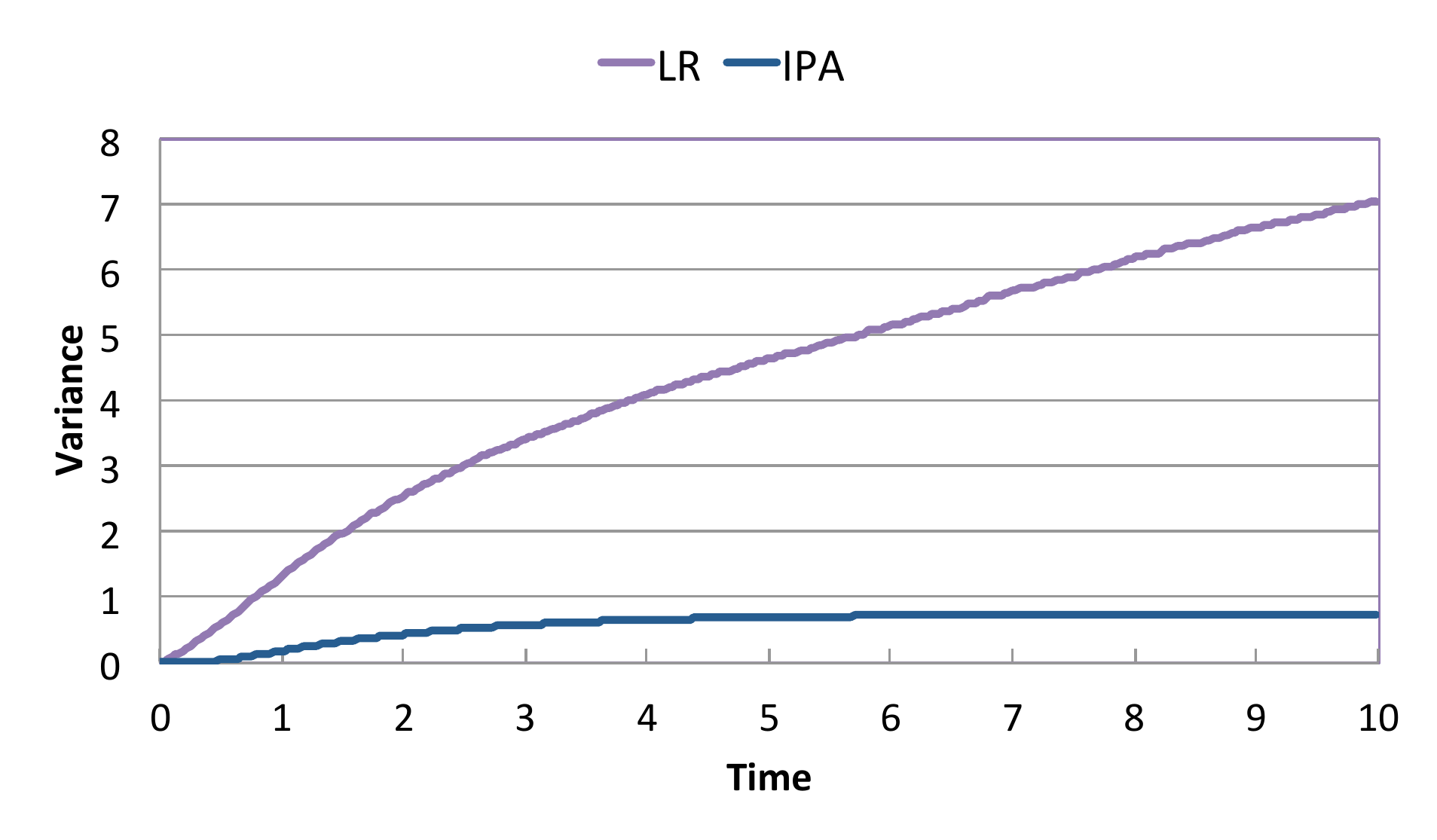}
	\caption{Empirical variance of the LR and IPA estimators for $\partial_b\Theta_t(1,-1,1)$, for $0\leq t\leq 10$. The estimates were obtained using $10^5$ trials and time step $\Delta=.01$.}\label{fig:varIPALR}
\end{figure}

\subsection{Rank-based interacting diffusion model of equity markets}
\label{subs-Atlas}

We consider a rank-based interacting diffusion model, referred to as the \emph{Atlas model}, which was first introduced by Fernholz in \cite[Example 5.3.3]{Fernholz2002} to study equity markets, and subsequently generalized by Banner, Fernholz and Karatzas \cite{Banner2005} and Ichiba et.\ al.\ \cite{Ichiba2011}. In the usual model of an equity market, the growth and volatility of each stock are fixed and depend on the identity of the stock, whereas in the class of Atlas models, the growth and volatility of each stock are determined by its relative rank within the market and thus are time-varying (and discontinuous). We consider the simplest version of the Atlas model (with $J\geq2$ stocks) in which the stocks have equal volatility, and the smallest stock, referred to as the \emph{Atlas stock}, has positive growth rate, whereas all other stocks have zero growth.  More precisely, let $g>0$, $\sigma>0$ and $\bm$ be a $J$-dimensional Brownian motion. Let $S=\{S(t),t\geq0\}$ denote the $J$-dimensional process whose $j$th component at time $t\geq0$, denoted $S^j(t)$, is equal to the value of the $j$th stock at time $t$. Then $S$ evolves according to  the stochastic differential equation (with discontinuous drift)
	\be\label{eq:dlogSt}d(\log S^j(t))=\gamma^j(t)dt+\sigma d\bm^j(t),\qquad t\geq0,\qquad j=1,\dots,J,\ee
where $\gamma^j(t)$ is equal to $Jg$ if $S^j(t)<S^k(t)$ for all $k\neq j$ and equal to zero otherwise, for $j=1,\dots,J$. Weak existence and uniqueness of solutions to \eqref{eq:dlogSt} follows from their relation to martingale problems, and existence and uniqueness results for martingale problems established in Stroock and Varadhan \cite{Stroock2006} and Bass and Pardoux \cite{Bass1987}. Let $\Z=\{\Z(t),t\geq0\}$ be the \emph{rank-ordered process}, where $\Z(t)$ is obtained by permuting the components of $\log S(t)$ (the logarithm is applied componentwise) so that they are in descending order, for $t\geq0$. Then according to \cite{Banner2005} (see the paragraph following the proof of Proposition 2.3 on page 2303) $\Z$ is a $J$-dimensional RBM in the Weyl chamber $G\doteq\{x\in\R^J:x^1\geq x^2\geq\cdots\geq x^J\}$ with drift vector $(0,\dots,0,Jg)^T$, diagonal covariance matrix $\sigma^2\Id_J$, and normal reflection along each of the faces $F_i\doteq\{x\in\partial G:x^i=x^{i+1}\}$, for $i=1,\dots,J-1$. 

In many applications in mathematical finance, it is of interest to compute sensitivities of functionals of the stocks in a market with respect to the underlying parameters of the market. For example, sensitivities of derivative prices with respect to the underlying market parameters, commonly referred to as the ``Greeks'', play an important role in risk management and hedging strategies (see, e.g., \cite[Chapter 7]{Glasserman2003}). Here, we study sensitivities related to  the Atlas model, which have not yet received much attention. A performance measure of interest in this context is the so-called diversity of the equity market, which is expressed in terms of the $J$-dimensional process  $\mu=\{\mu(t),t\geq0\}$ of \emph{relative market capitalizations} defined by $\mu(t)\doteq C(S(t))$ for $t\geq0$, where $C:(0,\infty)^J\mapsto\{x\in(0,1)^J:x^1+\dots+x^J=1\}$ is defined by
	\be\label{eq:mu}C^j(x)\doteq\frac{x^j}{x^1+\cdots+x^J},\qquad x\in(0,\infty)^J,\qquad j=1,\dots,J.\ee
Fix $p\in(0,1)$ and define $f:\R_+^J\to\R_+$ by 
	\be\label{eq:f}f(x)=(x^1)^p+\cdots+(x^J)^p,\qquad x\in\R_+^J.\ee
Then $f$ is the $p$th power of the \emph{diversity function} and $f(\mu(t))$ is a measure of the diversity of the market (see \cite[Example 3.4.4]{Fernholz2002}). Since $f(x)$ is invariant under permutations of the components of $x$ and $\Z$ is obtained by permuting the components of $\log S$, it follows from \eqref{eq:mu} and \eqref{eq:f} that $f(\mu(t))=\zeta(\Z(t))$, where $\zeta:G\mapsto[J^{p-1},1)$ is the continuously differentiable function defined by
	$$\zeta(x)\doteq\frac{e^{px^1}+\cdots+e^{px^J}}{\lb e^{x^1}+\cdots+e^{x^J}\rb^p},\qquad x\in G.$$
A straightforward computation shows that the gradient of $\zeta$ is bounded and continuous. Here, we first use a FD method based on the stochastic differential equation representation \eqref{eq:dlogSt} and then our IPA method for the reflected diffusion representation to estimate the sensitivity of $\E[f(\mu(t))]$ to an underlying parameter of the market.

Fix $\sigma>0$ and set $U\doteq(0,\infty)$. For each $\alpha\in U$ set the growth rate to be $g(\alpha)\doteq\sigma^2/2\alpha$, let $S^\alpha=\{S^\alpha(t),t\geq0\}$ be the $J$-dimensional process satisfying \eqref{eq:dlogSt} with $g=g(\alpha)$, let $\mu^\alpha=\{\mu^\alpha(t),t\geq0\}$ be the process of relative market capitalizations defined by $\mu^\alpha(t)\doteq C(S^\alpha(t))$ for $t\geq0$ and define
	$$\Theta_t(\alpha)\doteq\E\lsb f(\mu^\alpha(t))\rsb,\qquad t\geq0.$$
Let $t\geq0$ and $\alpha\in U$. Assuming that $\Theta_t(\cdot)$ is differentiable, it follows that, as $\ve\to0$,
	$$\Theta_t'(\alpha)=\frac{\E\lsb f(\mu^{\alpha+\ve}(t))\rsb-\E\lsb f(\mu^\alpha(t))\rsb}{\ve}+o(1).$$
In \cite{Yan2002} Yan proposed an Euler scheme for approximating solutions to stochastic differential equations with discontinuous drift and proved convergence properties of the Euler scheme. Using the algorithm described in \cite[Section 4]{Yan2002} one can obtain estimators $\wh{S^{\alpha+\ve}}$ and $\wh{S^\alpha}$ (using ``common random numbers'', see \cite[Chapter VII.1]{Asmussen2007}) for $S^{\alpha+\ve}$ and $S^\alpha$, respectively, and use the (forward-difference) approximation
	$$\Theta_t'(\alpha)\approx\frac{f(C(\wh{S^{\alpha+\ve}}(t)))-f(C(\wh{S^\alpha}(t)))}{\ve}.$$
Alternatively, let $\Z^\alpha$ denote the associated rank-ordered process so that $\Theta_t(\alpha)=\E[\zeta(\Z^\alpha(t))]$. Then we can use our method to estimate $\Theta_t(\alpha)$. Let $J=3$, $\sigma^2=3\times10^{-4}$, $p=1/2$ and $\Z^\alpha(0)=0$. In Figure \ref{fig:Atlas}, we plot estimates of $\Theta_t'(1)$ over the time interval $0\leq t\leq 5000$ obtained using Algorithm \ref{alg:ipa}. In Figure \ref{fig:AtlasVar} we plot the empirical variance of our IPA estimator (Algorithm \ref{alg:ipa}) and the FD estimator (described in \cite{Yan2002}) for $\Theta_{5000}'(1)$. As seen in Figure \ref{fig:AtlasVar}, the empirical variance of the FD estimators appears to diverge as $\ve$ goes to zero.

\begin{figure}
	\includegraphics[width=.8\textwidth]{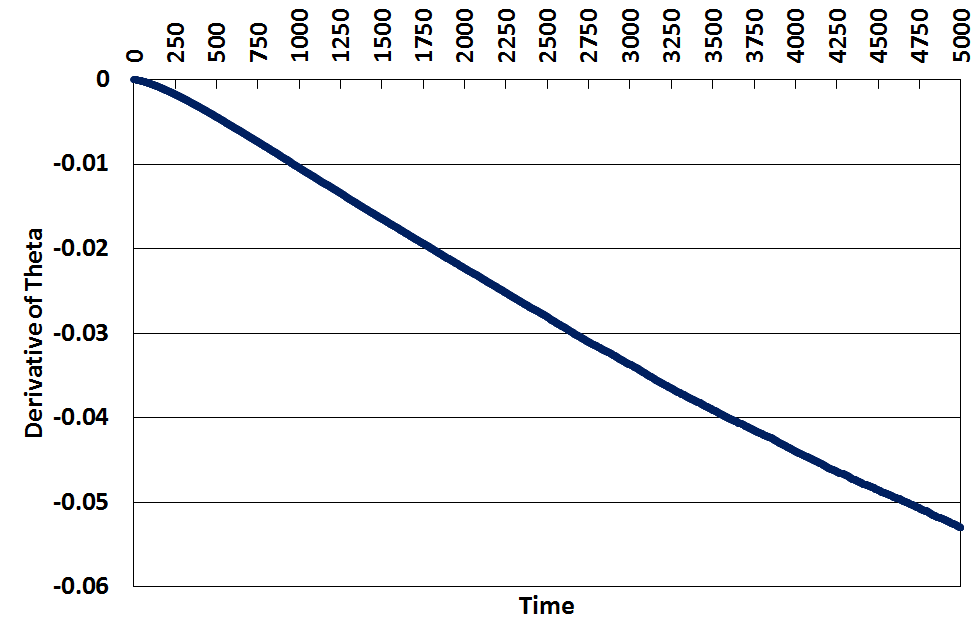}
	\caption{IPA estimates of $\Theta_t'(1)$ for $0\leq t\leq 5000$. The estimates were obtained using $10^5$ trials and time step $\Delta=1$.}\label{fig:Atlas}
\end{figure}

\begin{figure}
	\includegraphics[width=.8\textwidth]{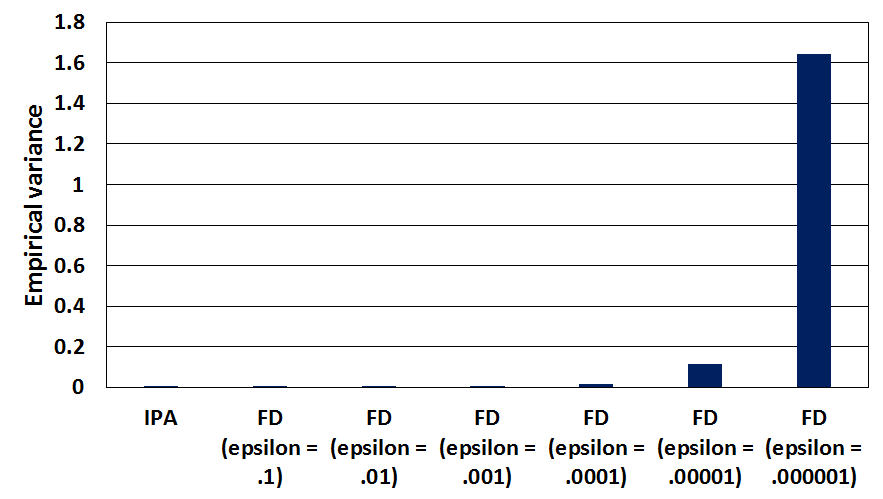}
	\caption{Empirical variance of the IPA estimator and FD estimators (for $\ve=.1,.01,.001,.0001,.00001,.000001$) of $\Theta_{5000}'(1)$. The estimates were obtained using $10^5$ trials and time step $\Delta=1$.}\label{fig:AtlasVar}
\end{figure}

\section{The Skorokhod problem and the derivative problem}\label{sec:spdp}

In this section we recall the deterministic SP, extended Skorokhod problem (ESP) and derivative problem. We state some useful properties and also state a new continuity result for the derivative problem (see Theorem \ref{thm:dmcontinuous} below) that is needed in the proof of Theorem \ref{thm:approx} and whose proof is deferred to  Section \ref{sec:dmcontinuous}. Throughout this section we fix $\alpha\in U$. Recall that Assumptions \ref{ass:independent}, \ref{ass:setB} and \ref{ass:projection} hold.

\subsection{The Skorokhod problem and extended Skorokhod problem}\label{sec:sp}

\begin{defn}\label{def:sp}
Given $\x\in\dr_G(\R^J)$, a pair $(\z,\push)\in\dr(G)\times\dr(\R_+^N)$ is a solution to the SP $\{(d_i(\alpha),n_i,c_i),i\in\allN\}$ for $\x$ if
	$$\z(t)=\x(t)+R(\alpha)\push(t),\qquad t\geq0,$$
where $\push(0)=0$ and for each $i\in\allN$, the $i$th component of $\push$, denoted $\push^i$, is nondecreasing and can only increase when $\z$ lies in face $F_i$; that is,
	$$\int_0^\infty 1_{\{\z(s)\not\in F_i\}}d\push^i(s)=0.$$
If there exists a unique solution $(\z,\push)$ to the SP $\{(d_i(\alpha),n_i,c_i),i\in\allN\}$ for $\x$, we write $\z=\sm^\alpha(\x)$ and refer to $\sm^\alpha$ as the associated SM.
\end{defn}

In the following we state the ESP, which is an extension of the SP that allows for a constraining term with unbounded variation. We introduce the ESP formulation because it will be more convenient to work with in subsequent proofs.

\begin{defn}\label{def:esp}
Given $\x\in\dr_G(\R^J)$, a pair $(\z,\y)\in\dr(\R^J)\times\dr(\R^J)$ is a solution to the ESP $\{(d_i(\alpha),n_i,c_i),i\in\allN\}$ for $\x$ if $\z(0)=\x(0)$ and the following conditions hold:
\begin{itemize}
	\item[1.] $\z(t)=\x(t)+\y(t)$ for all $t\geq0$;
	\item[2.] $\z(t)\in G$ for all $t\geq0$;
	\item[3.] for all $0\leq s<t<\infty$,
		$$\y(t)-\y(s)\in\conv\lsb\cup_{u\in(s,t]}d(\alpha,\z(u))\rsb;$$
	\item[4.] for all $t>0$, $\y(t)-\y(t-)\in\conv\lsb d(\alpha,\z(t))\rsb$.
\end{itemize}
If there exists a unique solution $(\z,\y)$ to the ESP $\{(d_i(\alpha),n_i,c_i),i\in\allN\}$ for $\x$, we write $\z=\esm^\alpha(\x)$ and refer to $\esm^\alpha$ as the associated ESM.
\end{defn}

\begin{remark}\label{rmk:X}
Given a reflected diffusion $\Z^\alpha$ with $\L^\alpha$ satisfying the conditions in Definition \ref{def:rd} define the $\{\F_t\}$-adapted $J$-dimensional continuous process $\X^\alpha=\{\X^\alpha(t),t\geq0\}$ by
	\be\label{eq:X}\X^\alpha(t)\doteq x_0(\alpha)+\int_0^tb(\alpha,\Z^{\alpha}(s))ds+\int_0^t\sigma(\alpha,\Z^{\alpha}(s))d\bm(s),\qquad t\geq0.\ee
Then by the properties of $(\Z^\alpha,\L^\alpha)$ stated in Definition \ref{def:rd}, $(\Z^\alpha,\L^\alpha)$ is a solution to the SP $\{(d_i(\alpha),n_i,c_i),i\in\allN\}$ for $\X^\alpha$. Define the constraining process $\Y^\alpha$ by $\Y^\alpha\doteq R(\alpha)\L^\alpha$. Then by the properties of $\Y^\alpha$ stated in Remark \ref{rmk:Y}, $(\Z^\alpha,\Y^\alpha)$ is a solution to the ESP $\{(d_i(\alpha),n_i,c_i),i\in\allN\}$ for $\X^\alpha$.
\end{remark}

The following is a time-shift property to the ESP. A similar property holds for the SP; however, it is not needed in this work.

\begin{lem}[{\cite[Lemma 2.3]{Ramanan2006}}]
\label{lem:sptimeshift}
Suppose $(\z,\y)$ is a solution to the ESP $\{(d_i(\alpha),n_i,c_i),i\in\allN\}$ for $\x\in\dr_G(\R^J)$. Let $s\geq0$ and define $\z^s\in\dr(G)$, $\y^s\in\dr(\R^J)$ and $\x^s\in\dr_G(\R^J)$ by
\begin{align}\label{eq:zs}
	\z^s(\cdot)&\doteq\z(s+\cdot),\qquad \y^s(\cdot)\doteq\y(s+\cdot)-\y(s),\qquad\x^s(\cdot)\doteq\z(s)+\x(s+\cdot)-\x(s).
\end{align}
Then $(\z^s,\y^s)$ is a solution to the ESP $\{(d_i(\alpha),n_i,c_i),i\in\allN\}$ for $\x^s$.
\end{lem}

The following proposition ensures the SP and ESP are well defined and the associated SM and ESM are Lipschitz continuous.

\begin{prop}\label{prop:sp}
Let $\alpha\in U$. 
\begin{itemize}
	\item[(i)] There exists $\lip_\sm(\alpha)<\infty$ such that if, for $k=1,2$, $\x_k\in\dr_G(\R^J)$ and $(\z_k,\ell_k)$ is the solution to the SP $\{(d_i(\alpha),n_i,c_i),i\in\allN\}$ for $\x_k$, then for all $t\geq0$,
	\begin{align}\label{eq:smlip}
		\sup_{0\leq s\leq t}|\z_1(s)-\z_2(s)|+\sup_{0\leq s\leq t}|\ell_1(s)-\ell_2(s)|\leq\lip_\sm(\alpha)\sup_{0\leq s\leq t}|\x_1(s)-\x_2(s)|.
	\end{align}
	\item[(ii)] There exists $\lip_{\esm}(\alpha)<\infty$ such that if, for $k=1,2$, $\x_k\in\dr_G(\R^J)$ and $(\z_k,\y_k)$ is the solution to the ESP $\{(d_i(\alpha),n_i,c_i),i\in\allN\}$ for $\x_k$, then for all $t\geq0$,
	\begin{align}\label{eq:esmlip}
		\sup_{0\leq s\leq t}|\z_1(s)-\z_2(s)|+\sup_{0\leq s\leq t}|\y_1(s)-\y_2(s)|\leq\lip_{\esm}(\alpha)\sup_{0\leq s\leq t}|\x_1(s)-\x_2(s)|.
	\end{align}
	\item[(iii)] Given $\x\in\dr_G(\R^J)$ there exists a unique solution $(\z,\ell)$ to the SP $\{(d_i(\alpha),n_i,c_i),i\in\allN\}$ for $\x$.
	\item[(iv)] Given $\x\in\dr_G(\R^J)$, there exists a unique solution $(\z,\y)$ to the ESP $\{(d_i(\alpha),n_i,c_i),i\in\allN\}$ for $\x$.
\end{itemize}
\end{prop}

\begin{proof}
The Lipschitz continuity of the ESM \eqref{eq:esmlip} stated in part (ii) follows from Assumption \ref{ass:setB} and \cite[Theorem 3.3]{Ramanan2006}. Part (iv) follows from part (ii) and \cite[Lemma 2.6]{Ramanan2006}. Parts (i) and (iii) then follow from parts (ii) and (iv), Assumption \ref{ass:independent}, and \cite[Lemma 2.1]{Lipshutz2017}.
\end{proof}

We close this section by stating a version of the boundary jitter property that was introduced in \cite[Definition 3.1]{Lipshutz2016}. The boundary jitter property will be important for establishing continuity properties of the so-called derivative map.

\begin{defn}\label{def:jitter}
A pair $(\z,\y)\in\cts(G)\times\cts(\R^J)$ satisfies the boundary jitter property if the following hold:
\begin{itemize}
	\item[1.] If $\z(t)\in\partial G$ for some $t_1<t<t_2$, then $\y$ is nonconstant on $(t_1\vee0,t_2)$.
	\item[2.'] $\z$ does not spend positive Lebesgue time in the boundary $\partial G$; that is,
		$$\int_0^\infty 1_{\partial G}(\z(s))ds=0.$$
	\item[3.] If $\z(t)\in\U$ for some $t>0$, then for each $i\in\allN(\z(t))$ and all $\delta\in(0,t)$ there exists $s\in(t-\delta,t)$ such that $\allN(\z(s))=\{i\}$.
	\item[4.] If $\z(0)\in\U$, then for each $i\in\allN(\z(0))$ and all $\delta>0$, there exists $s\in(0,\delta)$ such that $\allN(\z(s))=\{i\}$.
\end{itemize}
\end{defn}

\begin{remark}
Condition 2' of the boundary jitter property is slightly stronger than condition 2 of \cite[Definition 3.1]{Lipshutz2017}, which only requires that $\z$ does not spend positive Lebesgue time in the \emph{nonsmooth} part of the boundary (i.e., where two or more faces intersect). The stronger version of condition 2 stated in Definition \ref{def:jitter} is used in the proof of Theorem \ref{thm:dmcontinuous} below.
\end{remark}

For the following, given $\alpha\in U$ and the associated reflected diffusion $\Z^\alpha$, let $\Y^\alpha$ denote the constraining process introduced in Remark \ref{rmk:Y}.

\begin{prop}\label{prop:jitter}
Let $\alpha\in U$. Then a.s.\  $(\Z^\alpha,\Y^\alpha)$ satisfies the boundary jitter property.
\end{prop}

\begin{proof}
Conditions 1, 3 and 4 of the jitter property hold due to \cite[Theorem 3.3]{Lipshutz2017}. Condition 2' follows from \cite[Lemma 4.3]{Lipshutz2017}.
\end{proof}

\subsection{The derivative problem}

The derivative problem was first introduced in \cite{Lipshutz2016} as an axiomatic framework for studying directional derivatives of the ESM. The formulation of the derivative problem can be thought of as a linearization to the ESP (note the similarities between Definition \ref{def:esp} and Definition \ref{def:dp}).

\begin{defn}\label{def:dp}
Given $\x\in\dr_G(\R^J)$, suppose $(\z,\y)$ is a solution to the ESP $\{(d_i(\alpha),n_i,c_i),i\in\allN\}$ for $\x$. Let $\psi\in\dr(\R^J)$. Then $(\phi,\eta)\in\dr(\R^J)\times\dr(\R^J)$ is a solution to the derivative problem along $\z$ for $\psi$ if $\eta(0)\in\spaan[d(\alpha,\z(0))]$ and the following conditions hold:
\begin{itemize}
	\item[1.] $\phi(t)=\psi(t)+\eta(t)$ for all $t\geq0$;
	\item[2.] $\phi(t)\in \hyper_{\z(t)}$ for all $t\geq0$;
	\item[3.] for all $0\leq s<t<\infty$,
		$$\eta(t)-\eta(s)\in\spaan\lsb\cup_{u\in(s,t]}d(\alpha,\z(u))\rsb;$$
	\item[4.] for all $t>0$, $\eta(t)-\eta(t-)\in\spaan\lsb d(\alpha,\z(t))\rsb$.
\end{itemize}
If there exists a unique solution $(\phi,\eta)$ to the derivative problem along $\z$ for $\psi$, we write $\phi=\dm_\z^\alpha(\psi)$ and refer to $\dm_{\z}^\alpha$ as the derivative map associated with $\z$.
\end{defn}

\begin{remark}
Definition \ref{def:dp} is slightly different from the definition given in \cite[Definition 3.4]{Lipshutz2016}, where $\z$ was assumed to be continuous. Here we allow for $\z$ to lie in $\dr(\R^J)$ and impose condition 4, which follows from condition 3 when $\z$ is continuous. When $\z$ is not continuous, condition 3 only guarantees that $\eta(t)-\eta(t-)\in\spaan\lsb d(\alpha,\z(t-))\cup d(\alpha,\z(t))\rsb$.
\end{remark}

\begin{remark}\label{rmk:psib} 
Given a derivative process $\phib^\alpha$ along a reflected diffusion $\Z^\alpha$, define the continuous $\{\F_t\}$-adapted process $\psib^\alpha=\{\psib^\alpha(t),t\geq0\}$ taking values in $\R^{J\times M}$, for $t\geq0$, by
\begin{align}\label{eq:psib}
	\psib^\alpha(t)&\doteq x_0'(\alpha)+\int_0^tb_\alpha(\alpha,\Z^\alpha(s))ds+\int_0^tb_x(\alpha,\Z^\alpha(s))\phib^\alpha(s) ds\\ \notag
	&\qquad+\int_0^t\sigma_\alpha(\alpha,\Z^\alpha(s))d\bm(s)+\int_0^t\sigma_x(\alpha,\Z^\alpha(s))\phib^\alpha(s) d\bm(s)+R'(\alpha)\L^\alpha(t).
\end{align}
For each $m=1,\dots,M$, it follows from the properties of $(\phib_m^\alpha,\etab_m^\alpha)$ stated in Definition \ref{def:de} that a.s.\ $(\phib_m^\alpha,\etab_m^\alpha)$ is a solution to the derivative problem along $\Z^\alpha$ for $\psib_m^\alpha$. In other words, a.s.\ $\phib_m^\alpha=\dm_{\Z^{\alpha}}^\alpha(\psib_m^\alpha)$.
\end{remark}

In the following we state some important properties of the derivative problem and its associated derivative map. The first lemma states a time-shift property of the derivative problem.

\begin{lem}
\label{lem:dptimeshift}
Let $(\z,\y)$ be a solution to the ESP $\{(d_i,n_i,c_i),i\in\allN\}$ for $\x\in\dr_G(\R^J)$. Suppose $(\phi,\eta)$ is a solution to the derivative problem along $\z$ for $\psi\in\dr(\R^J)$. Let $s\geq0$ and define $\z^{s}\in\dr(G)$ as in \eqref{eq:zs}, and define $\phi^{s},\eta^{s},\psi^s\in\dr(\R^J)$ by
\begin{align}\label{eq:phis}
	\phi^{s}(\cdot)&\doteq\phi(s+\cdot),\qquad\eta^{s}(\cdot)\doteq\eta(s+\cdot)-\eta(s),\qquad\psi^{s}(\cdot)\doteq\phi(s)+\psi(s+\cdot)-\psi(s).
\end{align}
Then $(\phi^s,\eta^s)$ is a solution to the derivative problem along $\z^s$ for $\psi^s$.
\end{lem}

\begin{proof}
This was shown for $\x\in\cts_G(\R^J)$ in \cite[Lemma 5.2]{Lipshutz2016}. A similar argument can be used to prove the lemma in the case that $\x\in\dr_G(\R^J)$. The only difference is that condition 4 must also be verified. To see that condition 4 holds, fix $s\geq0$ and let $t>0$. By the definition of $\eta^s$, condition 4 to the DP and the definition of $\z^s$,
\begin{align*}
	\eta^s(t)-\eta^s(t-)&=\eta(s+t)-\eta((s+t)-)\in\spaan\lsb d(\alpha,\z(s+t))\rsb=\spaan\lsb d(\alpha,\z^s(t))\rsb.
\end{align*}
This proves condition 4 of the DP holds.
\end{proof}

\begin{prop}
\label{prop:dmlip}
Let $\alpha\in U$. There exists $\lip_\dm(\alpha)<\infty$ such that if $(\z,\y)$ is the solution to the ESP $\{(d_i(\alpha),n_i,c_i),i\in\allN\}$ for $\x\in\dr_G(\R^J)$, and, for $k=1,2$, $(\phi_k,\eta_k)$ is a solution to the derivative problem along $\z$ for $\psi_k\in\dr(\R^J)$, then for all $t<\infty$,
	\be\sup_{0\leq s\leq t}|\phi_1(s)-\phi_2(s)|\leq\lip_{\dm}(\alpha)\sup_{0\leq s\leq t}|\psi_1(s)-\psi_2(s)|.\ee
\end{prop}

\begin{proof}
When $\x\in\cts_G(\R^J)$, the proposition follows from \cite[Theorem 5.4]{Lipshutz2016}. When $\x$ is not continuous, the proof follows exactly analogously to the proof of \cite[Theorem 5.4]{Lipshutz2016} except that \cite[equation (3.5)]{Lipshutz2016} does not follow from condition 3 of the derivative problem, but rather is due to condition 4 of the derivative problem.
\end{proof}

For the following, let $\U$ denote the nonsmooth part of the boundary; that is,
	\be\label{eq:nonsmooth}\U\doteq\lcb x\in\partial G:|\allN(x)|\geq 2\rcb.\ee

\begin{prop}\label{prop:dp}
Let $\alpha\in U$. Given $\x\in\cts_G(\R^J)$ suppose the solution $(\z,\y)$ to the ESP for $\x$ satisfies the boundary jitter property (Definition \ref{def:jitter}). Then given $\psi\in\dr(\R^J)$ there exists a unique solution $(\phi,\eta)$ to the derivative problem along $\z$ for $\psi$. Furthermore, $\phi$ is continuous at times $t>0$ that $\z(t)\in G^\circ\cup\U$.
\end{prop}

\begin{proof}
Uniqueness follows from Proposition \ref{prop:dmlip}. Existence and the continuity property follows from \cite[Lemma 8.2]{Lipshutz2016} and \cite[Theorem 3.3]{Lipshutz2016}.
\end{proof}

We close this section with the following continuity property of the derivative map that will be instrumental in the proof of our main result. In addition to its use in this work, it is also used in \cite{Lipshutz2017b} to show that the joint reflected diffusion and derivative process is Feller continuous. In contrast to the other results in this section, we explicitly state exactly which assumptions are required for the following theorem. Recall that we have equipped $\dr(\R^J)$ with the Skorokhod $J_1$-topology.

\begin{theorem}\label{thm:dmcontinuous}
Suppose Assumptions \ref{ass:setB} and \ref{ass:projection} hold. Given $\x\in\cts_G(\R^J)$ suppose the solution $(\z,\y)$ to the ESP for $\x$ satisfies the boundary jitter property (Definition \ref{def:jitter}). Let $\{\x_\ellindex\}_{\ellindex\in\N}$ be a sequence in $\dr(\R^J)$ such that $\x_\ellindex$ converges to $\x$ in $\dr(\R^J)$ as $\ellindex\to\infty$ and for each $\ellindex\in\N$, let $(\z_\ellindex,\y_\ellindex)$ denote the solution to the ESP for $\x_\ellindex$. Suppose $\psi\in\cts(\R^J)$ satisfies $\psi(0)\in \hyper_{\z(0)}$ and $\{\psi_\ellindex\}_{\ellindex\in\N}$ is a sequence in $\dr(\R^J)$ such that $\psi_\ellindex$ converges to $\psi$ in $\dr(\R^J)$ as $\ellindex\to\infty$. Then $\dm_{\z_\ellindex}^\alpha(\psi_\ellindex)$ converges to $\dm_\z^\alpha(\psi)$ in $\dr(\R^J)$ as $\ellindex\to\infty$.
\end{theorem}

The proof of Theorem \ref{thm:dmcontinuous} is deferred to Section \ref{sec:dmcontinuous}. 

\section{Convergence of the Euler scheme for reflected diffusions}\label{sec:EulerRDproof}

For $\Delta>0$ recall the sequence $\{t_n^\Delta\}_{n\in\N_0}$ in $\R_+$ defined in \eqref{eq:tnDelta}, define the RCLL step function $\rho^\Delta:\R_+\mapsto\R_+$ by 
	\be\label{eq:rhoDelta}\rho^\Delta(t)=t_n^\Delta,\qquad t\in[t_n^\Delta,t_{n+1}^\Delta),\qquad n\in\N_0,\ee
recall the definition of the discrete filtration $\{\F_t^\Delta\}$ given in \eqref{eq:FtDelta}, and define the piecewise constant RCLL $\{\F_t^\Delta\}$-adapted process $W^\Delta=\{W^\Delta(t),t\geq0\}$ on $(\Omega,\F,\P)$ by
	\be\label{eq:WDelta}W^\Delta(t)=W(t_n^\Delta),\qquad t\in[t_n^\Delta,t_{n+1}^\Delta),\qquad n\in\N_0.\ee
Observe that $\rho^\Delta$ and $W^\Delta$ are constant on the interval $[t_n^\Delta,t_{n+1}^\Delta)$ for each $n\in\N_0$ and
\begin{align}\label{eq:rhoDeltatnDelta}
	\rho^\Delta(t_n^\Delta)-\rho^\Delta(t_n^\Delta-)&=\Delta,\\ \label{eq:WDeltatnDelta}
	W^\Delta(t_n^\Delta)-W^\Delta(t_n^\Delta-)&=\delta_n^\Delta W,
\end{align}
for each $n\in\N$, where $\delta_n^\Delta W$ is defined as in \eqref{eq:deltanDelta}. Recall the piecewise constant RCLL $\{\F_t^\Delta\}$-adapted processes $(\Z^{\alpha,\Delta},\L^{\alpha,\Delta})$ taking values in $G\times\R_+^N$ defined in \eqref{eq:ZLDelta0}--\eqref{eq:XinalphaDelta}. By \eqref{eq:ZLDelta0}--\eqref{eq:XinalphaDelta} and \eqref{eq:tnDelta}, we have, for each $n\in\N$,
\begin{align*}
	\Z^{\alpha,\Delta}(t_n^\Delta)&=Z^{\alpha,\Delta}(t_0^\Delta)+\sum_{j=1}^n\left(\Z^{\alpha,\Delta}(t_j^\Delta)-\Z^{\alpha,\Delta}(t_{j-1}^\Delta)\right)\\
	&=x_0(\alpha)+\sum_{j=1}^n\left(\Xi_j^{\alpha,\Delta}-\Z^{\alpha,\Delta}(t_{j-1}^\Delta)+\pi^\alpha(\Xi_j^{\alpha,\Delta})-\Xi_j^{\alpha,\Delta}\right)\\
	&=x_0(\alpha)+\sum_{j=1}^nb(\alpha,\Z^{\alpha,\Delta}(t_{j-1}^\Delta))\Delta+\sum_{j=1}^n\sigma(\alpha,\Z^{\alpha,\Delta}(t_{j-1}^\Delta))\delta_j^\Delta W+R(\alpha)\sum_{j=1}^n\xi^\alpha(\Xi_j^{\alpha,\Delta})\\
	&=x_0(\alpha)+\sum_{j=1}^nb(\alpha,\Z^{\alpha,\Delta}(t_{j-1}^\Delta))\Delta+\sum_{j=1}^n\sigma(\alpha,\Z^{\alpha,\Delta}(t_{j-1}^\Delta))\delta_j^\Delta W+R(\alpha)\L^{\alpha,\Delta}(t_n^\Delta),
\end{align*}
where the third equality uses the fact that, by Lemma \ref{lem:projdecomp}, $R(\alpha)\xi^\alpha(\Xi_j^{\alpha})=\pi^\alpha(\Xi_j^{\alpha,\Delta})-\Xi_j^{\alpha,\Delta}$ for each $j\in\N$. Since $\Z^{\alpha,\Delta}$, $\L^{\alpha,\Delta}$, $\rho^\Delta$ and $W^\Delta$ are constant on intervals of the form $[t_n^\Delta,t_{n+1}^\Delta)$ for $n\in\N_0$, it follows from \eqref{eq:ZLDelta0}, \eqref{eq:rhoDeltatnDelta}, \eqref{eq:WDeltatnDelta} and the last display that, for all $t\geq0$,
	\be\label{eq:Zh}\Z^{\alpha,\Delta}(t)=x_0(\alpha)+\int_0^tb(\alpha,\Z^{\alpha,\Delta}(s-))d\rho^\Delta(s)+\int_0^t\sigma(\alpha,\Z^{\alpha,\Delta}(s-))d\bm^\Delta(s)+R(\alpha)\L^{\alpha,\Delta}(t).\ee
In addition, by \eqref{eq:ZLDelta0}--\eqref{eq:Znalphah} and Lemma \ref{lem:projdecomp}, we see that a.s.\ $\L^{\alpha,\Delta}(0)=0$ and for each $i\in\allN$, the $i$th component of $\L^{\alpha,\Delta}$ is nondecreasing and can only increase when $\Z^{\alpha,\Delta}$ lies in $F_i$. 

\begin{remark}\label{rmk:XDelta}
Define the $J$-dimensional piecewise constant RCLL $\{\F_t^\Delta\}$-adapted process $\X^{\alpha,\Delta}=\{\X^{\alpha,\Delta}(t),t\geq0\}$ by
	\be\label{eq:Xh}\X^{\alpha,\Delta}(t)\doteq x_0(\alpha)+\int_0^tb(\alpha,\Z^{\alpha,\Delta}(s-))d\rho^\Delta(s)+\int_0^t\sigma(\alpha,\Z^{\alpha,\Delta}(s-))d\bm^\Delta(s),\qquad t\geq0.\ee
Then by the properties of $(\Z^{\alpha,\Delta},\L^{\alpha,\Delta})$ stated above and Definition \ref{def:sp}, we see that $(\Z^{\alpha,\Delta},\L^{\alpha,\Delta})$ is a solution to the SP $\{(d_i(\alpha),n_i,c_i),i\in\allN\}$ for $\X^{\alpha,\Delta}$.
\end{remark}

We now prove the convergence of the Euler scheme for the reflected diffusion. The argument is analogous to the one given in \cite{Slominski2001} for reflected diffusions in convex polyhedral domains with normal reflection. Given a function $f:\R_+\mapsto\R^J$ and $0<\Delta<t$, define
\be\label{eq:osc}\text{Osc}\lb f,\Delta,t\rb\doteq\sup\lcb|f(u)-f(s)|:s,u\in[0,t],|u-s|\leq\Delta\rcb.\ee

\begin{lem}[{\cite[Lemma A.4]{Slominski2001}}]
\label{lem:oscbmzero}
For each $\alpha\in U$, $p\geq2$ and $t<\infty$,
	$$\E\lsb|\osc\lb\bm,\Delta,t\rb|^p\rsb=O\lb\lb \Delta\log \frac{1}{\Delta}\rb^\frac{p}{2}\rb\qquad\text{as }\Delta\downarrow0.$$
\end{lem}

\begin{prop}\label{prop:Xhconverge}
	For each $\alpha\in U$, $p\geq2$ and $t<\infty$,
	$$\E\lsb\sup_{0\leq s\leq t}|\X^{\alpha,\Delta}(s)-\X^\alpha(s)|^p\rsb=O\lb\lb \Delta\log \frac{1}{\Delta}\rb^\frac{p}{2}\rb\qquad\text{as }\Delta\downarrow0.$$
\end{prop}

\begin{proof}
Fix $\alpha\in U$ and $p\geq2$. The proof is analogous to the proof of \cite[Theorem 3.2]{Slominski2001}, which assumes normal reflection along the boundary. In particular, following \cite[equations (3.3)--(3.4)]{Slominski2001}, there exists a constant $C<\infty$ such that, for all $t\geq0$,
\begin{align}\label{eq:XDeltaXbound}
	\E\left[\sup_{0\leq u\leq t}|\X^{\alpha,\Delta}(u)-\X^\alpha(u)|^p\right]&\leq C\left(\Delta\log\frac{1}{\Delta}\right)^{\frac{p}{2}}+C\int_0^t\E\left[\sup_{0\leq u\leq s}|\Z^{\alpha,\Delta}(u)-\Z^\alpha(u)|^p\right]ds.
\end{align}
Due to the facts that $\Z^{\alpha,\Delta}=\sm^\alpha(\X^{\alpha,\Delta})$ and $\Z^\alpha=\sm^\alpha(\X^\alpha)$ by Remarks \ref{rmk:X} and \ref{rmk:XDelta}, and the Lipschitz continuity of the SM shown in Proposition \ref{prop:sp}(i), we have, for $t\geq0$,
\begin{align*}
	\E\left[\sup_{0\leq u\leq t}|\Z^{\alpha,\Delta}(u)-\Z^\alpha(u)|^p\right]&\leq \lip_\sm^p\E\left[\sup_{0\leq u\leq t}|\X^{\alpha,\Delta}(u)-\X^\alpha(u)|^p\right].
\end{align*}
Substituting the last inequality into the integrand on the right hand side of \eqref{eq:XDeltaXbound} and applying Gronwall's inequality yields
	$$\E\left[\sup_{0\leq u\leq t}|\X^{\alpha,\Delta}(u)-\X^\alpha(u)|^p\right]\leq C\left(\Delta\log\frac{1}{\Delta}\right)^{\frac{p}{2}}\exp\left(C\lip_\sm^pt\right),\qquad t\geq0,$$
which completes the proof of the lemma.
\end{proof}

\begin{proof}[Proof of Theorem \ref{thm:EulerRD}]
Recall that $\alpha\in U$ is given. Fix $p\geq2$. Recall that $(\Z^\alpha,\L^\alpha)$ is a solution to the SP $\{(d_i(\alpha),n_i,c_i),i\in\allN\}$ for $\X^\alpha$ by Remark \ref{rmk:X} and $(\Z^{\alpha,\Delta},\L^{\alpha,\Delta})$ is a solution to the SP $\{(d_i(\alpha),n_i,c_i),i\in\allN)\}$ for $\X^{\alpha,\Delta}$ by Remark \ref{rmk:XDelta}. The theorem then it follows from the Lipschitz continuity of the SM (Proposition \ref{prop:sp}(i)) and Proposition \ref{prop:Xhconverge}.
\end{proof}

\section{Convergence of the Euler scheme for the derivative process}\label{sec:approx}

For $\Delta>0$ recall the sequence $\{t_n^\Delta\}_{n\in\N_0}$ in $\R_+$ defined in \eqref{eq:tnDelta}, the RCLL step function $\rho^\Delta$ defined in \eqref{eq:rhoDelta}, the discrete filtration $\{\F_t^\Delta\}$ defined in \eqref{eq:FtDelta}, the piecewise constant RCLL $\{\F_t^\Delta\}$-adapted process $W^\Delta$ defined in \eqref{eq:WDelta}, the piecewise constant RCLL $\{\F_t^\Delta\}$-adapted processes $(\Z^{\alpha,\Delta},\L^{\alpha,\Delta})$ taking values in $G\times\R_+^N$ defined in \eqref{eq:ZLDelta0}--\eqref{eq:XinalphaDelta}, and the piecewise constant RCLL $\{\F_t^\Delta\}$-adapted process $\phib^{\alpha,\Delta}$ taking values in $\R^{J\times M}$ defined in \eqref{eq:phibDelta0}--\eqref{eq:psibalphah}.

\subsection{Relation between the discretized derivative process and the derivative map}

Recall the definition of $\hyper_x$ given in \eqref{eq:Hx} for $x\in G$. Suppose $x_0(\alpha)\in F_i$ for some $i\in\allN$. Then, for all $\beta\in\R^M$ and $\ve > 0$ sufficiently small so that $\alpha+\ve\beta\in U$, we have
	$$\ve^{-1}\ip{x_0(\alpha+\ve\beta)-x_0(\alpha),n_i}=\ve^{-1}\ip{x_0(\alpha+\ve\beta),n_i}\geq0,$$
where the final inequality is due to the fact that $x_0(\cdot)$ takes values in $G$. Since the last display holds for all $\beta\in\R^M$, by considering limits as $\ve\to0$ from the left and right, this implies that the column vectors of $x_0'(\alpha)$ take values in $\{y\in\R^J:\ip{y,n_i}=0\}$. In particular, it follows that given $\alpha\in U$, the column vectors of $x_0'(\alpha)$ take values in $\hyper_{x_0(\alpha)}$. Thus, by \eqref{eq:tnDelta}, \eqref{eq:phibDelta0} and \eqref{eq:ZLDelta0},
	\be\label{eq:phibt0Deltahyper}\phib_m^{\alpha,\Delta}(t_0^\Delta)=[x_0'(\alpha)]_m\in\hyper_{x_0(\alpha)}=\hyper_{\Z^{\alpha,\Delta}(t_0^\Delta)},\qquad m=1,\dots,M.\ee
Here $[x_0'(\alpha)]_m$ denotes the $m$th column vector of $x_0'(\alpha)$. By \eqref{eq:phibDelta0}--\eqref{eq:mathcalXnDeltaalpha} and \eqref{eq:ZLDelta0}, we have, for each $n\in\N$,
	\be\label{eq:phibtnDeltahyper}\phib_m^{\alpha,\Delta}(t_n^\Delta)\in\hyper_{\Z^{\alpha,\Delta}(t_n)},\qquad m=1,\dots,M,\ee 
and
\begin{align}\label{eq:phibalphaDeltatnDelta}
	\phib^{\alpha,\Delta}(t_n^\Delta)&=\phib^{\alpha,\Delta}(t_0^\Delta)+\sum_{j=1}^n\left(\phib^{\alpha,\Delta}(t_j^\Delta)-\phib^{\alpha,\Delta}(t_{j-1}^\Delta)\right)\\ \notag
	&=x_0'(\alpha)+\sum_{j=1}^n\left(\mathcal{X}_j^{\alpha,\Delta}-\phib^{\alpha,\Delta}(t_{j-1}^\Delta)+\proj_{\Z^{\alpha,\Delta}(t_j^\Delta)}^\alpha\mathcal{X}_j^{\alpha,\Delta}-\mathcal{X}_j^{\alpha,\Delta}\right)\\ \notag
	&=x_0'(\alpha)+\sum_{j=1}^nb_\alpha(\alpha,\Z^{\alpha,\Delta}(t_{j-1}^\Delta))\Delta+\sum_{j=1}^nb_x(\alpha,\Z^{\alpha,\Delta}(t_{j-1}^\Delta))\phib^{\alpha,\Delta}(t_{j-1}^\Delta)\Delta\\ \notag
	&\qquad+\sum_{j=1}^n\sigma_\alpha(\alpha,\Z^{\alpha,\Delta}(t_{j-1}^\Delta))\delta_j^\Delta\bm+\sum_{j=1}^n\sigma_x(\alpha,\Z^{\alpha,\Delta}(t_{j-1}^\Delta))\phib^{\alpha,\Delta}(t_{j-1}^\Delta)\delta_j^\Delta\bm\\ \notag
	&\qquad+R'(\alpha)\sum_{j=1}^n(\L^{\alpha,\Delta}(t_j^\Delta)-\L^{\alpha,\Delta}(t_{j-1}^\Delta))+\sum_{j=1}^n\left(\proj_{\Z^{\alpha,\Delta}(t_j^\Delta)}^\alpha\mathcal{X}_j^{\alpha,\Delta}-\mathcal{X}_j^{\alpha,\Delta}\right)\\ \notag
	&=x_0'(\alpha)+\sum_{j=1}^nb_\alpha(\alpha,\Z^{\alpha,\Delta}(t_{j-1}^\Delta))\Delta+\sum_{j=1}^nb_x(\alpha,\Z^{\alpha,\Delta}(t_{j-1}^\Delta))\phib^{\alpha,\Delta}(t_{j-1}^\Delta)\Delta\\ \notag
	&\qquad+\sum_{j=1}^n\sigma_\alpha(\alpha,\Z^{\alpha,\Delta}(t_{j-1}^\Delta))\delta_j^\Delta\bm+\sum_{j=1}^n\sigma_x(\alpha,\Z^{\alpha,\Delta}(t_{j-1}^\Delta))\phib^{\alpha,\Delta}(t_{j-1}^\Delta)\delta_j^\Delta\bm\\ \notag
	&\qquad+R'(\alpha)\L^{\alpha,\Delta}(t_n^\Delta)+\etab^{\alpha,\Delta}(t_n^\Delta),
\end{align}
where $\etab^{\alpha,\Delta}=\{\etab^{\alpha,\Delta}(t),t\geq0\}$ is the piecewise constant RCLL $\{\F_t^\Delta\}$-adapted process taking values in $\R^{J\times M}$ defined by $\etab^{\alpha,\Delta}(0)=0$ and, for each $n\in\N$, $\etab^{\alpha,\Delta}(t)\doteq\etab^{\alpha,\Delta}(t_{n-1}^\Delta)$ for all $t\in[t_{n-1}^\Delta,t_n^\Delta)$ and 
	\be\label{eq:etabalphaDeltadef}\etab^{\alpha,\Delta}(t_n^\Delta)\doteq\sum_{j=1}^n\left(\proj_{\Z^{\alpha,\Delta}(t_j^\Delta)}^\alpha\mathcal{X}_j^{\alpha,\Delta}-\mathcal{X}_j^{\alpha,\Delta}\right).\ee
Since $\phib^{\alpha,\Delta}$, $\Z^{\alpha,\Delta}$, $\L^{\alpha,\Delta}$, $\rho^\Delta$, $W^\Delta$ and $\etab^{\alpha,\Delta}$ are constant on intervals fo the form $[t_n^\Delta,t_{n+1}^\Delta)$ for $n\in\N_0$, it follows from \eqref{eq:phibt0Deltahyper}, \eqref{eq:phibtnDeltahyper}, \eqref{eq:phibalphaDeltatnDelta}, \eqref{eq:phibDelta0}, \eqref{eq:rhoDelta} and \eqref{eq:WDelta} that, for all $t\geq0$,
	\be\label{eq:phibDeltahyper}\phib_m^{\alpha,\Delta}(t)\in\hyper_{\Z^{\alpha,\Delta}(t)},\qquad m=1,\dots,M,\ee
and
\begin{align}\label{eq:psibn}
	\phib^{\alpha,\Delta}(t)&=x_0'(\alpha)+\int_0^tb_\alpha(\alpha,\Z^{\alpha,\Delta}(s-))d\rho^\Delta(s)+\int_0^tb_x(\alpha,\Z^{\alpha,\Delta}(s-))\phib^{\alpha,\Delta}(s-)d\rho^\Delta(s)\\ \notag
	&\qquad+\int_0^t\sigma_\alpha(\alpha,\Z^{\alpha,\Delta}(s-))d\bm^\Delta(s)+\int_0^t\sigma_x(\alpha,\Z^{\alpha,\Delta}(s-))\phib^{\alpha,\Delta}(s-)d\bm^\Delta(s)\\ \notag
	&\qquad+R'(\alpha)\L^{\alpha,\Delta}(t)+\etab^{\alpha,\Delta}(t).
\end{align}
By \eqref{eq:etabalphaDeltadef} and Lemma \ref{lem:projx}, we see that for each $m=1,\dots,M$ and $n\in\N$,
	$$\etab_m^{\alpha,\Delta}(t_n^\Delta)-\etab_m^{\alpha,\Delta}(t_n^\Delta-)=\proj_{\Z^{\alpha,\Delta}(t_n^\Delta)}^\alpha[\mathcal{X}_n^{\alpha,\Delta}]_m-[\mathcal{X}_n^{\alpha,\Delta}]_m\in\spaan\left[d(\alpha,\Z^{\alpha,\Delta}(t_n^\Delta))\right].$$
Here $[\mathcal{X}_n^{\alpha,\Delta}]_m$ denotes the $m$th column vector of the $J\times M$ matrix $\mathcal{X}_n^{\alpha,\Delta}$. Since $\etab^{\alpha,\Delta}$ and $\Z^{\alpha,\Delta}$ are piecewise constant, it follows that, for all $m=1,\dots,M$ and $0\leq s<t<\infty$
	\be\label{eq:etabDeltats}\etab_m^{\alpha,\Delta}(t)-\etab_m^{\alpha,\Delta}(s)\in\spaan\left[\cup_{u\in(s,t]}d(\alpha,\Z^{\alpha,\Delta}(u))\right]\ee
and for all $m=1,\dots,M$ and $t>0$,
	\be\label{eq:etabDeltattminus}\etab_m^{\alpha,\Delta}(t)-\etab_m^{\alpha,\Delta}(t-)\in\spaan\left[d(\alpha,\Z^{\alpha,\Delta}(t))\right].\ee

\begin{remark}
Define the piecewise constant RCLL $\{\F_t^\Delta\}$-adapted process $\psib^{\alpha,\Delta}=\{\psib^{\alpha,\Delta}(t),t\geq0\}$ taking values in $\R^{J\times M}$, for $t\geq0$, by
\begin{align}\label{eq:psibn}
	\psib^{\alpha,\Delta}(t)&=x_0'(\alpha)+\int_0^tb_\alpha(\alpha,\Z^{\alpha,\Delta}(s-))d\rho^\Delta(s)+\int_0^tb_x(\alpha,\Z^{\alpha,\Delta}(s-))\phib^{\alpha,\Delta}(s-)d\rho^\Delta(s)\\ \notag
	&\qquad+\int_0^t\sigma_\alpha(\alpha,\Z^{\alpha,\Delta}(s-))d\bm^\Delta(s)+\int_0^t\sigma_x(\alpha,\Z^{\alpha,\Delta}(s-))\phib^{\alpha,\Delta}(s-)d\bm^\Delta(s)\\ \notag
	&\qquad+R'(\alpha)\L^{\alpha,\Delta}(t).
\end{align}
It follows from \eqref{eq:phibDeltahyper}, \eqref{eq:phibalphaDeltatnDelta}, \eqref{eq:psibn}, \eqref{eq:etabDeltats}, \eqref{eq:etabDeltattminus} that for each $m=1,\dots,M$, a.s.\ $(\phib_m^{\alpha,\Delta},\etab_m^{\alpha,\Delta})$ is a solution to the derivative problem along $\Z^{\alpha,\Delta}$ for $\psib_m^{\alpha,\Delta}$. In other words, a.s.\ 
	\be\label{eq:phibmDeltaDM}\phib_m^{\alpha,\Delta}=\dm_{\Z^{\alpha,\Delta}}(\psib_m^{\alpha,\Delta}),\qquad m=1,\dots,M.\ee
\end{remark}

\subsection{Proof of Theorem \ref{thm:approx}}

In preparation for proving Theorem \ref{thm:approx}, we first state some preliminary results. 

\begin{lem}\label{lem:supLbound}
Given $\alpha\in U$, $p\geq2$ and $t<\infty$ there exists $C_1^{(\alpha,p,t)}<\infty$ such that for all $\Delta>0$,
	\be\label{eq:LDeltax0alpha}\E\lsb\sup_{0\leq s\leq t}|\L^{\alpha,\Delta}(s)|^p\rsb\leq C_1^{(\alpha,p,t)}.\ee
\end{lem}

\begin{proof}
Fix $\alpha\in U$, $p\geq2$ and $t<\infty$. The fact that
\begin{align}\label{eq:XDeltax0alpha}
	\E\lsb\sup_{0\leq s\leq t}|\X^{\alpha,\Delta}(s)-x_0(\alpha)|^p\rsb<\infty
\end{align}
follows from a standard argument using \eqref{eq:Xh}, H\"older’s inequality, the BDG inequalities, Tonelli's theorem, Assumption \ref{ass:coefficients}, the facts thats $\Z^{\alpha,\Delta}=\sm^\alpha(\X^{\alpha,\Delta})$ by Remark \ref{rmk:XDelta} and $\z=\sm^\alpha(\x)$ where $\z(\cdot)=\x(\cdot)\equiv x_0(\alpha)$, the Lipschitz continuity of the SM (Proposition \ref{prop:sp}(i)) and Gronwall's inequality. Then, using the fact that $(\Z^{\alpha,\Delta},\L^{\alpha,\Delta})$ is the solution to the SP for $\X^{\alpha,\Delta}$ by Remark \ref{rmk:XDelta}, the fact that $(\z,0)$ is the solution to the SP for $\x$, \eqref{eq:XDeltax0alpha} and the Lipschitz continuity of the SM, we obtain \eqref{eq:LDeltax0alpha}.
\end{proof}

\begin{lem}\label{lem:suphphibalphahbound}
For each $\alpha\in U$, $p\geq2$ and $t<\infty$, there exists $C_2^{(\alpha,p,t)}<\infty$ such that for all $\Delta>0$,
	\be\label{eq:supphibpbound}\E\lsb\sup_{0\leq s\leq t}\norm{\phib^{\alpha,\Delta}(s)}^p\rsb\leq C_2^{(\alpha,p,t)}.\ee
\end{lem}

\begin{proof}
Fix $\alpha\in U$, $p\geq2$ and $t<\infty$. Let $\Delta>0$. For brevity, let $\lip\doteq\max(\lip_1,\lip_\dm(\alpha))<\infty$. By \eqref{eq:psibn}, H\"older's inequality, the BDG inequalities, Tonelli's theorem and bounds on the coefficients stated in Assumption \ref{ass:coefficients}, 
\begin{align*}
	\E\lsb\sup_{0\leq u\leq t}\norm{\psib^{\alpha,\Delta}(u)}^p\rsb&\leq4^{p-1}\norm{x_0'(\alpha)}^p+4^{p-1}\E\left[\sup_{0\leq s\leq t}|R'(\alpha)\L^{\alpha,\Delta}(s)|^p\right]\\
	&\qquad+8^{p-1}(t^{p-1}+C_p)\lip^p\int_0^t\lb1+\E\lsb\sup_{0\leq u\leq s}\norm{\phib^{\alpha,\Delta}(u)}^p\rsb\rb ds.
\end{align*}
By \eqref{eq:phibmDeltaDM}, the Lipschitz continuity of the derivative map $\dm_{\Z^{\alpha,\Delta}}$ (Proposition \ref{prop:dmlip}), and Lemma \ref{lem:supLbound}, we have
\begin{align*}
	\E\lsb\sup_{0\leq u\leq t}\norm{\phib^{\alpha,\Delta}(u)}^p\rsb&\leq4^{p-1}\lip^p\norm{x_0'(\alpha)}^p+4^{p-1}\norm{R'(\alpha)}^pC_1^{(\alpha,p,t)}\\
	&\qquad+8^{p-1}(t^{p-1}+C_p)\lip^{2p}\int_0^t\lb1+\E\lsb\sup_{0\leq u\leq s}\norm{\phib^{\alpha,\Delta}(u)}^p\rsb\rb ds.
\end{align*}
An application of Gronwall's inequality yields \eqref{eq:supphibpbound} with 
	$$C_2^{(\alpha,p,t)}\doteq\lb1+4^{p-1}\lip^p\norm{x_0'(\alpha)}^p+4^{p-1}\norm{R'(\alpha)}^pC_1^{(\alpha,p,t)}\rb e^{8^{p-1}(t^{p-1}+C_p)t}.$$
\end{proof}

\begin{prop}\label{prop:psibnconvergence}
	For each $\alpha\in U$, $p\geq2$ and $t<\infty$,
	\be\lim_{h\downarrow0}\E\lsb\sup_{0\leq s\leq t}\norm{\psib^{\alpha,\Delta}(s)-\psib^\alpha(s)}^p\rsb=0.\ee
\end{prop}

We first show how Theorem \ref{thm:approx} and Corollary \ref{cor:approx} can be deduced from 
Proposition \ref{prop:psibnconvergence}.  
The proof of Proposition \ref{prop:psibnconvergence} is given in Section \ref{sec:psibnconvergence}, with 
preliminary results established in Section \ref{subs-useful}. 

\begin{proof}[Proof of Theorem \ref{thm:approx}]
Let $m=1,\dots,M$. By Theorem \ref{thm:EulerRD} and Proposition \ref{prop:psibnconvergence}, a.s.\ $(\Z^{\alpha,\Delta},\psib_m^{\alpha,\Delta})$ converges to $(\Z^\alpha,\psib_m^\alpha)$ uniformly on compact time intervals as $\Delta\downarrow0$. By Proposition \ref{prop:jitter}, a.s.\ $(\Z^\alpha,\Y^\alpha)$ satisfies the boundary jitter property. Then by the facts that $\phib_m^\alpha=\dm_{\Z^\alpha}(\psib_m^\alpha)$ and $\phib_m^{\alpha,\Delta}=\dm_{\Z^{\alpha,\Delta}}(\psib_m^{\alpha,\Delta})$ for $\Delta>0$, and the continuity of the derivative map shown in Theorem \ref{thm:dmcontinuous}, a.s.\ $\phib_m^{\alpha,\Delta}$ converges to $\phib_m^\alpha$ in the $J_1$-topology as $\Delta\downarrow0$. Since this holds for each $m=1,\dots,M$, the proof is complete.
\end{proof}

\begin{proof}[Proof of Corollary \ref{cor:approx}]
	By part (iii) of Theorem \ref{thm:DF}, Proposition \ref{prop:jitter}, condition 2' of the boundary jitter property and \cite[Lemma 4.13]{Lipshutz2017}, a.s.\ $\phib^\alpha$ is continuous at $t$ and at almost every $s\in(0,t)$. Thus, by Theorem \ref{thm:approx}, a.s.\ $\phib^{\alpha,\Delta}$ converges to $\phib^\alpha$ at $t$ and at almost every $s\in(0,t)$. Along with the uniform integrability of $\phib^{\alpha,\Delta}$ stated in Theorem \ref{thm:approx}, this implies that \eqref{eq:F'alphaapprox} holds.
\end{proof}

The remainder of this section is devoted to the proof Proposition \ref{prop:psibnconvergence}.

\subsection{Some useful lemmas}
\label{subs-useful}

Recall the derivative processes $\phib^\alpha$ introduced in Definition \ref{def:de} and the associated process $\psib^\alpha$ introduced in Remark \ref{rmk:psib}. We first recall a basic estimate.

\begin{lem}[{\cite[Lemma 5.7]{Lipshutz2017}}]
\label{lem:suphphibalphabound}
For each $\alpha\in U$, $p\geq2$ and $t<\infty$, 
	$$\E\lsb\sup_{0\leq s\leq t}\norm{\phib^\alpha(s)}^p\rsb<\infty\qquad\text{and}\qquad\E\lsb\sup_{0\leq s\leq t}\norm{\psib^\alpha(s)}^p\rsb<\infty.$$
\end{lem}

We now prove an estimate that relates the derivative map along the reflected diffusion with the derivative map along the Euler discretization of the reflected diffusion.

\begin{lem}\label{lem:dmZalphaetabalpha}
	For each $\alpha\in U$, $p\geq2$, $m=1,\dots,M$ and $t<\infty$,
	\be\label{eq:dmZalphaetabalpha}\lim_{\Delta\downarrow0}\E\lsb\int_0^t|\dm_{\Z^{\alpha,\Delta}}(\psib_m^\alpha)(s)-\dm_{\Z^\alpha}(\psib_m^\alpha)(s)|^pds\rsb=0.\ee
\end{lem}

\begin{proof}
Fix $\alpha\in U$, $p\geq2$, $m=1,\dots,M$ and $t<\infty$. Recall the constraining process $\Y^\alpha$ introduced in Remark \ref{rmk:Y} and the process $\X^\alpha$ defined in \eqref{eq:X}. By Remark \ref{rmk:X} and Proposition \ref{prop:jitter}, a.s.\
\begin{itemize}
	\item[(a)] $(\Z^\alpha,\Y^\alpha)$ is the solution to the ESP for $\X^\alpha$, and
	\item[(b)] $(\Z^\alpha,\Y^\alpha)$ satisfies the boundary jitter property.
\end{itemize}
For $\Delta>0$, recall the process $\X^{\alpha,\Delta}$ defined in \eqref{eq:Xh}. Then by Remark \ref{rmk:XDelta} and Proposition \ref{prop:Xhconverge}, a.s. 
\begin{itemize}
	\item[(c)] $\Z^{\alpha,\Delta}=\esm(\X^{\alpha,\Delta})$, and
	\item[(d)] $\X^{\alpha,\Delta}$ converges to $\X^\alpha$ in $\dr(\R^J)$ as $\Delta\downarrow0$.
\end{itemize}
Since (a)--(d) hold a.s., it follows from Theorem \ref{thm:dmcontinuous} that a.s.\ $\dm_{\Z^{\alpha,\Delta}}(\psib_m^\alpha)$ converges to $\dm_{\Z^\alpha}(\psib_m^\alpha)$ in $\dr(\R^J)$ as $\Delta\downarrow0$. Thus, a.s.
	\be\label{eq:dmZhdmZ}\dm_{\Z^{\alpha,\Delta}}(\psib_m^\alpha)(s)\to\dm_{\Z^\alpha}(\psib_m^\alpha)(s)\text{ as $\Delta \downarrow0$ for almost every $s\in[0,t]$}.\ee
By the Lipschitz continuity of the derivative map (Proposition \ref{prop:dmlip}),
	\begin{align*}
	\sup_{0\leq s\leq t}|\dm_{\Z^{\alpha,\Delta}}(\psib_m^\alpha)(s)|^p+\sup_{0\leq s\leq t}|\dm_{\Z^{\alpha}}(\psib_m^\alpha)(s)|^p&\leq2(\lip_\dm(\alpha))^p\sup_{0\leq s\leq t}|\psib_m^\alpha(s)|^p.
	\end{align*}
Therefore, by Lemma \ref{lem:suphphibalphabound}, the dominated convergence theorem and \eqref{eq:dmZhdmZ}, \eqref{eq:dmZalphaetabalpha} holds.
\end{proof}

\subsection{Proof of Proposition \ref{prop:psibnconvergence}}\label{sec:psibnconvergence}

Fix $\alpha\in U$. For brevity, throughout this section we let 
$$\lip\doteq\max\lb\lip_1,\lip_2,\lip_\sm(\alpha),\lip_\dm(\alpha)\rb<\infty,$$
where $\lip_1$ and $\lip_2$ denote the constants in Assumption \ref{ass:coefficients}, $\lip_\sm(\alpha)$ denotes the constant in Proposition \ref{prop:sp} and $\lip_\dm(\alpha)$ denotes the constant in Proposition \ref{prop:dmlip}.

\begin{proof}[Proof of Proposition \ref{prop:psibnconvergence}]
Let $p\geq2$ and $t<\infty$. By \eqref{eq:psibn} and \eqref{eq:psib}, for each $s\in[0,t]$,
\begin{align}\label{eq:psibnpsib}
	&\norm{\psib^{\alpha,\Delta}(s)-\psib^\alpha(s)}\\ \label{eq:term1}
	&\qquad\leq\norm{\int_0^sb_\alpha(\alpha,\Z^{\alpha,\Delta}(u-))d\rho^\Delta(u)-\int_0^sb_\alpha(\alpha,\Z^{\alpha,\Delta}(u))du}\\ \label{eq:term2}
	&\qquad+\norm{\int_0^s\lcb b_\alpha(\alpha,\Z^{\alpha,\Delta}(u))-b_\alpha(\alpha,\Z^\alpha(u))\rcb du}\\ \label{eq:term3}
	&\qquad+\norm{\int_0^sb_x(\alpha,\Z^{\alpha,\Delta}(u-))\phib^{\alpha,\Delta}(u-)d\rho^\Delta(u)-\int_0^sb_x(\alpha,\Z^{\alpha,\Delta}(u))\phib^{\alpha,\Delta}(u)du}\\ \label{eq:term4}
	&\qquad+\norm{\int_0^s\lcb b_x(\alpha,\Z^{\alpha,\Delta}(u))-b_x(\alpha,\Z^\alpha(u))\rcb\phib^{\alpha,\Delta}(u)du}\\ \label{eq:term5}
	&\qquad+\norm{\int_0^sb_x(\alpha,\Z^\alpha(u))(\phib^{\alpha,\Delta}(u)-\phib^\alpha(u)) du}\\ \label{eq:term6}
	&\qquad+\norm{\int_0^s\sigma_\alpha(\alpha,\Z^{\alpha,\Delta}(u-))d\bm^\Delta(u)-\int_0^s\sigma_\alpha(\alpha,\Z^{\alpha,\Delta}(u))d\bm(u)}\\ \label{eq:term7}
	&\qquad+\norm{\int_0^s\lcb \sigma_\alpha(\alpha,\Z^{\alpha,\Delta}(u))-\sigma_\alpha(\alpha,\Z^\alpha(u))\rcb d\bm(u)}\\ \label{eq:term8}
	&\qquad+\norm{\int_0^s\sigma_x(\alpha,\Z^{\alpha,\Delta}(u-))\phib^{\alpha,\Delta}(u-)d\bm^\Delta(u)-\int_0^s\sigma_x(\alpha,\Z^{\alpha,\Delta}(u))\phib^{\alpha,\Delta}(u)d\bm(u)}\\ \label{eq:term9}
	&\qquad+\norm{\int_0^s\lcb \sigma_x(\alpha,\Z^{\alpha,\Delta}(u))-\sigma_x(\alpha,\Z^\alpha(u))\rcb\phib^{\alpha,\Delta}(u)d\bm(u)}\\ \label{eq:term10}
	&\qquad+\norm{\int_0^s \sigma_x(\alpha,\Z^\alpha(u))\lb\phib^{\alpha,\Delta}(u)-\phib^\alpha(u)\rb d\bm(u)}\\ \label{eq:term11}
	&\qquad+\norm{R'(\alpha)(\L^{\alpha,\Delta}(s)-\L^\alpha(s))}.
\end{align}
We treat each of the eleven terms on the right-hand side of the last display separately.
	
\eqref{eq:term1}: Since $\Z^{\alpha,\Delta}$ and $\rho^\Delta$ are constant on intervals of form $[t_{n-1}^\Delta,t_n^\Delta)$ for $n\in\N$, it follows that for $s\in[0,t]$,
	$$\int_0^sb_\alpha(\alpha,\Z^{\alpha,\Delta}(u-))d\rho^\Delta(u)-\int_0^sb_\alpha(\alpha,\Z^{\alpha,\Delta}(u))du=-b_\alpha(\alpha,\Z^{\alpha,\Delta}(t_{\lfloor s/\Delta\rfloor}^\Delta))(s-t_{\lfloor s/\Delta\rfloor}^\Delta).$$
Therefore, by Assumption \ref{ass:coefficients} and the fact that $|s-t_{\lfloor s/\Delta\rfloor}^\Delta|\leq\Delta$,
\begin{align}\label{eq:term1a}
	\E\lsb\sup_{0\leq s\leq t}\left|\int_0^sb_\alpha(\alpha,\Z^{\alpha,\Delta}(u-))d\rho^\Delta(u)-\int_0^sb_\alpha(\alpha,\Z^{\alpha,\Delta}(u))du\right|^p\rsb&\leq \Delta^p\lip^p.
\end{align}
	
\eqref{eq:term2}: By H\"older's inequality and Assumption \ref{ass:coefficients}, with the associated constant $\gamma$, 
\begin{align}\label{eq:term2a}
	&\E\lsb\sup_{0\leq s\leq t}\norm{\int_0^s\lcb b_\alpha(\alpha,\Z^{\alpha,\Delta}(u))-b_\alpha(\alpha,\Z^\alpha(u))\rcb du}^p\rsb\\ \notag
	&\qquad\leq t^p\lip^p\E\lsb\sup_{0\leq s\leq t}|\Z^{\alpha,\Delta}(s)-\Z^\alpha(s)|^{\gamma p}\rsb.
\end{align}
	
\eqref{eq:term3}: Since $\Z^{\alpha,\Delta}$ and $\phib^{\alpha,\Delta}$ are constant on intervals of form $[t_{n-1}^\Delta,t_n^\Delta)$ for $k\in\N$, it follows that for $s\in[0,t]$,
\begin{align*}
	&\int_0^sb_x(\alpha,\Z^{\alpha,\Delta}(u-))\phib^{\alpha,\Delta}(u-)d\rho^\Delta(u)-\int_0^sb_x(\alpha,\Z^{\alpha,\Delta}(u))\phib^{\alpha,\Delta}(u)du\\
	&\qquad=\int_{t_{\lfloor s/\Delta\rfloor}^\Delta}^sb_x(\alpha,\Z^{\alpha,\Delta}(u))\phib^{\alpha,\Delta}(u)du.
\end{align*}
Therefore, using H\"older's inequality and Assumption \ref{ass:coefficients},
\begin{align}\label{eq:term3a}
	&\E\lsb\sup_{0\leq s\leq t}\norm{\int_0^sb_x(\alpha,\Z^{\alpha,\Delta}(u-))\phib^{\alpha,\Delta}(u-)d\rho^\Delta(u)-\int_0^sb_x(\alpha,\Z^{\alpha,\Delta}(u))\phib^{\alpha,\Delta}(u)du}^p\rsb\\ \notag
	&\qquad\leq\Delta^p\lip^p\E\lsb\sup_{0\leq s\leq t}\norm{\phib^{\alpha,\Delta}(s)}^p\rsb.
\end{align}
	
\eqref{eq:term4}: By H\"older's inequality, Assumption \ref{ass:coefficients} and the Cauchy-Schwarz inequality,
\begin{align}\label{eq:term4a}
	&\E\lsb\sup_{0\leq s\leq t}\norm{\int_0^s\lcb b_x(\alpha,\Z^{\alpha,\Delta}(u))-b_x(\alpha,\Z^\alpha(u))\rcb\phib^{\alpha,\Delta}(u)du}^p\rsb\\ \notag
	&\qquad\leq t^p\lip^p\lcb\E\lsb\sup_{0\leq s\leq t}|\Z^{\alpha,\Delta}(s)-\Z^\alpha(s)|^{2\gamma p}\rsb\E\lsb\sup_{0\leq s\leq t}\norm{\phib^{\alpha,\Delta}(s)}^{2 p}\rsb\rcb^{\frac{1}{2}}.
\end{align}
	
\eqref{eq:term5}: By H\"older's inequality, Tonelli's theorem and Assumption \ref{ass:coefficients},
\begin{align}\label{eq:term5a}
	&\E\lsb\sup_{0\leq s\leq t}\norm{\int_0^sb_x(\alpha,\Z^\alpha(u))(\phib^{\alpha,\Delta}(u)-\phib^\alpha(u)) du}^p\rsb\\ \notag
	&\qquad\leq t^{p-1}\lip^p\int_0^t\E\lsb\norm{\phib^{\alpha,\Delta}(s)-\phib^\alpha(s)}^p\rsb ds.
\end{align}
	
\eqref{eq:term6}: Since $\Z^{\alpha,\Delta}$ and $W^\Delta$ are constant on intervals of form $[t_{n-1}^\Delta,t_n^\Delta)$ for $n\in\N$, it follows that for $s\in[0,t]$,
\begin{align*}
	&\int_0^s\sigma_\alpha(\alpha,\Z^{\alpha,\Delta}(u-))d\bm^\Delta(u)-\int_0^s\sigma_\alpha(\alpha,\Z^{\alpha,\Delta}(u))d\bm(u)\\
	&\qquad=-\sigma_\alpha(\alpha,\Z^{\alpha,\Delta}(t_{\lfloor s/\Delta\rfloor}^\Delta))(\bm(s)-\bm(t_{\lfloor s/\Delta\rfloor}^\Delta)).
\end{align*}
Therefore, by Assumption \ref{ass:coefficients}, \eqref{eq:osc} and the fact that $|s-t_{\lfloor s/\Delta\rfloor}^\Delta|\leq\Delta$,
\begin{align}\label{eq:term6a}
	&\E\lsb\sup_{0\leq s\leq t}\left|\int_0^s\sigma_\alpha(\alpha,\Z^{\alpha,\Delta}(u-))d\bm^\Delta(u)-\int_0^s\sigma_\alpha(\alpha,\Z^{\alpha,\Delta}(u))d\bm(u)\right|^p\rsb\\ \notag
	&\qquad\leq\lip^p\E\lsb|\osc(\bm,\Delta,t)|^p\rsb.
\end{align}
	
\eqref{eq:term7}: By the BDG inequalities and Assumption \ref{ass:coefficients},
\begin{align}\label{eq:term7a}
	&\E\lsb\sup_{0\leq s\leq t}\norm{\int_0^s\lcb \sigma_\alpha(\alpha,\Z^{\alpha,\Delta}(u))-\sigma_\alpha(\alpha,\Z^\alpha(u))\rcb d\bm(u)}^p\rsb\\ \notag
	&\qquad\leq C_p\lip^pt\E\lsb|\Z^{\alpha,\Delta}(s)-\Z^\alpha(s)|^{\gamma p}\rsb.
\end{align}
	
\eqref{eq:term8}: Since $\Z^{\alpha,\Delta}$, $\phib^{\alpha,\Delta}$ and $W^\Delta$ are constant on intervals of form $[t_{n-1}^\Delta,t_n^\Delta)$ for $k\in\N$, it follows that for $s\in[0,t]$,
\begin{align*}
	&\int_0^s\sigma_x(\alpha,\Z^{\alpha,\Delta}(u-))\phib^{\alpha,\Delta}(u-)d\bm^\Delta(u)-\int_0^s\sigma_x(\alpha,\Z^{\alpha,\Delta}(u))\phib^{\alpha,\Delta}(u-)d\bm(u)\\ \notag
	&\qquad=-\sigma_x(\alpha,\Z^{\alpha,\Delta}(t_{\lfloor
		s/\Delta\rfloor}^\Delta))\phib^{\alpha,\Delta}(t_{\lfloor s/\Delta\rfloor}^\Delta)(\bm(s)-\bm(t_{\lfloor s/\Delta\rfloor}^\Delta)).
\end{align*}
Therefore, by Assumption \ref{ass:coefficients}, the Cauchy-Schwarz inequality, \eqref{eq:osc} and the fact that $|s-t_{\lfloor s/\Delta\rfloor}^\Delta|\leq\Delta$,
\begin{align}\label{eq:term8a}
	&E\lsb\sup_{0\leq s\leq t}\norm{\int_0^s\sigma_x(\alpha,\Z^{\alpha,\Delta}(u-))\phib^{\alpha,\Delta}(u-)d\bm^\Delta(u)-\int_0^s\sigma_x(\alpha,\Z^{\alpha,\Delta}(u))\phib^{\alpha,\Delta}(u)d\bm(u)}^p\rsb\\ \notag
	&\qquad\leq \lip^p\lcb\E\lsb\sup_{0\leq s\leq t}\norm{\phib^{\alpha,\Delta}(s)}^{2p}\rsb\E\lsb|\osc(\bm,\Delta,t)|^{2p}\rsb\rcb^{\frac{1}{2}}.
\end{align}
	
	\eqref{eq:term9}: By the BDG inequalities, Assumption \ref{ass:coefficients} and the Cauchy-Schwarz inequality,
	\begin{align}\label{eq:term9a}
	&\E\lsb\sup_{0\leq s\leq t}\norm{\int_0^s\lcb \sigma_x(\alpha,\Z^{\alpha,\Delta}(u))-\sigma_x(\alpha,\Z^\alpha(u))\rcb\phib^{\alpha,\Delta}(u)d\bm(u)}^p\rsb\\ \notag
	&\qquad\leq C_p\lip^pt\lcb\E\lsb\sup_{0\leq s\leq t}|\Z^{\alpha,\Delta}(s)-\Z^\alpha(s)|^{2\gamma p}\rsb\E\lsb\sup_{0\leq s\leq t}\norm{\phib^{\alpha,\Delta}(s)}^{2 p}\rsb\rcb^{\frac{1}{2}}.
	\end{align}
	
	\eqref{eq:term10}: By the BDG inequalities, Tonelli's theorem and Assumption \ref{ass:coefficients},
	\begin{align}\label{eq:term10a}
	&\E\lsb\sup_{0\leq s\leq t}\norm{\int_0^s \sigma_x(\alpha,\Z^\alpha(u))\lb\phib^{\alpha,\Delta}(u)-\phib^\alpha(u)\rb d\bm(u)}^p\rsb\\ \notag
	&\qquad\leq C_p\lip^p\int_0^t\E\lsb\norm{\phib^{\alpha,\Delta}(s)-\phib^\alpha(s)}^p\rsb ds.
	\end{align}
	
	\eqref{eq:term11}: By Assumption \ref{ass:coefficients}, the facts that $(\Z^{\alpha,\Delta},\L^{\alpha,\Delta})$ solves the SP for $\X^{\alpha,\Delta}$ and $(\Z^\alpha,\L^\alpha)$ solves the SP for $\X^\alpha$ and the Lipschitz continuity of the SM (Proposition \ref{prop:sp}(i)),
	\begin{align}\label{eq:term11a}
	E\lsb\norm{\sup_{0\leq s\leq t}R'(\alpha)(\L^{\alpha,\Delta}(s)-\L^\alpha(s))}^p\rsb&\leq\lip^{2p}\E\lsb\sup_{0\leq s\leq t}|\X^{\alpha,\Delta}(s)-\X^\alpha(s)|^p\rsb.
	\end{align}
	
Before combining the terms, we observe that because $\phib_m^{\alpha,\Delta}=\dm_{\Z^{\alpha,\Delta}}(\psib_m^{\alpha,\Delta})$ and $\phib_m^\alpha=\dm_{\Z^\alpha}(\psib_m^\alpha)$ for each $m=1,\dots,M$, it follows from the Lipschitz continuity of the derivative map that
\begin{align*}
	&|\phib_m^{\alpha,\Delta}(s)-\phib_m^\alpha(s)|\\
	&\qquad\leq|\dm_{\Z^{\alpha,\Delta}}(\psib_m^{\alpha,\Delta})(s)-\dm_{\Z^{\alpha,\Delta}}(\psib_m^\alpha)(s)|+|\dm_{\Z^{\alpha,\Delta}}(\psib_m^\alpha)(s)-\dm_\Z(\psib_m^\alpha)(s)|\\ 
	&\qquad\leq\lip\sup_{0\leq u\leq s}|\psib_m^{\alpha,\Delta}(u)-\psib_m^\alpha(u)|+|\dm_{\Z^{\alpha,\Delta}}(\psib_m^\alpha)(s)-\dm_\Z(\psib_m^\alpha)(s)|.
\end{align*}
Thus,
\begin{align}\label{eq:phibhphib}
	\norm{\phib^{\alpha,\Delta}(s)-\phib^\alpha(s)}^p&\leq(2M)^{p-1}\lip^p\sup_{0\leq u\leq s}\norm{\psib^{\alpha,\Delta}(u)-\psib^\alpha(u)}^p\\ \notag
	&\qquad+(2M)^{p-1}\sum_{m=1}^M|\dm_{\Z^{\alpha,\Delta}}(\psib_m^\alpha)(s)-\dm_\Z(\psib_m^\alpha)(s)|^p.
\end{align}
Now combining \eqref{eq:psibnpsib}--\eqref{eq:phibhphib} and applying Gronwall's inequality, we obtain
\begin{align*}
	\E\lsb\sup_{0\leq s\leq t}\norm{\psib^{\alpha,\Delta}(s)-\psib^\alpha(s)}^p\rsb&\leq 11^{p-1}\lip^p C_\Delta\exp(11^{p-1}(M+1)^{p-1}\lip^{2p}(t^p+C_pt)),
\end{align*}
where
\begin{align*}
	C_\Delta&\doteq \Delta^p\lb1+\E\lsb\sup_{0\leq s\leq t}\norm{\phib^{\alpha,\Delta}(s)}^p\rsb\rb+\lip^p\E\lsb|\osc(\bm,\Delta,t)|^p\rsb\\
	&\qquad+\lcb\E\lsb\sup_{0\leq s\leq t}\norm{\phib^{\alpha,\Delta}(s)}^{2p}\rsb\E\lsb|\osc(\bm,\Delta,t)|^{2p}\rsb\rcb^{\frac{1}{2}}\\
	&\qquad+(t^p+C_pt)\E\lsb\sup_{0\leq s\leq t}|\Z^{\alpha,\Delta}(s)-\Z^\alpha(s)|^{\gamma p}\rsb\\
	&\qquad+(t^p+C_pt)\lcb\E\lsb\sup_{0\leq s\leq t}|\Z^{\alpha,\Delta}(s)-\Z^\alpha(s)|^{2\gamma p}\rsb\E\lsb\sup_{0\leq s\leq t}\norm{\phib^{\alpha,\Delta}(s)}^{2p}\rsb\rcb^{\frac{1}{2}}\\
	&\qquad+\lip^p\E\lsb\sup_{0\leq s\leq t}|\X^{\alpha,\Delta}(s)-\X^\alpha(s)|^p\rsb\\
	&\qquad+(2M)^{p-1}(t^{p-1}+C_p)\sum_{m=1}^M\E\lsb\int_0^t|\dm_{\Z^{\alpha,\Delta}}(\psib_m^\alpha)(s)-\dm_{\Z^\alpha}(\psib_m^\alpha)(s)|^pds\rsb.
\end{align*}
By Lemma \ref{lem:suphphibalphahbound}, Lemma \ref{lem:oscbmzero}, Theorem \ref{thm:EulerRD}, Proposition \ref{prop:Xhconverge} and Lemma \ref{lem:dmZalphaetabalpha}, $C_\Delta\to0$ as $\Delta\downarrow0$, thus completing the proof.
\end{proof}

\section{Continuity of the derivative map}\label{sec:dmcontinuous}

Throughout this section we fix $\alpha\in U$. For brevity, we drop the $\alpha$ notation and write $d(\cdot)$, $\dm_\z$ and $\proj_x$ for $d(\alpha,\cdot)$, $\dm_\z^\alpha$ and $\proj_x^\alpha$, respectively.

\subsection{Proof of Theorem \ref{thm:dmcontinuous}}

In this section we first prove Theorem \ref{thm:dmcontinuous} when $\psi_\ellindex=\psi$ for all $\ellindex\in\N$ and $\psi$ lies in a class of simple functions (see Lemma \ref{lem:psisimple} below). With that in mind, let $\simple(\R^J)$ denote the set of simple functions in $\dr(\R^J)$; that is,
$$\simple(\R^J)\doteq\lcb\sum_{j=1}^my_j1_{[s_j,s_{j+1})}:m\in\N,y_1,\dots,y_m\in\R^J,0\leq s_1<\cdots<s_{m+1}<\infty\rcb.$$
For $\z\in\cts(G)$ let 
$$B_\z\doteq\lcb t\geq0:\z(t)\in\partial G\rcb$$
denote the times in $[0,\infty)$ that $\z$ lies in the boundary $\partial G$, and for $\psi\in\dr(\R^J)$ let
$$J_\psi\doteq\lcb t>0:\psi(t)\neq\psi(t-)\rcb$$
denote the jump or discontinuity points of $\psi$ in $(0,\infty)$. Recall the definition of $\hyper_x$, $x\in G$, given in \eqref{eq:Hx}. For $\z\in\cts(G)$, define
\be\label{eq:simplez}\simple^\z(\R^J)\doteq\lcb\psi\in\simple(\R^J):\psi(0)\in \hyper_{\z(0)}\text{ and }B_\z\cap J_\psi=\emptyset\rcb.\ee

\begin{lem}\label{lem:dense}
Given $\z\in\cts(G)$, suppose $B_\z$ has zero Lebesgue measure. Then for $\psi\in\cts(\R^J)$ with $\psi(0)\in \hyper_{\z(0)}$, there is a sequence $\{\psi_k\}_{k\in\N}$ in $\simple^\z(\R^J)$ such that $\psi_k$ converges to $\psi$ uniformly on compact intervals as $k\to\infty$.
\end{lem}

\begin{proof}
Fix $\psi\in\cts(\R^J)$ with $\psi(0)\in \hyper_{\z(0)}$. Let $t<\infty$ and $\ve>0$ be arbitrary. It suffices to show that there exists $\tilde{\psi}\in\simple^\z(\R^J)$ such that $\sup_{0\leq s\leq t}|\psi(s)-\tilde{\psi}(s)|<\ve$. Since $\simple(\R^J)$ is dense in $\cts(\R^J)$ under the topology of uniform convergence on compact intervals (see, e.g., \cite[Theorem 6.2.2]{Whitt2002}), we can choose $\hat\psi\in\simple(\R^J)$ such that $\hat\psi(0)=\psi(0)$ and $\sup_{0\leq s\leq t}|\psi(s)-\hat\psi(s)|<\ve/2$. By the definition of $\simple(\R^J)$ and the fact that $\hat\psi(0)=\psi(0)$, there exist $m\in\N$, $y_1,\dots,y_m\in\R^J$ and $0=s_1<\cdots<s_{m+1}$ such that $y_1=\psi(0)$ and
	$$\hat\psi=\sum_{j=1}^my_j1_{[s_j,s_{j+1})}.$$
Without loss of generality, we can assume $s_{m+1}>t$. Since $\psi$ is uniformly continuous on $[0,t]$, we can choose $\delta>0$ sufficiently small so that $s_{m+1}>t+\delta$, $s_j+\delta<s_{j+1}-\delta$ for each $j=1,\dots,m$, and
	$$\sup\{|\psi(u)-\psi(s)|:s,u\in[0,t],|u-s|<\delta\}<\ve/2.$$
Set $\tilde{s}_1=s_1=0$. For each $j=2,\dots,m+1$, let $\tilde{s}_j\in(s_j-\delta,s_j+\delta)\setminus B_\z$, which is nonempty because $B_\z$ has zero Lebesgue measure. Define $\tilde{\psi}\in\simple^\z(\R^J)$ by 
	$$\tilde{\psi}=\sum_{j=1}^my_j1_{[\tilde{s}_j,\tilde{s}_{j+1})}.$$
Then $\tilde{s}_{m+1}>t$. Let $s\in[0,t]$ and $j\in\{1,\dots,m\}$ be such that $s\in[\tilde{s}_j,\tilde{s}_{j+1})\subset(s_j-\delta,s_{j+1}+\delta)$. Then there exists $u\in(s_j,s_{j+1})$ such that $|u-s|<\delta$ and so $\tilde{\psi}(s)=y_j=\hat\psi(u)$ and
\begin{align*}
	|\psi(s)-\tilde{\psi}(s)|&\leq|\psi(s)-\psi(u)|+|\psi(u)-\hat\psi(u)|<\ve.
\end{align*}
Since $s\in[0,t]$ was arbitrary, this completes the proof.
\end{proof}

We now state a key convergence lemma, Lemma \ref{lem:psisimple}, and show how it, together with Lemma \ref{lem:dense}, can be used to prove Theorem \ref{thm:dmcontinuous}. The proof of  Lemma \ref{lem:psisimple} is deferred to  Section \ref{sec:psicts}, after first establishing some preliminary results in Section \ref{sec:derivative problemsolutions}.

\begin{lem}\label{lem:psisimple}
For $\alpha\in U$, given $\x\in\cts_G(\R^J)$ suppose that the solution $(\z,\y)$ to the ESP for $\x$ satisfies the boundary jitter property. Let $\{\x_\ellindex\}_{\ellindex\in\N}$ be a sequence in $\dr(\R^J)$ such that $\x_\ellindex$ converges to $\x$ in $\dr(\R^J)$ as $\ellindex\to\infty$ and for each $\ellindex\in\N$, let $(\z_\ellindex,\y_\ellindex)$ denote the solution to the ESP for $\x_\ellindex$. Then for all $\psi\in\simple^\z(\R^J)$, $\dm_{\z_\ellindex}(\psi)$ converges to $\dm_\z(\psi)$ in $\dr(\R^J)$ as $\ellindex\to\infty$.
\end{lem}

\begin{proof}[Proof of Theorem \ref{thm:dmcontinuous}]
Fix $\alpha \in U$ and, for any $h \in \dr(\R^J)$,  denote $\Lambda_h^\alpha$ just by $\Lambda_h$.  Fix $\x\in\cts_G(\R^J)$ and a sequence $\{\x_\ellindex\}_{\ellindex\in\N}$ in $\dr(\R^J)$ such that $\x_\ellindex$ converges to $\x$ in $\dr(\R^J)$ as $\ellindex\to\infty$. Let $(\z,\y)$ denote the solution to the ESP for $\x$ and for each $\ellindex\in\N$, let $(\z_\ellindex,\y_\ellindex)$ denote the solution to the ESP for $\x_\ellindex$. Fix $\psi\in\cts(\R^J)$ and a sequence $\{\psi_\ellindex\}_{\ellindex\in\N}$ in $\dr(\R^J)$ such that $\psi_\ellindex$ converges to $\psi$ in $\dr(\R^J)$ as $\ellindex\to\infty$. Let $d_{J_1}(\cdot,\cdot)$ denote a metric on $\dr(\R^J)$ that is compatible with the Skorokhod $J_1$-topology. Let $d_{U}(\cdot,\cdot)$ denote a metric on $\dr(\R^J)$ that is compatible with the topology of uniform convergence on compact intervals. Since the topology of uniform convergence on compact intervals is a stronger topology than the Skorokhod $J_1$-topology, we can assume the metrics are chosen such that $d_{J_1}(\cdot,\cdot)\leq d_U(\cdot,\cdot)$. Then by the triangle inequality and Lipschitz continuity of the derivative map (Proposition \ref{prop:dmlip}),
	\begin{align*}
	d_{J_1}(\dm_{\z_\ellindex}(\psi_\ellindex),\dm_\z(\psi))&\leq d_U(\dm_{\z_\ellindex}(\psi_\ellindex),\dm_{\z_\ellindex}(\psi))+d_{J_1}(\dm_{\z_\ellindex}(\psi),\dm_\z(\psi))\\
	&\leq\lip_\dm d_U(\psi_\ellindex,\psi)+d_{J_1}(\dm_{\z_\ellindex}(\psi),\dm_\z(\psi)).
	\end{align*}
	Since $\psi\in\cts(\R^J)$ and the Skorokhod $J_1$-topology relativized to $\cts(\R^J)$ coincides with the topology of uniform convergence on compact intervals there, $\psi_\ellindex$ converges to $\psi$ uniformly on compact intervals as $\ellindex\to\infty$. Therefore, $d_U(\psi_\ellindex,\psi)$ converges to zero as $\ellindex\to\infty$. We are left to show $d_{J_1}(\dm_{\z_\ellindex}(\psi),\dm_\z(\psi))$ converges to zero as $\ellindex\to\infty$. According to Lemma \ref{lem:dense}, there is a sequence $\{\psi_k\}_{k\in\N}$ in $\simple^\z(\R^J)$ such that $\psi_k$ converges to $\psi$ in $\dr(\R^J)$ as $k\to\infty$. For each $k\in\N$, by the Lipschitz continuity of the derivative map,
	\begin{align*}
	d_{J_1}(\dm_{\z_\ellindex}(\psi),\dm_\z(\psi))&\leq d_{J_1}(\dm_{\z_\ellindex}(\psi),\dm_{\z_\ellindex}(\psi_k))+d_{J_1}(\dm_{\z_\ellindex}(\psi_k),\dm_\z(\psi_k))+d_{J_1}(\dm_\z(\psi_k),\dm_\z(\psi))\\
	&\leq2\lip_\dm d_U(\psi,\psi_k)+d_{J_1}(\dm_{\z_\ellindex}(\psi_k),\dm_\z(\psi_k)).
	\end{align*}
	By Lemma \ref{lem:psisimple}, $d_{J_1}(\dm_{\z_\ellindex}(\psi_k),\dm_\z(\psi_k))$ converges to zero as $\ellindex\to\infty$. Then letting $k\to\infty$, $d_U(\psi,\psi_k)$ converges to zero. This completes the proof.
\end{proof}

\subsection{Characterization of solutions to the derivative problem}\label{sec:derivative problemsolutions}

In this section we characterize solutions to the DP along $\z$ when the input is simple up until the first time $\z$ hits the intersection of two or more faces. As a preliminary result, in Lemma \ref{lem:DPpsiconstantST}, we first characterize solutions to the DP when the input $\psi$ is constant.

In order to characterize the solutions to the DP, we first need some definitions. Recall that $F_i=\{x\in G:\ip{x,n_i}=c_i\}$ denotes the $i$th face of $G$, for $i\in\allN$; and $\allN(x)=\{i\in\allN:x\in F_i\}$ for $x\in G$. The following is an upper semicontinuity property of the set-valued function $\allN(\cdot)$.

\begin{lem}[{\cite[Lemma 2.1]{Kang2007}}]
	\label{lem:allNusc}
	For each $x\in G$, there is an open neighborhood $U_x$ of $x$ in $\R^J$ such that $\allN(y)\subseteq\allN(x)$ for all $y\in U_x\cap G$.
\end{lem}

Suppose $(\z,\y)$ is the solution to the ESP for $\x\in\cts_G(\R^J)$ and $0\leq S<T<\infty$. Define $\hat{t}_1\doteq S$ and for $k\geq1$ such that $\hat{t}_k<T$, define
	\be\label{eq:rhokST}\hat{\rho}_k\doteq\inf\{t\in(\hat{t}_k,T]:\allN(\z(t))\not\subseteq\allN(\z(\hat{t}_k))\}\wedge T\ee
to be the first time after $\hat{t}_k$ that $\z$ lies on a boundary face that is distinct from the faces on which $\z(\hat{t}_k)$ lies. Also, define
	\be\label{eq:tkST}\hat{t}_{k+1}\doteq\sup\{t\in[\hat{\rho}_k,T]:\allN(\z(s))\subseteq\allN(\z(t))\;\forall\;s\in[\hat{\rho}_k,t]\}.\ee
We claim that $\hat{t}_k=T$ for some $k\in\N$, at which point we set $K\doteq k$ and end the sequence. To see that the claim holds, suppose for a proof by contradiction that $\hat{t}_k<T$ for all $k\in\N$. Since $\{\hat{t}_k\}_{k\in\N}$ is increasing and bounded from above, there exists $\hat{t}_\infty\leq T$ such that $\hat{t}_k\to \hat{t}_\infty$ as $k\to\infty$. By the upper semicontinuity of $\allN(\cdot)$ shown in Lemma \ref{lem:allNusc} and the continuity of $\z$, this implies there exists $k_0\in\N$ such that $\allN(\z(t))\subseteq\allN(\z(\hat{t}_\infty))$ for all $t\geq \hat{t}_{k_0}$. Since $\hat t_{k_0}<\hat\rho_{k_0}<\hat t_{k_0+1}<\hat t_\infty$, this implies
	\be\label{eq:allNzttinfty}\allN(\z(t))\subseteq\allN(\z(\hat{t}_\infty))\qquad\text{for all }t\geq \hat{\rho}_{k_0}.\ee 
However, \eqref{eq:tkST} and \eqref{eq:allNzttinfty} imply $\hat t_{k_0+1}\geq\hat t_\infty$, which yields a contradiction. With this contradiction thus obtained, we see that $\hat{t}_k=T$ for some $k\in\N$.

We now characterize solutions to the DP when the input $\psi$ is constant. In this case $\phi=\dm_\z(\psi)$ is constant while $\z$ is in the interior of $G$ and jumps whenever $\z$ hits the boundary $\partial G$ and the jump can be expressed in terms of a derivative projection matrix. (Recall the derivative projection matrices $\proj_x$, $x\in G$, introduced in Lemma \ref{lem:projx}.)

\begin{lem}\label{lem:DPpsiconstantST}
Given $\x\in\dr(\R^J)$ and $\psi\in\dr(\R^J)$, suppose $(\z,\y)$ is the solution to the ESP for $\x$ and $(\phi,\eta)$ is the solution to the derivative problem along $\z$ for $\psi$. Let $0\leq S<T<\infty$ and define the nested sequence $S=\hat{t}_1<\hat\rho_1\leq \hat{t}_2<\cdots\leq \hat{t}_K=T$ as in \eqref{eq:rhokST}--\eqref{eq:tkST}. Suppose $\psi$ is constant on $[S,T]$. Then
	\be\label{eq:phiTprojphiS}\phi(T)=\proj_{\z(T)}\proj_{\z(\hat{t}_{K-1})}\cdots\proj_{\z(\hat{t}_2)}\phi(S).\ee
\end{lem}

\begin{proof}
We will prove the relation
	\be\label{eq:phitk}\phi(\hat{t}_k)=\proj_{\z(\hat{t}_k)}\phi(\hat{t}_{k-1}),\qquad k=2,\dots, K.\ee 
We can then iterate \eqref{eq:phitk} and use the fact that $\hat{t}_1=S$ to obtain \eqref{eq:phiTprojphiS}. Let $k\in\{2,\dots,K\}$. We first claim that $\phi$ is constant on $[\hat{t}_{k-1},\hat{\rho}_ {k-1})$. Due to the uniqueness that follows from Lipschitz continuity of the derivative map stated in Proposition \ref{prop:dmlip}, in order to prove the claim it suffices to show that if $(\phi,\eta)$ satisfies conditions 1--4 of the derivative problem for $t\in[0,\hat{t}_{k-1}]$ and is constant on $[\hat{t}_{k-1},\hat{\rho}_{k-1})$, then $(\phi,\eta)$ satisfies conditions 1--4 of the derivative problem for $t\in[\hat{t}_{k-1},\hat{\rho}_{k-1})$. Let $t\in[\hat{t}_{k-1},\hat{\rho}_{k-1})$. Since $(\phi,\eta)$ satisfies condition 1 of the derivative problem at $\hat{t}_{k-1}$ and $\phi,\eta,\psi$ are constant on $[\hat{t}_{k-1},\hat{\rho}_{k-1})$, we have $\phi(t)=\phi(\hat{t}_{k-1})=\psi(\hat{t}_{k-1})+\eta(\hat{t}_{k-1})=\psi(t)+\eta(t)$, so condition 1 of the derivative problem holds at $t$. Next, since $\phi$ satisfies condition 2 of the derivative problem at $\hat{t}_{k-1}$, $\phi$ is constant on $[\hat{t}_{k-1},\hat{\rho}_{k-1})$ and $\allN(\z(t))\subseteq\allN(\z(\hat{t}_{k-1}))$ by the definition of $\hat{\rho}_{k-1}$ in \eqref{eq:rhokST}, we have $\phi(t)=\phi(\hat{t}_{k-1})\in \hyper_{\z(\hat{t}_{k-1})}\subseteq \hyper_{\z(t)}$, where the last inclusion holds by \eqref{eq:Hx}.   Thus, condition 2 of the derivative problem holds at $t$. Lastly, since $\eta$ is constant on $[\hat{t}_{k-1},\hat{\rho}_{k-1})$, conditions 3 and 4 of the derivative problem are straightforward to verify, so we omit the details. This concludes the proof that $(\phi,\eta)$ is constant on $[\hat{t}_{k-1},\hat{\rho}_{k-1})$.
	
Next, by the definition of $\hat{t}_k$ in \eqref{eq:tkST}, the upper semicontinuity of $\allN(\cdot)$ shown in Lemma \ref{lem:allNusc} and the continuity of $\z$, there exists $\delta>0$ such that $\allN(\z(t))\subseteq\allN(\z(\hat{t}_k))$ for all $t\in[\hat{\rho}_{k-1}-\delta,\hat{t}_k]$. Along with the fact that $\phi$ is constant on $[\hat{t}_{k-1},\hat{\rho}_{k-1})$, conditions 1--4 of the derivative problem,  the fact that $\psi$ is constant on $[S,T]$, and the definition of $d(\cdot)$ in \eqref{eq:dalphax}, this implies that $\phi(\hat{t}_k)\in \hyper_{\z(\hat{t}_k)}$ and
	$$\phi(\hat{t}_k)-\phi(\hat{t}_{k-1})=\eta(\hat{t}_k)-\eta(\hat{\rho}_{k-1}-\delta)\in\spaan\lsb\cup_{u\in(\hat{\rho}_{k-1}-\delta,\hat{t}_k]}d(\z(u))\rsb=\spaan\lsb d(\z(\hat{t}_k))\rsb.$$
The unique characterization of $\proj_{\z(\hat{t}_k)}$ stated in Lemma \ref{lem:projx} then implies \eqref{eq:phitk} holds.
\end{proof}

We have the following corollary of Lemma \ref{lem:DPpsiconstantST}.

\begin{cor}\label{cor:DPpsiconstantST}
Given $\x\in\dr(\R^J)$ and $\psi\in\dr(\R^J)$, suppose $(\z,\y)$ is the solution to the ESP for $\x$ and $(\phi,\eta)$ is the solution to the derivative problem along $\z$ for $\psi$. Let $0\leq S<T<\infty$ and suppose $\psi$ is constant on $[S,T]$. Then there exists $\hat m\in\N$ and $S<\hat s_1<\cdots<\hat s_{\hat m}=T$ such that
	\be\label{eq:phiTprojphiS}\phi(T)=\proj_{\z(\hat s_{\hat m})}\cdots\proj_{\z(\hat s_1)}\phi(S).\ee
\end{cor}

In Lemma \ref{lem:phisimpletheta_2} below, we characterize solutions to the DP along $\z$ for simple $\psi$ up until the first time $\z$ hits the intersection of two or more faces. In the next section we use an inductive argument (on the first time $\z$ hits $n$ or more faces) to prove continuity properties of the derivative map. With that in mind, given a solution $(\z,\y)$ to the ESP for $\x\in\cts_G(\R^J)$, define
\be\label{eq:thetan}\theta_n\doteq\inf\{t\geq 0:|\allN(\z(t))|\geq n\},\qquad n=1,\dots,N,\ee
and set $\theta_{N+1}\doteq\infty$. Suppose $\psi\in\simple(\R^J)$ is of the form
\be\label{eq:psisimple}\psi=\sum_{j=1}^my_j1_{[s_j,s_{j+1})}\ee
for some $m\in\N$, $y_1,\dots,y_m\in\R^J$ such that $y_j\neq y_{j+1}$ for $1\leq j<m$, and $0=s_1<\cdots<s_{m+1}<\infty$ with $s_{m+1}\leq\theta_2$. For each $1\leq j\leq m$, define $K_j\in\N_\infty$ and the sequence $\{\rho_j(k)\}_{k=1,\dots,K_j}$ in $[s_j,s_{j+1}]$ as follows: set $\rho_j(1)\doteq s_j$ and for $k\in\N$ such that $\rho_j(k)<s_{j+1}$, recursively define
\be\label{eq:tkj}\rho_j(k+1)\doteq\inf\{t>\rho_j(k):\allN(\z(t))\not\subseteq\allN(\z(\rho_j(k)))\}\wedge s_{j+1}.\ee
If $\rho_j(k)=s_{j+1}$ for some $k\in\N$, then set $K_j\doteq k$ and end the sequence so that
\be\label{eq:sjsj1}s_j=\rho_j(1)<\cdots<\rho_j(K_j)=s_{j+1}.\ee
On the other hand, if $\rho_j(k)<s_{j+1}$ for all $k\in\N$, then set $K_j\doteq\infty$. In the case that $K_j=\infty$, since $\rho_j(k)$ is increasing and bounded above by $s_{j+1}$, there exists $\rho_j(\infty)\leq s_{j+1}$ such that $\rho_j(k)\to \rho_j(\infty)$ as $k\to\infty$. If $\rho_j(\infty)=\infty$, then we must have $s_{j+1}=\infty$, and hence,  $\theta_2=\infty$. Alternatively, if $\rho_j(\infty)<\infty$, then the definition of $\rho_j(k)$ given in \eqref{eq:tkj} and the continuity of $\z$ imply that $\allN(\z(\rho_j(\infty)))|\geq2$, which along with the fact that $\rho_j(k)\leq\theta_2$ for all $k\in\N$ implies $\rho_j(\infty)=\theta_2$. In either case, we see that if $K_j=\infty$, then $\rho_j(k)\to\theta_2$ as $k\to\infty$. Hence, due to the fact that $s_{m+1}=\theta_2$, we have $K_j=\infty$ only if $j=m$. In this case,
\be\label{eq:smtheta_2}s_m=\rho_m(1)<\rho_m(2)<\cdots<s_{m+1}=\theta_2.\ee

\begin{lem}\label{lem:phisimpletheta_2}
Given $\x\in\cts_G(\R^J)$, let $(\z,\y)$ denote the solution to the ESP for $\x$ and define $\theta_2$ as in \eqref{eq:thetan}. Suppose $\psi\in\simple(\R^J)$ satisfies \eqref{eq:psisimple}. For each $1\leq j\leq m$, define $K_j\in\N_\infty$ and the sequence $\{\rho_j(k)\}_{k=1,\dots,K_j}$ as in \eqref{eq:tkj}. Then the solution $(\phi,\eta)$ to the derivative problem along $\z$ for $\psi$ on $[0,\theta_2)$ satisfies, for each $1\leq j\leq m$,
	\be\label{eq:phisimpletheta_2}\phi(t)=
	\proj_{\z(\rho_j(k))}\cdots\proj_{\z(\rho_j(1))}(\phi(s_j-)+y_j)\qquad\text{for }t\in[\rho_j(k),\rho_j(k+1)),\;1\leq k<K_j,
	\ee
where $\phi(0-)\doteq0$.
\end{lem}

\begin{proof}
	In order to prove the lemma, it suffices to show that $(\phi,\phi-\psi)$, where $\phi$ is defined as in \eqref{eq:phisimpletheta_2}, satisfies conditions 1--3 of the derivative problem along $\z$ for $\psi$ on the interval $[0,\theta_2)$ (since condition 4 follows from condition 3 when $\z$ is continuous). Condition 1 holds automatically. We next show that condition 2 holds. Given $t\in[0,\theta_2)$ let $1\leq j\leq m$ and $1\leq k<K_j$ be such that $t\in[\rho_j(k),\rho_j(k+1))$. By \eqref{eq:tkj}, $\allN(\z(t))\subseteq\allN(\z(\rho_j(k)))$. Therefore, due to the fact that $\proj_{\z(\rho_j(k))}$ maps $\R^J$ into $\hyper_{\z(\rho_j(k))}$, it follows that 
	$$\phi(t)=\proj_{\z(\rho_j(k))}\cdots\proj_{\z(\rho_j(1))}(\phi(s_j-)+y_j)\in \hyper_{\z(\rho_j(k))}\subseteq \hyper_{\z(t)}.$$
	Therefore, condition 2 of the derivative problem holds. We are left to show that condition 3 of the derivative problem holds. First observe that for $1\leq j\leq m$ and $1\leq k<K_j$, \eqref{eq:psisimple} and \eqref{eq:phisimpletheta_2} imply that $\psi$ and $\phi$ are constant on $[\rho_j(k),\rho_j(k+1))$, and hence $\eta\doteq\phi-\psi$ is constant there as well. Therefore, it suffices to show that 
	\be\label{eq:etatjketatjkminus}\eta(\rho_j(k))-\eta(\rho_j(k)-)\in\spaan[d(\z(\rho_j(k)))]\ee 
	for each $1\leq j\leq m$ and $1\leq k<K_j$. Suppose $1\leq j\leq m$ and $k=1$. Then $\rho_j(1)=s_j$ by definition, so the facts that $\eta\doteq\phi-\psi$, $\phi(s_j)=\proj_{\z(s_j)}(\phi(s_j-)+y_j)$ by \eqref{eq:phisimpletheta_2}, $\psi(s_j)=\psi(s_j-)+y_j$ by \eqref{eq:psisimple} and $\proj_{\z(s_j)}(\phi(s_j-)+y_j)-(\phi(s_j-)+y_j)\in\spaan[d(\z(s_j))]$ by Lemma \ref{lem:projx}, together imply
	\begin{align*}
	\eta(\rho_j(1))-\eta(\rho_j(1)-)&=\phi(s_j)-\phi(s_j-)-(\psi(s_j)-\psi(s_j-))\\ 
	&=\proj_{\z(s_j)}(\phi(s_j-)+y_j)-(\phi(s_j-)+y_j)\\
	&\in\spaan[d(\z(\rho_j(1)))].
	\end{align*}
	Next, suppose $1\leq j\leq m$ and $1<k<K_j$. Then the facts that $\eta\doteq\phi-\psi$, $\phi(\rho_j(k))=\proj_{\z(\rho_j(k))}(\phi(\rho_j(k)-))$ by \eqref{eq:phisimpletheta_2}, $\psi$ is continuous at $\rho_j(k)$ by \eqref{eq:psisimple} and $\proj_{\z(\rho_j(k))}(\phi(\rho_j(k)-))-\phi(\rho_j(k)-)\in\spaan[d(\z(\rho_j(k))]$ together imply
	\begin{align*}
	\eta(\rho_j(k))-\eta(\rho_j(k)-)&=\phi(\rho_j(k))-\phi(\rho_j(k)-)\\ 
	&=\proj_{\z(\rho_j(k))}\phi(\rho_j(k)-)-\phi(\rho_j(k)-)\\ 
	&\in\spaan[d(\z(\rho_j(k)))].
	\end{align*}
	This verifies \eqref{eq:etatjketatjkminus}, and thus proves the lemma.
\end{proof}

\subsection{Proof of Lemma \ref{lem:psisimple}}\label{sec:psicts}

We use the principle of mathematical induction to prove the following statement holds for $n=2,\dots,N+1$. Since $\theta_{N+1}\doteq\infty$ by definition, this will complete the proof of Lemma \ref{lem:psisimple}. Given $T<\infty$, we say that a function $\lambda:[0,T)\to[0,T)$ is a \emph{time-change} if $\lambda$ is nondecreasing and onto.

\begin{statement}\label{state:inductioncts}
Given $\x\in\cts_G(\R^J)$, suppose that the solution $(\z,\y)$ to the ESP for $\x$ satisfies the boundary jitter property (Definition \ref{def:jitter}). Let $\{\x_\ellindex\}_{\ellindex\in\N}$ be a sequence in $\dr(\R^J)$ such that $\x_\ellindex$ converges to $\x$ in $\dr(\R^J)$ as $\ellindex\to\infty$ and for each $\ellindex\in\N$, let $(\z_\ellindex,\y_\ellindex)$ denote the solution to the ESP for $\x_\ellindex$. Then given $\psi\in\simple^\z(\R^J)$, there is a sequence $\{\lambda_\ellindex\}_{\ellindex\in\N}$ of time-changes mapping $[0,\theta_n)$ to $[0,\theta_n)$ such that for all $T<\infty$,
	\be\label{eq:timechange1thetan}\lim_{\ellindex\to\infty}\sup_{t\in[0,\theta_n\wedge T)}|\lambda_\ellindex(t)-t|=0,\ee
and for all $T<\theta_n$, 
	\be\label{eq:timechange2thetan}\lim_{\ellindex\to\infty}\sup_{t\in[0,T]}\left|\dm_{\z_\ellindex}(\psi)(\lambda_\ellindex(t))-\dm_\z(\psi)(t)\right|=0.\ee
\end{statement}

We first prove the base case $n=2$.

\begin{lem}\label{lem:theta_2psiconst}
Statement \ref{state:inductioncts} holds with $n=2$.
\end{lem}

Since the proof of the base case is lengthy, we first provide a brief outline of the argument. First, we show that because $(\z,\y)$ satisfies the boundary jitter property and $\x_\ellindex$ converges to $\x$ as $\ellindex\to\infty$, it follows that, roughly speaking, on a compact interval $[0,T]$ contained in $[0,\theta_2)$ and for $\ellindex$ sufficiently large, $\z_\ellindex$ hits the same boundary faces that $\z$ hits, and $\z_\ellindex$ hits those boundary faces in the same order that $\z$ does. The time-change $\lambda_{\ellindex}$ is then constructed so that, again roughly speaking, on $[0,T]$, $\psi\circ\lambda_\ellindex$ jumps at the same time that $\psi$ jumps, and $\z_\ellindex\circ\lambda_\ellindex$ hits the same boundary faces that $\z$ hits, and both paths, $\z_\ellindex\circ\lambda_\ellindex$ and $\z$, hit those boundary faces in the same order and at the same times. The paths $\dm_{\z_\ellindex}(\psi)\circ\lambda_\ellindex$ and $\dm_{\z}(\psi)$ will then be equal on $[0,T]$ since the jumps of both paths depend only on the jump times and jump sizes of $\psi\circ\lambda_\ellindex$ and $\psi$, respectively, and on the times that $\z_\ellindex\circ\lambda_\ellindex$ and $\z$, respectively, hit the boundary faces.

\begin{proof}[Proof of Lemma \ref{lem:theta_2psiconst}]
Let $\x$, $(\z,\y)$, $\{\x_\ellindex\}_{\ellindex\in\N}$, $\{(\z_\ellindex,\y_\ellindex)\}_{\ellindex\in\N}$ and $\psi$ be as in Statement \ref{state:inductioncts}. Define $\theta_2$ as in \eqref{eq:thetan}. If $\theta_2=0$, then \eqref{eq:timechange1thetan} and \eqref{eq:timechange2thetan} hold trivially. For the remainder of the proof we assume that $\theta_2>0$. Since $\x_\ellindex$ converges to $\x$ in $\dr(\R^J)$ as $j\to\infty$, $\x\in\cts_G(\R^J)$ and the Skorokhod $J_1$-topology relativized to $\cts(\R^J)$ coincides with the topology of uniform convergence on compact intervals, the Lipschitz continuity of the ESM stated in Proposition \ref{prop:sp}(ii) implies the following: 
	\begin{itemize}
		\item[(a)] $(\z_\ellindex,\y_\ellindex)$ converges to $(\z,\y)$ uniformly on compact intervals as $\ellindex\to\infty$.
	\end{itemize}
Since we are only concerned with the interval $[0,\theta_2)$ in this proof and $\psi\in\simple^\z(\R^J)$, we can assume without loss of generality that $\psi$ is of the form \eqref{eq:psisimple} for some $m\in\N$, $y_1\in \hyper_{\z(0)}$, $y_2,\dots,y_m\in\R^J$ such that $y_j\neq y_{j+1}$ for $1\leq j<m$, and $0=s_1<\cdots<s_{m+1}<\infty$ with $s_{m+1}\leq\theta_2$. For each $1\leq j\leq m$, define $K_j\in\N_\infty$ and the sequence $\{\rho_j(k)\}_{k=1,\dots,K_j}$ in $[s_j,s_{j+1}]$ as in \eqref{eq:tkj} and the subsequent two sentences. As noted following \eqref{eq:tkj}, $K_j<\infty$ for each $1\leq j<m$ and
	\begin{align}\label{eq:tmKm}
	\rho_m(K_m)&=\theta_2=\infty&&\text{if }K_m<\infty,\\ \label{eq:limktmk}
	\lim_{k\to\infty}\rho_m(k)&=\theta_2&&\text{if }K_m=\infty.
	\end{align}
By the definition of $\rho_j(k)$ in \eqref{eq:tkj}, the definition of $\theta_2$ in \eqref{eq:thetan} and condition 1 of the boundary jitter property, for each $1\leq j\leq m$ and $1<k<K_j$, there is a unique index $i_j(k)\in\allN$ such that 
	\begin{itemize}
		\item[(b)] $\allN(\z(\rho_j(k)))=\{i_j(k)\}$ and $\y$ is nonconstant on every neighborhood of $\rho_j(k)$.
	\end{itemize}
By the upper semicontinuity of $\allN(\cdot)$ shown in Lemma \ref{lem:allNusc}, the continuity of $\z$, the definition of $\rho_j(k)$ in \eqref{eq:tkj} and (b), for each $1\leq j\leq m$ and $1\leq k<K_j$, there exists $\xi_j(k)\in(\rho_j(k),\rho_j(k+1))$, which is not necessarily unique, such that
	\begin{align}\label{eq:zGcircsjxj1}
	\z(t)&\in G^\circ,&&t\in[s_j,\xi_j(1)],&&1\leq j\leq m,\\ \label{eq:zGcircxjk1tjk}
	\z(t)&\in G^\circ,&&t\in[\xi_j(k-1),\rho_j(k)),&&1\leq j\leq m,\;1<k<K_j,\\ \label{eq:allNztjkxijk}
	\allN(\z(t))&\subseteq\{i_j(k)\},&&t\in[\rho_j(k),\xi_j(k)],&&1\leq j\leq m,\;1<k<K_j,\\ \label{eq:zGcircxijKj1sj1}
	\z(t)&\in G^\circ,&&t\in[\xi_j(K_j-1),s_{j+1}],&&1\leq j<m.
	\end{align}
Set $\ellindex_1\doteq1$. By \eqref{eq:zGcircsjxj1}--\eqref{eq:zGcircxijKj1sj1}, the upper semicontinuity of $\allN(\cdot)$, (b) and (a), for each $1<\hat k<K_m$, we can choose $\ellindex_{\hat k}\geq\ellindex_{\hat k-1}$ such that for all $\ellindex\geq\ellindex_{\hat k}$,
	\begin{align}\label{eq:allNzell0xi11}
	\z_\ellindex(t)&\in G^\circ,&&t\in[s_j,\xi_j(1)],&&1\leq j\leq m,\\ \label{eq:zellGcirctj1xij1}
	\allN(\z_\ellindex(t))&\subseteq\{i_j(k)\},&&t\in[\xi_j(k-1),\xi_j(k)],&&1\leq j<m,\;1< k<K_j,\\ 
	\z_\ellindex(t)&\in G^\circ,&&t\in[\xi_j(K_j-1),s_{j+1}],&&1\leq j<m,\\ \label{eq:allNzelsubsetimk}
	\allN(\z_\ellindex(t))&\subseteq\{i_m(k)\},&&t\in[\xi_m(k-1),\xi_m(k)],&&1< k\leq \hat k,
	\end{align}
and 
	\begin{itemize}
		\item[(c)] $\y_\ellindex$ is nonconstant on $(\xi_j(k-1),\xi_j(k))$ for $1\leq j<m$ and $1<k<K_j$,
		\item[(d)] $\y_\ellindex$ is nonconstant on $(\xi_m(k-1),\xi_m(k))$ for $1<k\leq\hat{k}$.
	\end{itemize}
	By \eqref{eq:zellGcirctj1xij1}, \eqref{eq:allNzelsubsetimk}, (c), (d) and the fact that $\y_\ellindex$ can only increase when $\z_\ellindex$ lies in $\partial G$, we have
	\begin{itemize}
		\item[(e)] for each $1\leq j<m$ and $1<k<K_j$, $\allN(\z_\ellindex(t))=\{i_j(k)\}$ for some $t\in(\xi_j(k-1),\xi_j(k))$,
		\item[(f)] for each $1<k\leq\hat k$, $\allN(\z_\ellindex(t))=\{i_m(k)\}$ for some $t\in(\xi_m(k-1),\xi_m(k))$.
	\end{itemize}
Given $\ellindex\in\N,$ define $\theta_2^\ellindex\doteq\inf\{t\geq0:|\allN(\z_\ellindex(t))|\geq 2\}$ and for each $1\leq j\leq m$, define $K_j^\ellindex\in\N_\infty$ and the sequence $\{\rho_j^\ellindex(k)\}_{k=1,\dots,K_j^\ellindex}$ in $[s_j,s_{j+1}]$ as follows: set $\rho_j^\ellindex\doteq s_j$ and, for $k\in\N$ such that $\rho_j^\ellindex(k)<s_{j+1}$, define 
	\be\label{eq:tkjell}\rho_j^\ellindex(k+1)\doteq\inf\{t>\rho_j^\ellindex(k):\allN(\z_\ellindex(t))\not\subseteq\allN(\z_\ellindex(\rho_j(k)))\}\wedge s_{j+1}.\ee
If $\rho_j^\ellindex(k)=s_{j+1}$ for some $k\in\N$, then set $K_j^\ellindex\doteq s_{j+1}$. Alternatively, if $\rho_j^\ellindex(k)<s_{j+1}$ for all $k\in\N$, then set $K_j^\ellindex\doteq\infty$. Let $1\leq\hat k<K_m$ and suppose $\ellindex\geq\ellindex_{\hat{k}}$. Then \eqref{eq:tkjell}, the definition of $K_j^\ellindex$ for $1\leq j\leq m$, \eqref{eq:allNzell0xi11}--\eqref{eq:allNzelsubsetimk}, (e) and (f) imply that 
\begin{itemize}
		\item[(g)] $K_j^\ellindex=K_j$ for $1\leq j<m$,
		\item[(h)] $\rho_j^\ellindex(k)\in(\xi_j(k-1),\xi_j(k))$ and $\allN(\z_\ellindex(\rho_j^\ellindex(k)))=\{i_j(k)\}=\allN(\z(\rho_j(k)))$ for $1\leq j<m$ and $1< k< K_j$,
		\item[(i)] $\rho_m^\ellindex(k)\in(\xi_m(k-1),\xi_m(k))$ and $\allN(\z_\ellindex(\rho_m^\ellindex(k)))=\{i_m(k)\}=\allN(\z(\rho_m(k)))$ for $1<k\leq\hat k$.
\end{itemize}
Now fix $1\leq j\leq m$ and $1<k<K_j$. We show that 
	\be\label{eq:tkjtotk}\lim_{\ellindex\to\infty}\rho_j^\ellindex(k)=\rho_j(k).\ee 
Since $\xi_j(k-1)<\rho_j(k)<\xi_j(k)$, we can choose $\delta>0$ sufficiently small so that $(\rho_j(k)-\delta,\rho_j(k)+\delta)\subseteq[\xi_j(k-1),\xi_j(k)]$. Let $1\leq\hat k<K_m$ be such that $\rho_j(k)\leq \rho_m(\hat k)$ (e.g., if $j<m$, choose $\hat k=1$; and if $j=m$, choose $\hat k=k$). By (a), \eqref{eq:zGcircxjk1tjk}, (b), (c) and (d), there exists a positive integer $\ellindex^\delta\geq \ellindex_{\hat k}$ sufficiently large such that if $\ellindex\geq\ellindex^\delta$, then $\z_\ellindex(t)\in G^\circ$ for $t\in[\xi_j(k-1),\rho_j(k)-\delta]$ and $\y_\ellindex$ is nonconstant on $(\rho_j(k)-\delta,\rho_j(k)+\delta)$. This, along with (e) and (f), implies that $\rho_j^\ellindex(k)\in(\rho_j(k)-\delta,\rho_j(k)+\delta)$ for all $\ellindex\geq\ellindex^\delta$. Since $\delta>0$ was arbitrary, \eqref{eq:tkjtotk} holds.
	
We now define the time-changes $\lambda_\ellindex:[0,\theta_2)\mapsto[0,\theta_2)$. Let $\{\ellindex_{\hat k}\}_{1\leq\hat{k}<K_m}$ be such that for each $1\leq\hat{k}<K_m$ and all $\ellindex\geq\ellindex_{\hat{k}}$, \eqref{eq:allNzell0xi11}--\eqref{eq:allNzelsubsetimk} hold. Fix $\ellindex\in\N$. If $\ellindex<\ellindex_1$, then set $\lambda_\ellindex(t)\doteq t$ for all $t\in[0,\theta_2)$. Now suppose $\ellindex\geq\ellindex_1$. Let $1\leq\hat{k}<K_m$ be such that $\ellindex_{\hat{k}}\leq\ellindex<\ellindex_{\hat{k}+1}$. For $1\leq j<m$ and $1\leq k< K_j$, and for $j=m$ and $1\leq k< \hat{k}$, define
	\be\label{eq:lambdaellttjktjk1}\lambda_\ellindex(t)\doteq \rho_j^\ellindex(k)+\frac{\rho_j^\ellindex(k+1)-\rho_j^\ellindex(k)}{\rho_j(k+1)-\rho_j(k)}(t-\rho_j(k)),\qquad t\in[\rho_j(k),\rho_j(k+1)),\ee
and, if $\theta_2<\infty$, define	
	\be\label{eq:lambdaelltheta_2}\lambda_\ellindex(t)\doteq \rho_m^\ellindex(\hat{k}+1)+\frac{\theta_2-\rho_m^\ellindex(\hat{k}+1)}{\theta_2-\rho_m(\hat{k}+1)}(t-\rho_m(\hat{k}+1)),\qquad t\in[\rho_m(\hat{k}+1),\theta_2),\ee
whereas if $\theta_2=\infty$, define
	\be\label{eq:lambdaellinfty}\lambda_\ellindex(t)\doteq t+(\rho_m^\ellindex(\hat k+1)-\rho_m(\hat{k}+1))e^{-(t-\rho_m(\hat{k}+1))},\qquad t\in[\rho_m(\hat{k}+1),\infty).\ee
It is readily verified from \eqref{eq:lambdaellttjktjk1}, \eqref{eq:lambdaelltheta_2} and \eqref{eq:lambdaellinfty} that for each $\ellindex\geq\ellindex_1$, $\lambda_\ellindex$ is a nondecreasing, continuous, piecewise linear function mapping $[0,\theta_2)$ onto $[0,\theta_2)$ such that for each $1\leq j\leq m$,
\begin{align}\label{eq:lambdatjk}
	\lambda_\ellindex(\rho_j(k))&=\rho_j^\ellindex(k),&&1<k<K_j,\\ \label{eq:lambdatmk}
	\lambda_\ellindex(\rho_m(k))&=\rho_m^\ellindex(k),&& 1<k\leq \hat{k}.
\end{align}
We now turn to the proofs of \eqref{eq:timechange1thetan} and \eqref{eq:timechange2thetan}.
	
We begin with the proof of \eqref{eq:timechange1thetan}. Let $T<\infty$. We first treat the case that $T<\theta_2$. By \eqref{eq:tmKm} and \eqref{eq:limktmk}, we can choose $1\leq\hat k<K_m$ such that $T<\rho_m(\hat k)$. Then by \eqref{eq:lambdaellttjktjk1}, \eqref{eq:lambdaelltheta_2} and \eqref{eq:lambdaellinfty}, for all $\ellindex\geq\ellindex_{\hat{k}}$,
	\be\label{eq:supsisi1}\sup_{t\in[0,T)}|\lambda_\ellindex(t)-t|\leq\max_{1<k\leq\hat k}|\rho_j^\ellindex(k)-\rho_j(k)|.\ee
Letting $\ellindex\to\infty$, it follows from \eqref{eq:tkjtotk} that \eqref{eq:timechange1thetan} holds. Next, consider the case that $\theta_2\leq T<\infty$. Let $\ve\in(0,\theta_2)$ be arbitrary. Since $\theta_2-\ve<\theta_2$ and \eqref{eq:timechange1thetan} holds whenever $T<\theta_2$, we have
	\be\label{eq:limjsuptheta_2ve}\lim_{\ellindex\to\infty}\sup_{t\in[0,\theta_2-\ve)}|\lambda_\ellindex(t)-t|=0.\ee
Therefore, it suffices to show that
	\be\label{eq:limvelimjuptheta_2ve}\lim_{\ve\downarrow0}\lim_{\ellindex\to\infty}\sup_{t\in[\theta_2-\ve,\theta_2)}|\lambda_\ellindex(t)-t|=0.\ee
By the triangle inequality and the fact that $\lambda_{\ellindex}:[0,\theta_2)\mapsto[0,\theta_2)$ is nondecreasing, we have, for all $\ve\in(0,\theta_2)$ and $t\in[\theta_2-\ve,\theta_2)$, 
\begin{align}\label{eq:lambdajitt}
	|\lambda_\ellindex(t)-t|&\leq|\lambda_\ellindex(t)-\theta_2|+|\theta_2-t|\\ \notag
	&\leq|\lambda_\ellindex(\theta_2-\ve)-\theta_2|+\ve\\ \notag
	&\leq|\lambda_\ellindex(\theta_2-\ve)-(\theta_2-\ve)|+2\ve.
\end{align}
By \eqref{eq:limjsuptheta_2ve} and the continuity of $\lambda_\ellindex$, $\lambda_\ellindex(\theta_2-\ve)\to\theta_2-\ve$ as $\ellindex\to\infty$, which along with \eqref{eq:lambdajitt} yields \eqref{eq:limvelimjuptheta_2ve}. This establishes \eqref{eq:timechange1thetan}.
	
We now turn to the proof of \eqref{eq:timechange2thetan}, which will complete the proof that Statement \ref{state:inductioncts} holds with $n=2$. Let $T<\theta_2$. By \eqref{eq:tmKm} and \eqref{eq:limktmk}, we can choose $1<\hat k<K_m$ such that $T<\rho_m(\hat k)$. We show that for each $\ellindex\geq\ellindex_{\hat{k}}$,
	\be\label{eq:dmzjpsilambdajdmzpsi}\dm_{\z_\ellindex}(\psi)(\lambda_\ellindex(t))=\dm_\z(\psi)(t)\qquad\text{for all }t\in[0,\rho_m(\hat k)),\ee
which will complete the proof of \eqref{eq:timechange2thetan}. Fix $\ellindex\geq\ellindex_{\hat{k}}$. By Lemma \ref{lem:phisimpletheta_2}, for each $1\leq j\leq m$ and $1\leq k<K_j$, $\dm_\z(\psi)$ satisfies
	\be\label{eq:dmzpsit}\dm_\z(\psi)(t)=\proj_{\z(\rho_j(k))}\cdots\proj_{\z(\rho_j(1))}(\dm_\z(\psi)(s_j-)+y_j),\ee
for all $t\in[\rho_j(k),\rho_j(k+1))$, where $\dm_\z(\psi)(0-)\doteq 0$. Similarly, by Lemma \ref{lem:phisimpletheta_2}, for each $1\leq j\leq m$ and $1\leq k<K_j^\ellindex=K_j$ (by (g)), $\dm_{\z_\ellindex}(\psi)$ satisfies
\begin{align}\label{eq:dmzjpsit}
	\dm_{\z_\ellindex}(\psi)(t)&=\proj_{\z_\ellindex(\rho_j^\ellindex(k))}\cdots\proj_{\z_\ellindex(\rho_j^\ellindex(1))}(\dm_{\z_\ellindex}(\psi)(s_j-)+y_j)\\ \notag
	&=\proj_{\z(\rho_j(k))}\cdots\proj_{\z(\rho_j(1))}(\dm_{\z_\ellindex}(\psi)(s_j-)+y_j),
\end{align}
for $t\in[\rho_j^\ellindex(k),\rho_j^\ellindex(k+1))$, where $\dm_{\z_\ellindex}(\psi)(0-)\doteq 0$ and we have used (h) and (i) in the second equality and properties of the derivative projection matrix in Lemma \ref{lem:projx}. By \eqref{eq:lambdatjk}, \eqref{eq:lambdatmk}, for $1\leq j<m$ and $1\leq k<K_j$, or for $j=m$ and $1\leq k<\hat k$,
	\begin{align}\label{eq:dmzellpsilambdaellt}
	\dm_{\z_\ellindex}(\psi)(\lambda_\ellindex(t))&=\proj_{\z(\rho_j(k))}\cdots\proj_{\z(\rho_j(1))}(\dm_{\z_\ellindex}(\psi)(s_j-)+y_j),\quad t\in[\rho_j(k),\rho_j(k+1)).
	\end{align}
Then since $\dm_{\z_\ellindex}(\psi)(0)=\psi(0)=\dm_\z(\psi)(0)$,  by \eqref{eq:dmzpsit}, \eqref{eq:dmzellpsilambdaellt} and a simple recursion argument, we see that \eqref{eq:dmzjpsilambdajdmzpsi} holds.
\end{proof}

Before proving the induction step, we need the following helpful lemma.

\begin{lem}[{\cite[Lemma 8.3]{Lipshutz2016}}]
\label{lem:projxcontract}
There is a norm on $\R^J$, denoted $\norm{\cdot}_B$, such that under $\norm{\cdot}_B$ the derivative projection matrix $\proj_x$ is a contraction for all $x\in G$; that is, $\norm{\proj_xy}_B\leq\norm{y}_B$ for all $x\in G$ and $y\in\R^J$. Furthermore, since norms on $\R^J$ are equivalent, there exists $C_B<\infty$ such that $\norm{y}_B\leq C_B|y|$ for all $y\in\R^J$.
\end{lem}

\begin{lem}\label{lem:thetanpsiconst}
Let $2\leq n\leq N$ and suppose that Statement \ref{state:inductioncts} holds. Then Statement \ref{state:inductioncts} holds with $n+1$ in place of $n$.
\end{lem} 

\begin{proof}
Let $\x$, $(\z,\y)$, $\{\x_\ellindex\}_{\ellindex\in\N}$, $\{(\z_\ellindex,\y_\ellindex)\}_{\ellindex\in\N}$ and $\psi$ be as in Statement \ref{state:inductioncts}. If $\theta_{n+1}=\theta_n$ then \eqref{eq:timechange1thetan} and \eqref{eq:timechange2thetan} hold by assumption. For the remainder of the proof we assume that $\theta_{n+1}>\theta_n$. We split the proof into four parts (whose titles are indicated using italic font). The first two parts are devoted to defining the time-changes and the latter two parts are devoted to proving convergence results for the time-changes and the time-changed paths.

\emph{Partition of the interval $[0,\theta_{n+1})$ into subintervals $\{[t_k,t_{k+1})\}_{1\leq k<K+1}$}: Set $t_0\doteq0$, $t_1\doteq\theta_n$ and for $k\geq1$ such that $t_k<\theta_{n+1}$, define $\rho_k$ to be the first time after $t_k$ that $\z$ hits a face that does not lie in the subset of faces that $\z(t_k)$ lies in; that is,
	\be\label{eq:rhokn}\rho_k\doteq\inf\{t>t_k:\allN(\z(t))\not\subseteq\allN(\z(t_k))\}\ee
and let $t_k$ be the first time after $\rho_k$ that $\z$ lies at the intersection of $n$ or more faces; that is,
	\be\label{eq:tkn}t_{k+1}\doteq\inf\{t\geq\rho_k:|\allN(\z(t))|\geq n\}.\ee
If $t_k=\theta_{n+1}$ for some $k\in\N$, then set $K\doteq k$ and end the sequence. Alternatively, if $t_k<\theta_{n+1}$ for all $k\in\N$, define $K\doteq\infty$. In this case, we claim that $t_k\to\theta_{n+1}$ as $k\to\infty$. The claim clearly holds if $t_k\to\infty$ as $k\to\infty$; and on the other hand, if $t_k$ is uniformly bounded for all $k\in\N$, then there must be an accumulation point and it follows from \eqref{eq:rhokn}--\eqref{eq:tkn} that the accumulation point must be $\theta_{n+1}$. In either case, the following holds:
	\be\label{eq:a}\text{Given $T<\theta_{n+1}$, there exists $1\leq k<K+1$ such that $T<t_k$.}\ee
	
\emph{Definition of the time-changes $\lambda_\ellindex:[0,\theta_{n+1})\mapsto[0,\theta_{n+1})$}: We use the induction hypothesis and time-shift properties of the ESP (Lemma \ref{lem:sptimeshift}) and derivative problem to define the time-changes $\{\lambda_\ellindex\}_{\ellindex\in\N}$ on intervals of the form $[t_k,t_{k+1})$ for $0\leq k<K$. We begin with the interval $[0,t_1)$. Since $t_1=\theta_n$ by definition and Statement \ref{state:inductioncts} holds by the induction hypothesis, there is a sequence $\{\lambda_\ellindex^1\}_{\ellindex\in\N}$ of time-changes mapping $[0,t_1)$ to $[0,t_1)$ such that for all $T<\infty$,
	\be\label{eq:b}\lim_{\ellindex\to\infty}\sup_{t\in[0,t_1\wedge T)}|\lambda_\ellindex^1(t)-t|=0,\ee
and for all $T<t_1$, 
	\be\label{eq:c}\lim_{\ellindex\to\infty}\sup_{t\in[0,T]}\left|\dm_{\z_\ellindex}(\psi)(\lambda_\ellindex^1(t))-\dm_\z(\psi)(t)\right|=0.\ee
For each $\ellindex\in\N$, define $\lambda_\ellindex$ on the interval $[0,t_1)$ by 
	\be\lambda_\ellindex(t)\doteq\lambda_\ellindex^1(t)\qquad\text{for }t\in[0,t_1).\ee
Now suppose $1\leq k<K$. Since $h$ spends zero Lebesgue measure on the boundary $\partial G$ due to condition 2 of the boundary jitter property (Definition \ref{def:jitter}), it follows from the definition of $\rho_k$ in \eqref{eq:rhokn} and the upper semicontinuity of $\allN(\cdot)$ shown in Lemma \ref{lem:allNusc} that there exists
	\be\label{eq:tkxikrhok}t_k<\xi_k<\rho_k\ee 
such that
\begin{align}\label{eq:zGcircxikrhok}
	\z(\xi_k)&\in G^\circ\qquad\text{and}\qquad|\allN(\z(t))|<n\text{ for all }t\in[\xi_k,\rho_k).
\end{align} 
Since $\z(\rho_k)\in\partial G$ by \eqref{eq:rhokn}, the definition of $\simple^\z(\R^J)$ in \eqref{eq:simplez} implies that $\psi$ is constant in a neighborhood of $\rho_k$. In particular, by choosing $\xi_k\in(t_k,\rho_k)$ possibly larger, we can assume that 
	\be\label{eq:psiconstantxikrhok}\psi\text{ is constant on }[\xi_k,\rho_k].\ee 
Define 
\begin{align}\label{eq:zkykxk}
	\z^k(\cdot)&\doteq\z(\xi_k+\cdot),\qquad \y^k(\cdot)\doteq\y(\xi_k+\cdot)-\y(\xi_k),\qquad\x^k(\cdot)\doteq\z(\xi_k)+\x(\xi_k+\cdot)-\x(\xi_k).
\end{align}
By the time-shift property of the ESP, $(\z^k,\y^k)$ is a solution to the ESP for $\x^k$. In addition, since $(\z,\y)$ satisfies the boundary jitter property and $\z(\xi_k)\in G^\circ$ by \eqref{eq:zGcircxikrhok}, it follows from \eqref{eq:zs} that 
	\be\label{eq:d}(\z^k,\y^k)\text{ satisfies the boundary jitter property}.\ee
Finally, by \eqref{eq:zs}, \eqref{eq:zGcircxikrhok} and \eqref{eq:tkn},
\begin{align} \label{eq:thetanxik}
	\theta_n^k&\doteq\inf\{t\geq0:|\allN(\z^{\xi_k}(t))|\geq n\}\\ \notag
	&=\inf\{t\geq\xi_k:|\allN(\z(t))|\geq n\}-\xi_k\\ \notag
	&=t_{k+1}-\xi_k. 
\end{align}
For each $\ellindex\in\N$, define $\z_\ellindex^k$, $\y_\ellindex^k$ and $\x_\ellindex^k$
\begin{align}\label{eq:zlkylkxlk}
	\z_\ellindex^k(\cdot)&\doteq\z_\ellindex(\xi_k+\cdot),\qquad \y_\ellindex^k(\cdot)\doteq\y_\ellindex(\xi_k+\cdot)-\y_\ellindex(\xi_k),\qquad\x_\ellindex^k(\cdot)\doteq\z_\ellindex(\xi_k)+\x_\ellindex(\xi_k+\cdot)-\x_\ellindex(\xi_k).
\end{align}
Again, by the time-shift property of the ESP, $(\z_\ellindex^k,\y_\ellindex^k)$ is a solution to the ESP for $\x_\ellindex^k$. Since $\x_\ellindex$ converges to $\x$ in uniformly on compact intervals as $\ellindex\to\infty$, it follows that
	\be\label{eq:e}\x_\ellindex^k\text{ converges to $\x^k$ uniformly on compact intervals as }\ellindex\to\infty.\ee
Therefore, by the Lipschitz continuity of the ESM (Proposition \ref{prop:sp}(ii)),
	\be\label{eq:f}(\z_\ellindex^k,\y_\ellindex^k)\text{ converges to $(\z^k,\y^k)$ uniformly on compact intervals as }\ellindex\to\infty.\ee
Define 
	\be\label{eq:psixik}\psi^k(\cdot)\doteq\dm_\z(\psi)(\xi_k)+\psi(\xi_k+\cdot)-\psi(\xi_k).\ee
By the time-shift property of the derivative problem (Lemma \ref{lem:dptimeshift}),
	\be\label{eq:psikdmshift}\dm_{\z^k}(\psi^k)(\cdot)=\dm_\z(\psi)(\xi_k+\cdot).\ee
The definitions of $\psi^k$ and $\z^k$, the fact that $\z(\xi_k)\in G^\circ$ so $\hyper_{\z(\xi_k)}=\R^J$, the fact that $\psi\in\simple^\z(\R^J)$ and \eqref{eq:simplez} imply that
	\be\label{eq:g}\psi^k\in\simple^{\z^k}(\R^J).\ee
By \eqref{eq:d}, \eqref{eq:e}, \eqref{eq:g} and the fact that Statement \ref{state:inductioncts} holds by assumption (with $\z^k$, $\y^k$, $\x^k$, $\z_\ellindex^k$, $\y_\ellindex^k$, $\x_\ellindex^k$, $\psi^k$ and $\theta_n^k$ in place of $\z$, $\y$, $\x$, $\z_\ellindex$, $\y_\ellindex$, $\x_\ellindex$, $\psi$ and $\theta_n$, respectively), there is a sequence $\{\lambda_\ellindex^k\}_{\ellindex\in\N}$ of time-changes mapping $[0,\theta_n^k)$ to $[0,\theta_n^k)$ such that for all $T<\infty$,
	\be\label{eq:limelllambdaell1rhok}\lim_{\ellindex\to\infty}\sup_{t\in[0,\theta_n^k\wedge T)}|\lambda_\ellindex^k(t)-t|=0,\ee
and for all $T<\theta_n^k$,
	\be\label{eq:limelllambdaell2rhok}\lim_{\ellindex\to\infty}\sup_{t\in[0,T]}\left|\dm_{\z_\ellindex^k}(\psi^k(\lambda_\ellindex^k(t))-\dm_{\z^k}(\psi^k)(t)\right|=0,\ee
For each $\ellindex\in\N$, define $\lambda_\ellindex$ on the interval $[t_k,t_{k+1})=[t_k,\xi_k+\theta_n^k)$ (due to \eqref{eq:thetanxik}) by
	\be\label{eq:lambdajtktk1}\lambda_\ellindex(t)\doteq
	\begin{cases}
		t&\text{for }t\in[t_k,\xi_k)\\
		\lambda_\ellindex^k(t-\xi_k)&\text{for }t\in[\xi_k,t_{k+1}).
	\end{cases}\ee

\emph{Convergence of the time-changed paths on intervals of the form $[0,t_k)$}: We use the principle of mathematical induction to show that for each $1\leq k<K+1$ and for all $T<\infty$,  
	\be\label{eq:timechange1tk}\lim_{\ellindex\to\infty}\sup_{t\in[0,t_k\wedge T)}|\lambda_\ellindex(t)-t|=0,\ee
and for all $T<t_k$,
	\be\label{eq:timechange2tk}\lim_{\ellindex\to\infty}\sup_{t\in[0,T]}|\dm_{\z_\ellindex}(\psi)(\lambda_\ellindex(t))-\dm_\z(\psi)(t)|=0.\ee
The base case $k=1$ follows immediately from \eqref{eq:b} and \eqref{eq:c}. Now suppose $1\leq k<K+1$ and \eqref{eq:timechange1tk} holds for all $T<\infty$ and \eqref{eq:timechange2tk} holds for all $T<t_k$. We first prove \eqref{eq:timechange1tk} holds for all $T<\infty$ with $k+1$ in place of $k$. Let $T<\infty$. We have
\begin{align}\label{eq:limjsuprhokTlambdaj}
	\sup_{t\in[0,t_{k+1}\wedge T)}|\lambda_\ellindex(t)-t|&=\sup_{t\in[0,t_k\wedge T)}|\lambda_\ellindex(t)-t|\vee\sup_{t\in[\xi_k\wedge T,t_{k+1}\wedge T)}|\lambda_\ellindex(t)-t|.
\end{align}
By \eqref{eq:lambdajtktk1} and \eqref{eq:thetanxik}, we have
\begin{align}\label{eq:limjsuprhokTlambdaj2}
	\sup_{t\in[\xi_k\wedge T,t_{k+1}\wedge T)}|\lambda_\ellindex(t)-t|&\leq\sup_{t\in[0,(\theta_n^k\wedge(T-\xi_k))\vee0)}|\lambda_\ellindex^k(t)-t|.
\end{align}
By \eqref{eq:limjsuprhokTlambdaj}, our assumption that \eqref{eq:timechange1tk} holds, \eqref{eq:limjsuprhokTlambdaj2} and \eqref{eq:limelllambdaell1rhok}, we see that \eqref{eq:timechange1tk} holds with $k+1$ in place of $k$. 

Next, we prove \eqref{eq:timechange2tk} holds for all $T<t_{k+1}$. If $T<t_k$, then \eqref{eq:timechange2tk} follows by assumption. Let $t_k\leq T<t_{k+1}$. We have
\begin{align}\label{eq:supdmhlpsidmhpsi}
	\sup_{t\in[0,T]}\left|\dm_{\z_\ellindex}(\psi)(\lambda_\ellindex(t))-\dm_\z(\psi)(t)\right|&\leq\sup_{t\in[0,\xi_k\wedge T]}\left|\dm_{\z_\ellindex}(\psi)(\lambda_\ellindex(t))-\dm_\z(\psi)(t)\right|\\ \notag
	&\qquad+\sup_{t\in[\xi_k\wedge T,T]}|\dm_{\z_\ellindex}(\psi)(\lambda_\ellindex(t))-\dm_\z(\psi)(t)|.
\end{align}
We first show that
	\be\label{eq:limjnormxik}\lim_{\ellindex\to\infty}\sup_{t\in[0,\xi_k\wedge T]}\left|\dm_{\z_\ellindex}(\psi)(\lambda_\ellindex(t))-\dm_\z(\psi)(t)\right|=0.\ee
Since $|\allN(\z(t_k))|\geq2$ by \eqref{eq:tkn}, Proposition \ref{prop:dp} implies that $\dm_\z(\psi)(\cdot)$ is continuous at $t_k$. Let $\ve>0$. By \eqref{eq:tkxikrhok} and the continuity of $\dm_\z(\psi)(\cdot)$ at $t_k$, we can choose $\delta>0$ sufficiently small such that 
	\be\label{eq:tkdeltaxikrhokdelta}0<t_k-\delta<t_k+\delta<\xi_k<\rho_k-\delta\ee
and
	\be\label{eq:dmzpsidmzpsitkve}|\dm_\z(\psi)(t)-\dm_\z(\psi)(t_k)|<\ve\qquad\text{for all }t\in[t_k-\delta,t_k+\delta].\ee
Due to the upper semicontinuity of $\allN(\cdot)$ shown in Lemma \ref{lem:allNusc} and the definition of $\rho_k$ in \eqref{eq:rhokn}, we can choose $\delta>0$ possibly smaller to ensure that
	\be\label{eq:allNztallNztk}\allN(\z(t))\subseteq\allN(\z(t_k))\qquad\text{for all }t\in[t_k-\delta,\rho_k).\ee
It follows from \eqref{eq:timechange2tk} (which holds for $T<t_k$ by assumption) that
	\be\label{eq:limlsup0tkdelta}\lim_{\ellindex\to\infty}\sup_{t\in[0,t_k-\delta]}\left|\dm_{\z_\ellindex}(\psi)(\lambda_\ellindex(t))-\dm_\z(\psi)(t)\right|=0.\ee
Then by \eqref{eq:limlsup0tkdelta} and \eqref{eq:dmzpsidmzpsitkve}, we can choose $\ellindex_0\in\N$ sufficiently large such that for all $\ellindex\geq\ellindex_0$, 
\begin{align}\label{eq:dmjpsilambdajtkdelta}
	|\dm_{\z_\ellindex}(\psi)(\lambda_\ellindex(t_k-\delta))-\dm_\z(\psi)(t_k)|&\leq|\dm_{\z_\ellindex}(\psi)(\lambda_\ellindex(t_k-\delta))-\dm_\z(\psi)(t_k-\delta)|\\ \notag
	&\qquad+|\dm_\z(\psi)(t_k-\delta)-\dm_\z(\psi)(t_k)|\\ \notag
	&<2\ve.
\end{align}
By \eqref{eq:f}, the fact that \eqref{eq:timechange1tk} holds with $k+1$ in place of $k$ and the relation $\rho_k<t_{k+1}$, we see that 
	$$\lim_{\ellindex\to\infty}\sup_{t\in[0,\rho_k-\delta]}|\z_\ellindex(\lambda_\ellindex(t))-\z(t)|=0.$$
Then by the last display, \eqref{eq:allNztallNztk} and the upper semicontinuity of $\allN(\cdot)$, we can choose $\ellindex_0\in\N$ possibly larger such that for all $\ellindex\geq\ellindex_0$,
	\be\label{eq:allNzelltallNztk}\allN(\z_\ellindex(\lambda_\ellindex(t)))\subseteq\allN(\z(t_k))\qquad\text{for all }t\in[t_k-\delta,\rho_k-\delta].\ee
Let $t\in[t_k-\delta,\xi_k]$ be arbitrary. Since $\psi$ is constant on $[t_k-\delta,\xi_k]$, according to Corollary \ref{cor:DPpsiconstantST}, there exists $m\in\N$ and a sequence $\lambda_{\ellindex}(t_k-\delta)<s_1<\cdots<s_m=\lambda_{\ellindex}(t)$ such that for all $\ellindex\geq\ellindex_0$,
	\be\label{eq:dmhelllambdalt}\dm_{\z_\ellindex}(\psi)(\lambda_\ellindex(t))=\proj_{\z_\ellindex(s_m)}\cdots\proj_{\z_\ellindex(s_1)}[\dm_{\z_\ellindex}(\psi)(\lambda_\ellindex(t_k-\delta))].\ee
In addition, since $\dm_\z(\psi)(t_k)\in\hyper_{\z(t_k)}$ (by condition 2 of Definition \ref{def:dp}) and \eqref{eq:allNzelltallNztk} holds, it follows from the uniqueness of the derivative projection matrices $\proj_x$ stated in Lemma \ref{lem:projx} that 
	$$\proj_{\z_\ellindex(s_j)}\dm_\z(\psi)(t_k)=\dm_{\z}(\psi)(t_k),\qquad 1\leq j\leq m.$$ 
By \eqref{eq:dmhelllambdalt}, the last display and the fact that $\proj_{{\z_\ellindex}(s_1)},\dots,\proj_{{\z_\ellindex}(s_m)}$ are linear operators, we have, for all $\ellindex\geq\ellindex_0$,
\begin{align*}
	\dm_{\z_\ellindex}(\psi)(\lambda_\ellindex(t))-\dm_\z(\psi)(t_k)&=\proj_{\z_\ellindex(s_m)}\cdots\proj_{\z_\ellindex(s_1)}[\dm_{\z_\ellindex}(\psi)(\lambda_\ellindex(t_k-\delta))-\dm_\z(\psi)(t_k)].
\end{align*}
Let $\norm{\cdot}_B$ be the norm on $\R^J$ introduced in Lemma \ref{lem:projxcontract}. Then by the last display, Lemma \ref{lem:projxcontract}, \eqref{eq:dmzpsidmzpsitkve} and \eqref{eq:dmjpsilambdajtkdelta}, for all $\ellindex\geq\ellindex_0$,
\begin{align*}
	\norm{\dm_{\z_\ellindex}(\psi)(\lambda_\ellindex(t))-\dm_\z(\psi)(t)}_B&\leq\norm{\dm_{\z_\ellindex}(\psi)(\lambda_\ellindex(t))-\dm_\z(\psi)(t_k)}_B+\norm{\dm_\z(\psi)(t)-\dm_\z(\psi)(t_k)}_B\\
	&\leq\norm{\proj_{\z_\ellindex(s_m)}\cdots\proj_{\z_\ellindex(s_1)}[\dm_{\z_\ellindex}(\psi)(\lambda_\ellindex(t_k-\delta))-\dm_\z(\psi)(t_k)]}_B\\
	&\qquad+C_B|\dm_\z(\psi)(t)-\dm_\z(\psi)(t_k)|\\
	&\leq\norm{\dm_{\z_\ellindex}(\psi)(\lambda_\ellindex(t_k-\delta))-\dm_\z(\psi)(t_k)}_B+C_B\ve\\\
	&<3C_B\ve.
\end{align*}
Since $t\in[t_k-\delta,\xi_k]$ was arbitrary, for all $\ellindex\geq\ellindex_0$,
	$$\sup_{t\in[t_k-\delta,\xi_k]}\norm{\dm_{\z_\ellindex}(\psi)(\lambda_\ellindex(t))-\dm_\z(\psi)(t)}_B\leq 3C_B\ve.$$
Since norms on $\R^J$ are equivalent and $\ve>0$ was arbitrary, we have
	\be\label{eq:limlsuptkdeltaxik}\lim_{\ellindex\to\infty}\sup_{t\in[t_k-\delta,\xi_k]}|\dm_{\z_\ellindex}(\psi)(\lambda_\ellindex(t))-\dm_\z(\psi)(t)|=0.\ee
Along with \eqref{eq:limlsup0tkdelta}, \eqref{eq:limlsuptkdeltaxik} implies that \eqref{eq:limjnormxik} holds.
	
Next, we show that
	\be\label{eq:limlxiTTdmhldmh}\lim_{\ellindex\to\infty}\sup_{t\in[\xi_k\wedge T,T]}\left|\dm_{\z_\ellindex}(\psi)(\lambda_\ellindex(t))-\dm_\z(\psi)(t)\right|=0.\ee
Define $\z^k$, $\y^k$, $\x^k$ and in \eqref{eq:zkykxk}, $\z_\ellindex^k$, $\y_\ellindex^k$ and $\x_\ellindex^k$ as in \eqref{eq:zlkylkxlk} and $\psi^k$ as in \eqref{eq:psixik}. For each $\ellindex\in\N$, define $\psi_\ellindex^k\in\simple(\R^J)$ by
	\be\label{eq:psiellxikt}\psi_\ellindex^k(\cdot)\doteq\dm_{\z_\ellindex}(\psi)(\xi_k)+\psi(\xi_k+\cdot)-\psi(\xi_k).\ee
Then by the time-shift property of the derivative problem (Lemma \ref{lem:dptimeshift}), 
	\be\label{eq:psiellkdmshift}\dm_{\z_\ellindex^k}(\psi_\ellindex^k)(\cdot)=\dm_{\z_\ellindex}(\psi)(\xi_k+\cdot).\ee
Thus, by \eqref{eq:psiellkdmshift}, \eqref{eq:psikdmshift} and the triangle inequality,
\begin{align}\label{eq:dmzlpsilambdaldmzpsi}
	\sup_{t\in[\xi_k\wedge T,T]}\left|\dm_{\z_\ellindex}(\psi)(\lambda_\ellindex(t))-\dm_\z(\psi)(t)\right|&=\sup_{t\in[0,(T-\xi_k)\vee0]}\left|\dm_{\z_\ellindex^k}(\psi_\ellindex^k)(\lambda_\ellindex^k(t))-\dm_{\z^k}(\psi^k)(t)\right|\\ \notag
	&\leq\sup_{t\in[0,(T-\xi_k)\vee0]}\left|\dm_{\z_\ellindex^k}(\psi_\ellindex^k)(\lambda_\ellindex^k(t))-\dm_{\z_\ellindex^k}(\psi^k)(\lambda_\ellindex^k(t))\right|\\ \notag
	&\qquad+\sup_{t\in[0,(T-\xi_k)\vee0]}\left|\dm_{\z_\ellindex^k}(\psi^k)(\lambda_\ellindex^k(t))-\dm_{\z^k}(\psi^k)(t)\right|.
\end{align}
By the Lipschitz continuity of the derivative map (Proposition \ref{prop:dmlip}), the definition of $\psi_\ellindex^k$ in \eqref{eq:psiellxikt}, the definition of $\psi^k$ in \eqref{eq:psixik}, we have
\begin{align*}
	\sup_{t\in[0,(T-\xi_k)\vee0]}\left|\dm_{\z_\ellindex^k}(\psi_\ellindex^k)(\lambda_\ellindex^k(t))-\dm_{\z_\ellindex^k}(\psi^k)(\lambda_\ellindex^k(t))\right|&\leq\lip_{\dm}(\alpha)\sup_{t\in[0,\lambda_\ellindex^k((T-\xi_k)\vee0)]}\left|\psi_\ellindex^k(t)-\psi^k(t)\right|\\
	&\leq\lip_{\dm}(\alpha)|\dm_{\z_\ellindex}(\psi)(\xi_k)-\dm_\z(\psi)(\xi_k)|.
\end{align*}
Then by \eqref{eq:limjnormxik} and the fact that $\lambda_{\ellindex}(\xi_k)=\xi_k$ due to \eqref{eq:lambdajtktk1}, letting $\ellindex\to\infty$ yields
\begin{align}\label{eq:psiellxikpsixik}
	\lim_{\ellindex\to\infty}\sup_{t\in[0,(T-\xi_k)\vee0]}\left|\dm_{\z_\ellindex^k}(\psi_\ellindex^k)(\lambda_\ellindex^k(t))-\dm_{\z_\ellindex^k}(\psi^k)(\lambda_\ellindex^k(t))\right|=0.
\end{align}
Since $(T-\xi_k)\vee0<\theta_n^k$, it follows from \eqref{eq:limelllambdaell2rhok} that
\begin{align}\label{eq:dmzlkpsiklambdalkdmzkpsik}
	\lim_{\ellindex\to\infty}\sup_{t\in[0,(T-\xi_k)\vee0]}\left|\dm_{\z_\ellindex^k}(\psi^k)(\lambda_\ellindex^k(t))-\dm_{\z^k}(\psi^k)(t)\right|&=0.
\end{align}
Together, \eqref{eq:dmzlpsilambdaldmzpsi}, \eqref{eq:psiellxikpsixik} and \eqref{eq:dmzlkpsiklambdalkdmzkpsik} imply \eqref{eq:limlxiTTdmhldmh} holds. Since \eqref{eq:limjnormxik} and \eqref{eq:limlxiTTdmhldmh} both hold, we see that \eqref{eq:timechange2tk} holds with $k+1$ in place of $k$. With the induction step established, the principle of mathematical induction implies that for each $1\leq k<K+1$, \eqref{eq:timechange1tk} holds for all $T<\infty$ and \eqref{eq:timechange2tk} holds for all $T<t_k$.
	
\emph{Convergence of the time-changed paths on $[0,\theta_{n+1})$}: We first prove \eqref{eq:timechange1thetan} holds for all $T<\infty$ with $\theta_{n+1}$ in place of $n$. Let $T<\infty$. Suppose $T<\theta_{n+1}$. Then there exists $1\leq k<K+1$ such that $T<t_k$, so by \eqref{eq:timechange1tk},
	\be\label{eq:limjsup0Tsup0tk1}\lim_{\ellindex\to\infty}\sup_{t\in[0,T)}|\lambda_\ellindex(t)-t|\leq\lim_{\ellindex\to\infty}\sup_{t\in[0,t_k)}|\lambda_\ellindex(t)-t|=0.\ee
Now suppose $T\geq\theta_{n+1}$. Let $\ve\in(0,\theta_n)$ be arbitrary. Since \eqref{eq:limjsup0Tsup0tk1} holds for all $T<\theta_{n+1}$, we have
	\be\label{eq:limjsupthetandelta}\lim_{\ellindex\to\infty}\sup_{t\in[0,\theta_{n+1}-\ve)}|\lambda_\ellindex(t)-t|=0.\ee
Thus, we are left to show that
	\be\label{eq:limellsupthetanvethetan}\lim_{\ve\downarrow0}\lim_{\ellindex\to\infty}\sup_{t\in[\theta_{n+1}-\ve,\theta_{n+1})}|\lambda_\ellindex(t)-t|=0.\ee
By the triangle inequality and the fact that $\lambda_\ellindex$ is nondecreasing, for all $\ve\in(0,\theta_{n+1})$ and $t\in[\theta_{n+1}-\ve,\theta_{n+1})$,
	\be\label{eq:lambdaellminust}|\lambda_\ellindex(t)-t|\leq|\lambda_\ellindex(t)-\theta_{n+1}|+|\theta_{n+1}-t|\leq|\lambda_\ellindex(\theta_{n+1}-\ve)-\theta_{n+1}|+\ve.\ee
By \eqref{eq:limjsupthetandelta} and the continuity of $\lambda_\ellindex$, $\lambda_\ellindex(\theta_{n+1}-\ve)\to\theta_{n+1}-\ve$ as $\ellindex\to\infty$. Thus, \eqref{eq:limellsupthetanvethetan} holds, which implies \eqref{eq:timechange1thetan} holds with $\theta_{n+1}$ in place of $\theta_n$. We are left with the final step of proving \eqref{eq:timechange2thetan} holds for all $T<\theta_{n+1}$. Let $T<\theta_{n+1}$. By \eqref{eq:a} there exists $1\leq k<K+1$ such that $T<t_k$. It then follows from \eqref{eq:timechange2tk} that \eqref{eq:timechange2thetan} holds.
\end{proof}

\begin{proof}[Proof of Lemma \ref{lem:psisimple}]
By Lemma \ref{lem:theta_2psiconst}, Lemma \ref{lem:thetanpsiconst} and the principle of mathematical induction, Statement \ref{state:inductioncts} holds for $n=2,\dots,N+1$. Since $\theta_{N+1}=\infty$ by definition, the lemma follows. 
\end{proof}

\begin{flushleft}
{\bf Acknowledgments}. We thank Mike Giles for useful feedback on a preliminary draft of this paper.
\end{flushleft}

\bibliographystyle{plain}
\bibliography{montecarlo}

\end{document}